\RequirePackage{fix-cm}
\documentclass[11pt, reqno]{amsart}
\usepackage{amsmath}
\usepackage{setspace}
\usepackage{fancyhdr}
\usepackage{mathrsfs}
\usepackage{xifthen}
\usepackage{verbatim}
\usepackage{hyperref}
\usepackage[all]{xcolor}

\usepackage{float}

\usepackage{comment}

\usepackage{geometry}
\usepackage{fancyhdr}
\usepackage{textcomp}
\usepackage{amsthm}
\usepackage{amstext}
\usepackage{amsfonts}
\usepackage{amssymb}
\usepackage{fixltx2e}
\usepackage{graphicx}

\usepackage{tikz}
\usepackage{tikz-cd}
\usepackage{caption}
\usetikzlibrary{calc,backgrounds}
\usetikzlibrary{matrix,arrows,decorations.pathmorphing}

\usepackage{bigstrut}

\makeatletter

\newcommand{\schff}{X}
\newcommand{\adicff}{\mathcal{X}}
\newcommand{\Spa}{\mathrm{Spa}}
\newcommand{\Spd}{\mathrm{Spd}}
\newcommand{\HN}{\mathrm{HN}}
\newcommand{\HNvec}{\overrightarrow{\mathrm{HN}}}
\newcommand{\Surj}{\mathcal{S}\mathrm{urj}}
\newcommand{\Inj}{\mathcal{I}\mathrm{nj}}
\newcommand{\Hom}{\mathcal{H}\mathrm{om}}

\newcommand{\rk}{\mathrm{rk}}
\newcommand{\rkdegsum}{\chi}

\newcommand{\inj}{\hookrightarrow}
\newcommand{\surj}{\twoheadrightarrow}
\newcommand{\nonneg}{{\geq 0}}
\newcommand{\nonpos}{{\leq 0}}
\newcommand{\mumax}{\mu_\text{max}}
\newcommand{\mumin}{\mu_\text{min}}
\newcommand{\trivbundle}{\Ocal}

\newcommand{\undescent}{\mathring}
\newcommand{\vertstretch}{\tilde}
\newcommand{\rankred}{\breve}
\newcommand{\maxslopered}{\overline}
\newcommand{\genslopered}{\tilde}
\newcommand{\commonslopered}{\acute}

\newcommand{\uniformizer}{\pi}
\newcommand{\pseudounif}{\varpi}
\newcommand{\finext}{E}
\newcommand{\algclosedperfdfield}{F}
\newcommand{\genperfdring}{R}
\newcommand{\integerring}[1]{{#1}^\circ}
\newcommand{\witt}{W}

\newcommand{\Perfd}{\mathrm{Perf}}
\newcommand{\Sch}{\mathrm{Sch}}

\numberwithin{equation}{section}

\usepackage{color}

\newcommand{\F}{\mathbb{F}}
\newcommand{\Q}{\mathbb{Q}}
\newcommand{\Z}{\mathbb{Z}}
\newcommand{\R}{\mathbb{R}}

\newcommand{\PP}{\mathbb{P}}

\newcommand{\cris}{\text{cris}}

\DeclareMathOperator{\rank}{rank}

\DeclareMathOperator{\Spec}{Spec\,}

\DeclareMathOperator{\Ext}{Ext}
\DeclareMathOperator{\Pic}{Pic}

\DeclareMathOperator{\Sym}{Sym}

\DeclareMathOperator{\Proj}{Proj\,}

\DeclareMathOperator{\dR}{dR}

\newcommand{\Cal}[1]{\mathcal{#1}}

\DeclareFontFamily{OT1}{rsfs}{}
\DeclareFontShape{OT1}{rsfs}{n}{it}{<-> rsfs10}{}
\DeclareMathAlphabet{\mathscr}{OT1}{rsfs}{n}{it}

\newcommand{\Dcal}{\mathcal{D}}
\newcommand{\Ecal}{\mathcal{E}}
\newcommand{\Fcal}{\mathcal{F}}
\newcommand{\Gcal}{\mathcal{G}}
\newcommand{\Hcal}{\mathcal{H}}

\newcommand{\Ocal}{\mathcal{O}}

\newcommand{\Qcal}{\mathcal{Q}}

\newcommand{\Tcal}{\mathcal{T}}
\newcommand{\Ucal}{\mathcal{U}}
\newcommand{\Vcal}{\mathcal{V}}
\newcommand{\Wcal}{\mathcal{W}}

\newcommand{\Ycal}{\mathcal{Y}}

\newcommand{\ecal}{\mathcal{e}}
\newcommand{\fcal}{\mathcal{f}}

\newcommand{\qcal}{\mathcal{q}}

\newcommand{\dsm}{\oplus}

\newtheorem{thm}{Theorem}[subsection]
\newtheorem{lemma}[subsubsection]{Lemma}
\newtheorem{prop}[subsubsection]{Proposition}
\newtheorem{cor}[subsubsection]{Corollary}

\newtheorem{conj}[subsubsection]{Conjecture}

\theoremstyle{remark}
\newtheorem*{remark}{Remark}
\newtheorem{defn}[subsubsection]{Definition}
\newtheorem{example}[subsubsection]{Example}

\newtheorem*{thm*}{Theorem}
\makeatletter
\def\th@remark{%
  \thm@headfont{\bfseries}%
  \normalfont 
}
\def\imod#1{\allowbreak\mkern5mu({\operator@font mod}\,\,#1)}
\makeatother

\usepackage{enumitem}		
\theoremstyle{theorem}
\newtheorem{theorem}[subsubsection]{Theorem}

\numberwithin{equation}{section}

\makeatother

\begin{document}
	
	\tikzset{
		node style sp/.style={draw,circle,minimum size=\myunit},
		node style ge/.style={circle,minimum size=\myunit},
		arrow style mul/.style={draw,sloped,midway,fill=white},
		arrow style plus/.style={midway,sloped,fill=white},
	}
    
	\title{Classification of quotient bundles on the Fargues-Fontaine curve}
   
    \author[S. Hong]{Serin Hong}
    \address{Department of Mathematics, University of Michigan, 530 Church Street, Ann Arbor MI 48109}
    \email{serinh@umich.edu}
    
    \begin{abstract} We completely classify all quotient bundles of a given vector bundle on the Fargues-Fontaine curve. As consequences, we have two additional classification results: a complete classification of all vector bundles that are generated by a fixed number of global sections, and a nearly complete classification of subsheaves of a given vector bundle. For the proof, we combine the dimension counting argument for moduli of bundle maps developed in \cite{Arizona_extvb} with a series of reduction arguments based on some reinterpretation of the classifying conditions. 

    \end{abstract}
	
	\maketitle

	\tableofcontents
	
	\rhead{}

	\chead{}
\section{Introduction}

In \cite{FF_curve}, Fargues and Fontaine constructed a remarkable scheme, now commonly referred to as the \emph{Fargues-Fontaine curve}, which serves as the ``fundamental curve'' for $p$-adic Hodge theory and the local Langlands program. In fact, many constructions in these fields have geometric interpretations in terms of vector bundles on the Fargues-Fontaine curve. Most notably, Fargues-Scholze \cite{FS_geomLL} builds upon the idea of Fargues \cite{Fargues_geomLL} to construct the local Langlands correspondence in terms of certain sheaves on the stack of vector bundles on the Fargues-Fontaine curve.

In this paper we obtain several classification results regarding vector bundles on the Fargues-Fontaine curve. Our main result is a complete classification of all quotient bundles of a given vector bundle. As a special case, we obtain a complete classification of all vector bundles that are generated by a fixed number of global sections. In addition, a dual statement of our main result gives a nearly complete classification of subsheaves of a given vector bundle.

\subsection{Statement of results}$ $

For a precise statement of our results, we briefly recall the classification of vector bundles on the Fargues-Fontaine curve. 

\begin{theorem}[Fargues-Fontaine {\cite[Th\'eor\`eme 8.2.10]{FF_curve}}, Kedlaya {\cite[Theorem 4.5.7]{Kedlaya_slopefiltrations_revisited}}] \label{classification of vector bundles on FF curve, intro} Fix a prime number $p$. Let $\finext$ be a finite extension of $\Q_p$, and let $\algclosedperfdfield$ be an algebraically closed perfectoid field of characteristic $p$. Denote by $\schff = \schff_{\finext, \algclosedperfdfield}$ the Fargues-Fontaine curve associated to the pair $(\finext, \algclosedperfdfield)$. 

\begin{enumerate}[label=(\arabic*)]
\item The scheme $\schff$ is complete in the sense that the divisor of an arbitrary nonzero rational function on $\schff$ has degree zero. As a consequence, there is a well-defined notion of the slope of a vector bundle on $\schff$. 
\smallskip

\item For every rational number $\lambda$, there is a unique stable bundle of slope $\lambda$ on $\schff$, denoted by $\trivbundle(\lambda)$. 
\smallskip

\item Every semistable bundle of slope $\lambda$ is of the form $\trivbundle(\lambda)^{\oplus m}$. 
\smallskip

\item\label{HN decomp of vector bundles, intro} Every vector bundle $\Vcal$ on $\schff$ admits a canonical Harder-Narasimhan filtration which splits into a direct decomposition
\[
\Vcal \simeq \bigoplus_i \trivbundle(\lambda_i)^{\oplus m_i}
\]
where $\lambda_i$'s run over the Harder-Narasimhan slopes of $\Vcal$; in other words, the isomorphism class of $\Vcal$ is determined by the Harder-Narasimhan polygon $\HN(\Vcal)$ of $\Vcal$ (as defined in Definition \ref{def of HN filtration, decomposition and polygon}).
\end{enumerate}
\end{theorem}

For a vector bundle $\Vcal$ with a direct sum decomposition as in \ref{HN decomp of vector bundles, intro} of Theorem \ref{classification of vector bundles on FF curve, intro}, we set
\[ \Vcal^{\leq \mu}:= \bigoplus_{\lambda_i \leq \mu} \trivbundle(\lambda_i)^{\oplus m_i}\quad\quad \text{ and } \quad\quad \Vcal^{\geq \mu}:= \bigoplus_{\lambda_i \geq \mu} \trivbundle(\lambda_i)^{\oplus m_i} \quad\quad\text{ for every } \mu \in \Q.\]
Now we can state our main result as follows:

\begin{theorem}\label{classification of quotient bundles, intro}
Let $\Ecal$ be a vector bundle on $\schff$. Then a vector bundle $\Fcal$ on $\schff$ is a quotient bundle of $\Ecal$ if and only if the following equivalent conditions are satisfied:
\begin{enumerate}[label=(\roman*)]
\item\label{rank inequalities for quotients, intro} For every $\mu \in \Q$, we have $\rank(\Ecal^{\leq \mu}) \geq \rank(\Fcal^{\leq \mu})$ with equality if and only if $\Ecal^{\leq \mu}$ and $\Fcal^{\leq \mu}$ are isomorphic. 

\smallskip




\item\label{dual strong slopewise dominance for quotients, intro} If we align the Harder-Narasimhan polygons $\HN(\Ecal)$ and $\HN(\Fcal)$ so that their right endpoints lie at the origin, then for each $i = 1, \cdots, \rank(\Fcal)$, the slope of $\HN(\Fcal)$ on $[-i, -i+1]$ is greater than or equal to the slope of $\HN(\Ecal)$ on $[-i-1, -i]$ unless $\HN(\Ecal)$ and $\HN(\Fcal)$ agree on $[-i, 0]$. 
\end{enumerate}

\smallskip
\begin{figure}[H]
\begin{tikzpicture}[scale=1]	

		\coordinate (right) at (0, 0);
		\coordinate (q0) at (-1,2);
		\coordinate (q1) at (-2.5, 3.4);
		\coordinate (q2) at (-6, 4.8);
		\coordinate (q3) at (-9, 4);
		

		\coordinate (p0) at (-2, 1.5);
		\coordinate (p1) at (-4.5, 2);
		\coordinate (p2) at (-6, 1.3);
		\coordinate (p3) at (-7, 0.1);
				
		\draw[step=1cm,thick] (right) -- (q0) --  (q1) -- (q2) -- (q3);
		\draw[step=1cm,thick] (right) -- (p0) --  (p1) -- (p2) -- (p3);
		
		\draw [fill] (q0) circle [radius=0.05];		
		\draw [fill] (q1) circle [radius=0.05];		
		\draw [fill] (q2) circle [radius=0.05];		
		\draw [fill] (q3) circle [radius=0.05];		
		\draw [fill] (right) circle [radius=0.05];
		
		\draw [fill] (p0) circle [radius=0.05];		
		\draw [fill] (p1) circle [radius=0.05];		
		\draw [fill] (p2) circle [radius=0.05];		
		\draw [fill] (p3) circle [radius=0.05];		
		
		\draw[step=1cm,dotted] (-6.5, -0.4) -- (-6.5, 5);
       		\draw[step=1cm,dotted] (-5.5, -0.4) -- (-5.5, 5);
       		\draw[step=1cm,dotted] (-6, -0.4) -- (-6, 5);

		\node at (-6.8,-0.8) {\scriptsize $-i-1$};
		\node at (-5.2,-0.8) {\scriptsize $-i+1$};
		\node at (-6,-0.8) {\scriptsize $-i$};
		
		\path (q3) ++(-0.8, 0.05) node {$\HN(\Ecal)$};
		\path (p3) ++(-0.8, 0.05) node {$\HN(\Fcal)$};
		\path (right) ++(0.3, -0.05) node {$O$};

\end{tikzpicture}
\caption{Illustration of the condition \ref{dual strong slopewise dominance for quotients, intro} in Theorem \ref{classification of quotient bundles, intro}.}
\end{figure}
\end{theorem}


The necessity part of Theorem \ref{classification of quotient bundles, intro} is a consequence of the slope formalism for vector bundles on the Fargues-Fontaine curve. Hence Theorem \ref{classification of quotient bundles, intro} asserts that the slope formalism is the only obstruction for the existence of a surjective bundle map between two given vector bundles on $\schff$. We emphasize that this is a very unique feature for the slope category of vector bundles on the Fargues-Fontaine curve. The main reason for this feature is that for any given vector bundles $\Vcal, \Wcal$ on $\schff$ the space $\mathrm{Hom}(\Vcal, \Wcal)$ is either empty or huge, where the nonemptiness is determined by the slope formalism. 


If we take $\Ecal = \trivbundle_\schff^{\oplus n}$ for some positive integer $n$ in Theorem \ref{classification of quotient bundles, intro}, we obtain the following classification of finitely globally generated vector bundles on $\schff$. 

\begin{cor}\label{classification of globally generated bundles, intro}
A vector bundle $\Fcal$ on $\schff$ is generated by $n$ global sections if and only if the following conditions are satisfied:
\begin{enumerate}[label=(\roman*)]
\item\label{nonpositivity of slopes for globally generated bundles, intro} All Harder-Narasimhan slopes of $\Fcal$ are nonnegative. 

\item\label{rank bound for globally generated bundles, intro} $\rank(\Fcal) \leq n$ with equality if and only if $\Fcal \simeq \trivbundle_\schff^{\oplus n}$. 
\end{enumerate}
\end{cor}

In addition, dualizing the statement of Theorem \ref{classification of quotient bundles, intro} yields a classification of a majority of subsheaves of a given vector bundle on $\schff$.

\begin{cor}\label{almost classification of subbundles, intro}
Let $\Ecal$ be a vector bundle on $\schff$. Then a vector bundle $\Dcal$ on $\schff$ is 
a subsheaf of $\Ecal$ if the following equivalent conditions are satisfied:
\begin{enumerate}[label=(\roman*)]
\item\label{rank inequalities for subbundles, intro} For every $\mu \in \Q$, we have $\rank(\Ecal^{\geq \mu}) \geq \rank(\Dcal^{\geq \mu})$ with equality if and only if $\Ecal^{\geq \mu}$ and $\Dcal^{\geq \mu}$ are isomorphic. 

\smallskip




\item\label{strong slopewise dominance for subbundles, intro} If we align the Harder-Narasimhan polygons $\HN(\Dcal)$ and $\HN(\Ecal)$ so that their left endpoints lie at the origin, then for each $i = 1, \cdots, \rank(\Dcal)$, the slope of $\HN(\Dcal)$ on $[i-1, i]$ is less than or equal to the slope of $\HN(\Ecal)$ on $[i, i+1]$ unless $\HN(\Dcal)$ and $\HN(\Ecal)$ agree on $[0, i]$. 
\end{enumerate}

\smallskip
\begin{figure}[H]
\begin{tikzpicture}[scale=1]

		\coordinate (left) at (0, 0);
		\coordinate (q0) at (1,2);
		\coordinate (q1) at (2.5, 3.4);
		\coordinate (q2) at (6, 4.8);
		\coordinate (q3) at (9, 4);
		

		\coordinate (p0) at (2, 1.5);
		\coordinate (p1) at (4.5, 2);
		\coordinate (p2) at (6, 1.3);
		\coordinate (p3) at (7, 0.1);
				
		\draw[step=1cm,thick] (left) -- (q0) --  (q1) -- (q2) -- (q3);
		\draw[step=1cm,thick] (left) -- (p0) --  (p1) -- (p2) -- (p3);
		
		\draw [fill] (q0) circle [radius=0.05];		
		\draw [fill] (q1) circle [radius=0.05];		
		\draw [fill] (q2) circle [radius=0.05];		
		\draw [fill] (q3) circle [radius=0.05];		
		\draw [fill] (left) circle [radius=0.05];
		
		\draw [fill] (p0) circle [radius=0.05];		
		\draw [fill] (p1) circle [radius=0.05];		
		\draw [fill] (p2) circle [radius=0.05];		
		\draw [fill] (p3) circle [radius=0.05];		
		
		\draw[step=1cm,dotted] (5.5, -0.4) -- (5.5, 5);
       		\draw[step=1cm,dotted] (6.5, -0.4) -- (6.5, 5);
       		\draw[step=1cm,dotted] (6, -0.4) -- (6, 5);

		\node at (5.2,-0.8) {\scriptsize $i-1$};
		\node at (6.8,-0.8) {\scriptsize $i+1$};
		\node at (6,-0.8) {\scriptsize $i$};
		
		\path (q3) ++(0.8, 0.05) node {$\HN(\Ecal)$};
		\path (p3) ++(0.8, 0.05) node {$\HN(\Dcal)$};
		\path (right) ++(-0.3, -0.05) node {$O$};

\end{tikzpicture}
\caption{Illustration of the condition \ref{strong slopewise dominance for subbundles, intro} in Corollary \ref{almost classification of subbundles, intro}.}
\end{figure}
\end{cor}

We remark that Corollary \ref{almost classification of subbundles, intro} does not give a complete classification of all subsheaves. The main issue is that the cokernel of an injective bundle map may have a torsion while the kernel of a surjective bundle map is always torsion-free. In fact, Corollary \ref{almost classification of subbundles, intro} gives a complete classification of all subsheaves with torsion-free cokernel. 

It is also worthwhile to note that our main result holds verbatim for the projective line $\PP^1$ over an arbitrary field. While the statement for $\PP^1$ can be proved by some elementary linear algebra, it can also be seen by the same proof for Theorem \ref{classification of quotient bundles, intro} using a dimension formula for spaces of bundle maps between two vector bundles on $\PP^1$. 
We refer the readers to the appendix of this article for details.


\subsection{Applications of the main result}\label{intro applications}$ $

Our main result and its proof have applications in some problems that naturally arise in $p$-adic geometry. 

First, Theorem \ref{classification of quotient bundles, intro} has an application towards the problem of classifying all vector bundles $\Ecal$ on $\schff$ that arise as an extension of two fixed vector bundles $\Dcal$ and $\Fcal$ on $\schff$. In fact, this is the main motivating problem for our work, as it naturally arises in the study of geometric objects such as the stack of vector bundles on $\schff$ and the flag varieties. When both $\Dcal$ and $\Fcal$ are semistable, we have a complete answer by the work of the author and his collaborators in \cite{Arizona_extvb}, which in turn leads to an explicit description of the connected components of the stack of vector bundles on $\schff$ by Hansen \cite{Hansen_degenvb}. In the subsequent paper \cite{Hong_extvb}, the author applies Theorem \ref{classification of quotient bundles, intro} to extend the main result of \cite{Arizona_extvb} as follows:

\begin{thm}[{\cite[Theorem 3.2.1]{Hong_extvb}}]\label{classification of extensions under semistable assumption}
Let $\Dcal, \Ecal$, and $\Fcal$ be vector bundles on $\schff$ such that one of $\Dcal, \Ecal$, and $\Fcal$ is semistable. Assume that the maximum slope of $\Dcal$ is less than the minimum slope of $\Fcal$. Then there exists a short exact sequence of the form
\[0 \longrightarrow \Dcal \longrightarrow \Ecal \longrightarrow \Fcal \longrightarrow 0\]
if and only if the following conditions are satisfied:
\begin{enumerate}[label = (\roman*)]
\item\label{extension existence of surj map} For every $\mu \in \Q$, we have $\rank(\Ecal^{\leq \mu}) \geq \rank(\Fcal^{\leq \mu})$ with equality if and only if $\Ecal^{\leq \mu}$ and $\Fcal^{\leq \mu}$ are isomorphic. 
\smallskip

\item\label{extension existence of dual surj map} For every $\mu \in \Q$, we have $\rank(\Ecal^{\geq \mu}) \geq \rank(\Dcal^{\geq \mu})$ with equality if and only if $\Ecal^{\geq \mu}$ and $\Dcal^{\geq \mu}$ are isomorphic.
\smallskip

\item\label{extension HN polygons inequality} $\HN(\Dcal \oplus \Fcal) \geq \HN(\Ecal)$, which means that $\HN(\Dcal \oplus \Fcal)$ lies above $\HN(\Ecal)$ with the same endpoints. 
\end{enumerate}
\end{thm}

The conditions \ref{extension existence of surj map} and \ref{extension existence of dual surj map} can be stated purely in terms of HN polygons as in Theorem \ref{classification of quotient bundles, intro}. Note that these conditions are clearly necessary; indeed, they are equivalent to existence of surjective bundle maps $\Ecal \surj \Fcal$ and $\Ecal^\vee \surj \Dcal^\vee$ by Theorem \ref{classification of quotient bundles, intro}, where $\Ecal^\vee$ and $\Dcal^\vee$ denote the duals of $\Ecal$ and $\Dcal$. In addition, the necessity of the condition \ref{extension HN polygons inequality} is a consequence of the slope formalism. 

We expect that Theorem \ref{classification of extensions under semistable assumption} holds without the semistability assumption on one of $\Dcal, \Ecal$, or $\Fcal$. With this generalization of Theorem \ref{classification of extensions under semistable assumption}, we should be able to obtain an explicit description on the geometry of the $p$-adic flag variety in terms of two natural stratifications, namely the Harder-Narasimhan stratification and the Newton stratification, in line with the work of many authors including Caraiani-Scholze \cite{CS_annals}, Chen-Fargues-Shen \cite{CFS_admlocus}, Shen \cite{Shen_HNstrata}, Chen \cite{Chen_FRconjnonbasic}, Viehmann \cite{Viehmann_weakadmlocNewton}, and Nguyen-Viehmann \cite{NV_HNstrata}.

As another application, the author in the sequel paper \cite{Hong_subvb} adapts our argument and several constructions from this paper to establish a complete classification for subsheaves as follows:
\begin{thm}[{\cite[Theorem 3.1.1]{Hong_subvb}}]\label{conjecture classification of subbundles, intro}
Let $\Ecal$ be a vector bundle on $\schff$. Then a vector bundle $\Dcal$ is a subsheaf of $\Ecal$ if and only if it satisfies the following equivalent conditions:
\begin{enumerate}[label=(\roman*)]
\item\label{rank inequalities for subsheaves, intro} For every $\mu \in \Q$, we have $\rank(\Ecal^{\geq \mu}) \geq \rank(\Dcal^{\geq \mu})$.

\smallskip

\item\label{slopewise dominance for subsheaves, intro} If we align the Harder-Narasimhan polygons $\HN(\Dcal)$ and $\HN(\Ecal)$ so that their left endpoints lie at the origin, then for each $i = 1, \cdots, \rank(\Dcal)$, the slope of $\HN(\Dcal)$ on $[i-1, i]$ is less than or equal to the slope of $\HN(\Ecal)$ on $[i-1, i]$.  
\end{enumerate}
\end{thm}
Theorem \ref{conjecture classification of subbundles, intro} in particular gives a classification of all subsheaves $\Dcal$ of $\Ecal$ such that $\Ecal/\Dcal$ is a torsion sheaf. Such subsheaves are of particular interest as they arise from \emph{modifications of vector bundles}, which play a pivotal role in the study of various geometric objects such as the $B_{\dR}^+$-affine Grassmannians and the moduli of local shtukas.


\subsection{Outline of the strategy}\label{introstrategy}$ $

It is relatively easy to see that the condition \ref{rank inequalities for quotients, intro} in Theorem \ref{classification of quotient bundles, intro} is indeed necessary and that it is equivalent to the condition \ref{dual strong slopewise dominance for quotients, intro}.  Therefore the main part of our proof will concern the sufficiency of the condition \ref{rank inequalities for quotients, intro} in Theorem \ref{classification of quotient bundles, intro}. 

Our argument will be based on the dimension counting method for certain moduli spaces of bundle maps as developed in \cite{Arizona_extvb}. We define the moduli functors
\begin{itemize}
\item $\Hom(\Ecal, \Fcal)$ which parametrizes bundle maps $\Ecal \to \Fcal$, and

\item $\Surj(\Ecal, \Fcal)$ which parametrizes surjective bundle maps $\Ecal \surj \Fcal$.
\end{itemize}
These functors are represented by diamonds in the sense of Scholze \cite{Scholze_diamonds}. The goal is to show that 
the diamond $\Surj(\Ecal, \Fcal)$ is not empty if the condition \ref{rank inequalities for quotients, intro} in Theorem \ref{classification of quotient bundles, intro} is satisfied. To this end, we consider auxiliary spaces $\Hom(\Ecal, \Fcal)_\Qcal$ which (roughly) parametrizes bundle maps $\Ecal \to \Fcal$ with image isomorphic to a specified subsheaf $\Qcal$ of $\Fcal$. Then showing nonemptiness of $\Surj(\Ecal, \Fcal)$ boils down to establishing the following inequality on dimensions of the topological spaces:
\begin{equation}\label{key inequality intro}
\dim |\Hom(\Ecal, \Fcal)_\Qcal| < \dim |\Hom(\Ecal, \Fcal)| \quad\quad \text{ if } \Qcal \neq \Fcal.
\end{equation}
The dimension theory for diamonds allows us to rewrite this inequality in terms of degrees of certain vector bundles related to $\Ecal, \Fcal$ and $\Qcal$. 	

However, the details of our arguments are 
completely different from those in \cite{Arizona_extvb}. The main reason is that, unlike the quantities considered in \cite{Arizona_extvb}, the quantities we need to study in this paper do not generally have good interpretations in terms of areas of polygons related to the Harder-Narasimhan slopes. In fact, our proof of the inequality \eqref{key inequality intro} will consist of a series of reduction steps as follows:
\begin{enumerate}[label=Step \arabic*., leftmargin=5.3em]
\item\label{reduction to integer slopes, intro} We reduce the proof of \eqref{key inequality intro} to the case when all slopes of $\Ecal, \Fcal$ and $\Qcal$ are integers. 

\item\label{reduction to equal ranks, intro} We further reduce the proof of \eqref{key inequality intro} to the case $\rank(\Qcal) = \rank(\Fcal)$. 

\item\label{reduction to equal slopes, intro} When $\rank(\Qcal) = \rank(\Fcal)$, we complete the proof of \eqref{key inequality intro} by gradually ``reducing" the slopes of $\Fcal$ to the slopes of $\Qcal$.  
\end{enumerate}

As a key ingredient of our reduction argument, we introduce and study the notion of \emph{slopewise dominance} for vector bundles on the Fargues-Fontaine curve. This notion provides a combinatorial interpretation of the inequality in the condition \ref{rank inequalities for quotients, intro} of Theorem \ref{classification of quotient bundles, intro} in terms of Harder-Narasimhan polygons, and allows us to use the equivalence between the conditions \ref{rank inequalities for quotients, intro} and \ref{dual strong slopewise dominance for quotients, intro} of Theorem \ref{classification of quotient bundles, intro} to its full capacity. In particular, this notion yields several implications of the condition \ref{dual strong slopewise dominance for quotients, intro} which are difficult to directly deduce from the condition \ref{rank inequalities for quotients, intro}, and plays a pivotal role in our process of ``reducing" the slopes of $\Fcal$ to the slopes of $\Qcal$ in \ref{reduction to equal slopes, intro}. This notion is also crucial for applications of Theorem \ref{classification of quotient bundles, intro}, such as Theorem \ref{conjecture classification of subbundles, intro} and Theorem \ref{classification of extensions under semistable assumption} which are discussed in the sequel papers \cite{Hong_extvb} and \cite{Hong_subvb}.



\subsection*{Acknowledgments} The major part of this study was done at the Oberwolfach workshop on the arithmetic of Shimura varieties. The author would like to thank the organizers of the workshop for creating such a wonderful academic environment. The author also would like to sincerely thank David Hansen for a stimulating discussion about the problem, and the anonymous referee for their valuable suggestions which greatly helped in improving and clarifying the manuscript.


\section{Preliminaries on the Fargues-Fontaine curve}\label{background}

\subsection{The construction}$ $

Throughout this paper, we fix the following data:
\begin{itemize}
\item $p$ is a prime number;

\item $\finext$ is a finite extension of $\Q_p$ with residue field $\F_q$;

\item $\algclosedperfdfield$ is an algebraically closed perfectoid field of characteristic $p$. 
\end{itemize}
The Fargues-Fontaine curve can be constructed in two different flavors, namely as a scheme and as an adic space. We first present the construction as an adic space since it is simpler to describe than the construction as a scheme is. 


\begin{defn}\label{adicFFC} Denote by $\integerring{\finext}$ and $\integerring{\algclosedperfdfield}$ the rings of integers of $\finext$ and $\algclosedperfdfield$, respectively. Let $\uniformizer$ be a uniformizer of $\finext$, and let $\pseudounif$ be a pseudouniformizer of $\algclosedperfdfield$. We write $\witt_{\finext^\circ}(\integerring{\algclosedperfdfield}):=\witt(\integerring{\algclosedperfdfield}) \otimes_{\witt(\F_q)} \integerring{\finext}$ for the ring  of ramified Witt vectors of $\integerring{\algclosedperfdfield}$ with coefficients in $\integerring{\finext}$, and $[\pseudounif]$ for the Teichmuller lift of $\pseudounif$. Define
\[
\Ycal_{\finext,\algclosedperfdfield}:=\Spa(\witt_{\integerring{\finext}}(\integerring{\algclosedperfdfield}))\setminus\{|p[\pseudounif]|=0\},
\]
and let $\phi:\Ycal_{\finext,\algclosedperfdfield}\to \Ycal_{\finext,\algclosedperfdfield}$ be the Frobenius automorphism of $\Ycal_{\finext,\algclosedperfdfield}$ induced by the $q$-Frobenius $\varphi_q$ on $\witt_{\integerring{\finext}}(\integerring{\algclosedperfdfield})$. The (mixed-characteristic) \emph{adic Fargues-Fontaine curve} associated to the pair $(\finext, \algclosedperfdfield)$ is
\[
\adicff_{\finext,\algclosedperfdfield}:=\Ycal_{\finext,\algclosedperfdfield}/\phi^\Z.
\]
\end{defn}

\begin{remark}
This definition makes sense since the action of $\phi$ on $\Ycal_{\finext,\algclosedperfdfield}$ turns out to be properly discontinuous. 
\end{remark}

\begin{prop}[{\cite[Theorem 4.10]{Kedlaya_noeth}}] $\adicff_{\finext,\algclosedperfdfield}$ is a Noetherian adic space over $\Spa(\finext)$.
\end{prop}

\begin{remark}
When $\finext$ is replaced by a finite extension of $\F_p((t))$, there is a related construction by Hartl-Pink \cite{HP_equalcharFFcurve} which we may regard as the equal-characteristic Fargues-Fontaine curve. Our main results are equally valid for vector bundles on the equal-characteristic Fargues-Fontaine curve with the same proof. 
\end{remark}

Our next goal is to relate the above construction of $\adicff_{\finext,\algclosedperfdfield}$ to the schematic construction of the Fargues-Fontaine curve. To this end, 
we first define some vector bundles on $\adicff_{\finext,\algclosedperfdfield}$. By descent, giving a vector bundle $\Vcal$ on $\adicff_{\finext,\algclosedperfdfield}$ amounts to giving a $\phi$-equivariant vector bundle $\undescent{\Vcal}$ on $\Ycal_{\finext,\algclosedperfdfield}$, that is, a vector bundle $\undescent{\Vcal}$ on $\Ycal_{\finext,\algclosedperfdfield}$ together with an isomorphism $\phi^*\undescent{\Vcal}\stackrel{\sim}{\to}\undescent{\Vcal}$. 

\begin{defn}
\label{o-r-over-s}
Let $\lambda = r/s$ be a rational number written in lowest terms with $s>0$. Let $v_1, v_2, \cdots, v_s$ be a trivializing basis of $\trivbundle_{\Ycal_{\finext,\algclosedperfdfield}}^{\oplus s}$. Define an isomorphism $\phi^* \trivbundle_{\Ycal_{\finext,\algclosedperfdfield}}^{\oplus s} \stackrel{\sim}{\to} \trivbundle_{\Ycal_{\finext,\algclosedperfdfield}}^{\oplus s}$ by 
\[ v_1 \mapsto v_2, \quad v_2 \mapsto v_3, \quad\cdots, \quad v_{s-1} \mapsto v_s, \quad v_s \mapsto \uniformizer^{-r} v_1,\]
where we abuse notation to view $v_1, v_2, \cdots, v_s$ as a trivializing basis for $\phi^* \trivbundle_{\Ycal_{\finext,\algclosedperfdfield}}^{\oplus s}$ as well. We write $\trivbundle(\lambda)$ for the vector bundle on $\adicff_{\finext,\algclosedperfdfield}$ corresponding to the vector bundle $\trivbundle_{\Ycal_{\finext,\algclosedperfdfield}}^{\oplus s}$ with the isomorphism $\phi^* \trivbundle_{\Ycal_{\finext,\algclosedperfdfield}}^{\oplus s} \stackrel{\sim}{\to} \trivbundle_{\Ycal_{\finext,\algclosedperfdfield}}^{\oplus s}$ as defined above. 
\end{defn}

The following fact suggests that we can regard $\trivbundle(1)$ as an ``ample" line bundle on $\adicff_{\finext,\algclosedperfdfield}$. 
\begin{prop}[{\cite[Lemma 8.8.4 and Proposition 8.8.6]{KL15}}] Let $\Fcal$ be a coherent sheaf on $\adicff_{\finext,\algclosedperfdfield}$. Then for all sufficiently large $n \in \Z$, the twisted sheaf $\Fcal(n):= \Fcal \otimes \trivbundle(1)^{\otimes n}$ satisfies the following properties:
\begin{enumerate}[label = (\roman*)]
\item $H^1(\adicff_{\finext,\algclosedperfdfield}, \Fcal(n)) = 0$.
\smallskip

\item The sheaf $\Fcal(n)$ is generated by finitely many global sections. 
\end{enumerate}
\end{prop}

We now recover the schematic construction of the Fargues-Fontaine curve as follows:
\begin{defn}\label{schematic FF curve}
We define the \emph{schematic Fargues-Fontaine curve} associated to the pair $(\finext,\algclosedperfdfield)$ by 
\[
\schff_{\finext,\algclosedperfdfield} := \Proj \left( \bigoplus_{n \geq 0 } H^0(\adicff_{\finext,\algclosedperfdfield}, \trivbundle(n) ) \right).
\]
\end{defn}

\begin{remark}
The original construction of the schematic Fargues-Fontaine curve in \cite[\S6]{FF_curve} was given in terms of the period ring $B_\cris^+$ in $p$-adic Hodge theory (see also \cite{FF_curvesurvey}, \S4.1):
\[
\schff_{\finext,\algclosedperfdfield} = \Proj \left( \bigoplus_{n \geq 0 } (B_\cris^+)^{\varphi_q = \uniformizer^n}\right).
\]
This definition agrees with Definition \ref{schematic FF curve} via the identification $H^0(\adicff_{\finext,\algclosedperfdfield}, \trivbundle(n) ) \simeq  (B_\cris^+)^{\varphi_q = \uniformizer^n}$.
\end{remark}

\begin{prop}[{\cite[Th\'eor\`eme 6.5.2]{FF_curve}}]\label{why FF curve is a curve} The scheme $\schff_{\finext,\algclosedperfdfield}$ is noetherian, connected, and regular of Krull dimension one.
\end{prop}

\begin{remark}
The scheme $\schff_{\finext,\algclosedperfdfield}$ admits a natural structure morphsim $\schff_{\finext,\algclosedperfdfield} \to \Spec(\finext)$ induced by a canonical isomorphism $(B_\cris^+)^{\varphi_q = 1} \cong \finext$. However, it is not a curve in the usual sense;
in fact, one can show that the residue field at a closed point on $\schff_{\finext,\algclosedperfdfield}$ is a complete algebraically closed extension of $\finext$, which in particular implies that $\schff_{\finext,\algclosedperfdfield}$ is not of finite type over $\finext$. 
\end{remark}

For our purpose, the two constructions of the Fargues-Fontaine curve are essentially equivalent, as we have a version of GAGA for the Fargues-Fontaine curve. 
\begin{theorem}[{\cite[Theorems 6.3.12 and 8.7.7]{KL15}}]\label{GAGA for FF curve}
There is a natural map 
\[
\adicff_{\finext,\algclosedperfdfield} \rightarrow \schff_{\finext,\algclosedperfdfield}
\]
which induces by pullback an equivalence of the categories of vector bundles. 
\end{theorem}

Following Kedlaya-Liu \cite[\S8.7]{KL15}, we can extend the construction of the adic Fargues-Fontaine curve to relative settings.
\begin{defn}\label{relative FF curve}
Let $S = \Spa (\genperfdring, \genperfdring^+)$ be an affinoid perfectoid space over $\Spa (\algclosedperfdfield)$, and let $\pseudounif_\genperfdring$ be a pseudouniformizer of $\genperfdring$. Denote by $\integerring{\finext}$ the ring of integers of $\finext$, and by $\integerring{\genperfdring}$ the ring of power bounded elements of $\genperfdring$. We take the ring of ramified Witt vectors
$\witt_{\finext^\circ}(\genperfdring^+):=\witt(\genperfdring^+) \otimes_{\witt(\F_q)} \integerring{\finext}$
and write $[\pseudounif_\genperfdring]$ for the Teichmuller lift of $\pseudounif_\genperfdring$. Define
\[
\Ycal_{\finext,S}:=\Spa(\witt_{\integerring{\finext}}(\genperfdring^+), \witt_{\integerring{\finext}}(\genperfdring^+))\setminus\{|p[\pseudounif_\genperfdring]|=0\},
\]
and let $\phi:\Ycal_{\finext,S}\to \Ycal_{\finext,S}$ be the Frobenius automorphism of $\Ycal_{\finext,S}$ induced by the $q$-Frobenius $\varphi_q$ on $\witt_{\integerring{\finext}}({\genperfdring}^+)$. The \emph{relative adic Fargues-Fontaine curve} associated to the pair $(\finext, S)$ is 
\[
\adicff_{\finext,S}:=\Ycal_{\finext,S}/\phi^\Z.
\]
More generally, for an arbitrary perfectoid space $S$ over $\Spa(\algclosedperfdfield)$, we choose an affinoid cover $S = \bigcup S_i = \bigcup \Spa(\genperfdring_i, \genperfdring_i^+)$ and define the relative adic Fargues-Fontaine curve $\adicff_{\finext,S}$ by gluing the adic spaces $\adicff_{\finext, S_i}$. 
\end{defn}

\begin{remark}
By construction, the relative curve $\adicff_{\finext, S}$ comes with a natural map $\adicff_{\finext, S} \to \adicff_{\finext, \algclosedperfdfield}$. However, the relative curve $\adicff_{\finext, S}$ cannot be obtained from $\adicff_{\finext, \algclosedperfdfield}$ by base change; indeed, neither $\adicff_{\finext, S}$ nor $\adicff_{\finext, \algclosedperfdfield}$ is defined over $\Spa(\algclosedperfdfield)$. 
\end{remark}

\subsection{Classification of vector bundles and Harder-Narasimhan polygons}\label{basic HN theory for vector bundles on FF curve}$ $

For the rest of this paper, we will simply write $\adicff := \adicff_{\finext,\algclosedperfdfield}$ and $\schff := \schff_{\finext,\algclosedperfdfield}$. Moreover, we will speak interchangeably about vector bundles on $\adicff$ and $\schff$ in light of Theorem \ref{GAGA for FF curve}. 

In this subsection we review the main classification theorem for vector bundles on the Fargues-Fontaine curve and discuss some of its immediate consequences.


\begin{prop}[{\cite[\S8.2.1]{FF_curve}}]\label{completeness of FF curve}
There exists a natural isomorphism from the Picard group $\Pic(\schff)$ to $\Z$, which maps $\trivbundle(d)$ to $d$ for any $d \in \Z$. 
\end{prop}


\begin{defn}
Let $\Vcal$ be a nonzero vector bundle on $\schff$. 
\begin{enumerate}[label=(\arabic*)]
\item We write $\rk(\Vcal)$ for the rank of $\Vcal$, and $\Vcal^\vee$ for the dual bundle of $\Vcal$.  
\smallskip

\item We define the \emph{degree} of $\Vcal$, denoted by $\deg(\Vcal)$, to be the image of the determinant line bundle $\wedge^{\rk(\Vcal)} (\Vcal)$ under the natural isomorphism $\Pic(\schff) \cong \Z$ in Proposition \ref{completeness of FF curve}. 
\smallskip

\item We define the \emph{slope} of $\Vcal$ by
\[\mu(\Vcal) := \dfrac{\deg(\Vcal)}{\rk(\Vcal)}.\]
\end{enumerate}
\end{defn}

Let us now recall the usual notions of stability and semistability. 

\begin{defn}
Let $\Vcal$ be a nonzero vector bundle on $\adicff$. 
\begin{enumerate}[label=(\arabic*)]
\item We say that $\Vcal$ is \emph{stable} if $\mu(\Wcal) < \mu(\Vcal)$ for all nonzero proper subbundles $\Wcal \subset \Vcal$.
\smallskip

\item We say that $\Vcal$ is \emph{semistable} if $\mu(\Wcal) \leq \mu(\Vcal)$ for all nonzero proper subbundles $\Wcal \subset \Vcal$. 
\end{enumerate}
\end{defn}

We collect some fundamental facts about semistable vector bundles on $\adicff$. 

\begin{prop}[{\cite[Th\'eor\`eme 8.2.10]{FF_curve}}]\label{classification of semistable bundles}
Let $\lambda$ be a rational number. 
\begin{enumerate}[label=(\arabic*)]
\item The vector bundle $\trivbundle(\lambda)$ represents the unique isomorphism class of stable vector bundles on $\adicff$ of slope $\lambda$. 
\smallskip

\item Every semistable vector bundle of slope $\lambda$ is isomorphic to $\trivbundle(\lambda)^{\oplus n}$ for some $n$. 
\end{enumerate}
\end{prop}


\begin{lemma}\label{basic properties of stable bundles}
Let $r$ and $s$ be relatively prime integers with $s>0$. 
\begin{enumerate}[label=(\arabic*)]
\item\label{rank and degree of stable bundles} The bundle $\trivbundle(r/s)$ has rank $s$, degree $r$, and slope $r/s$.
\smallskip

\item\label{tensor product of stable bundles} For any relatively prime integers $r'$ and $s'$ with $s'>0$, we have
\[
\trivbundle\left(\frac rs\right)\otimes\trivbundle\left(\frac{r'}{s'}\right)\simeq\trivbundle\left(\frac rs+\frac{r'}{s'}\right)^{\dsm\gcd(s s',rs'+r's)}.
\]
In particular, the bundle $\trivbundle(r/s)\otimes\trivbundle(r'/s')$ has rank $ss'$, degree $rs'+r's$, and slope $r/s+r'/s'$.
\smallskip

\item\label{dual of stable bundles} $\trivbundle(r/s)^\vee \simeq \trivbundle(-r/s)$. 
\end{enumerate}
\end{lemma}
\begin{proof}
All statements are straightforward to check using Definition \ref{o-r-over-s}. 
\end{proof}

\begin{theorem}[{\cite[Proposition 8.2.3]{FF_curve}},{\cite[Proposition 4.1.3]{Kedlaya_slopefiltrations_revisited}}]\label{cohomological vanishing of stable bundles}
We have the following cohomological computations:
\begin{enumerate}[label=(\arabic*)]
\item $H^0(\adicff, \trivbundle(\lambda)) = 0$ if and only if $\lambda < 0$. 
\smallskip

\item $H^1(\adicff, \trivbundle(\lambda)) = 0$ if and only if $\lambda \geq 0$. 
\end{enumerate}
\end{theorem}


It turns out that every vector bundle on $\adicff$ admits a direct sum decomposition into stable bundles, as stated in the following theorem:
\begin{theorem}[{\cite[Th\'eor\`eme 8.2.10]{FF_curve}}]\label{existence of HN decomp}
Every vector bundle $\Vcal$ on $\adicff$ admits a unique filtration 
\begin{equation}\label{HN filtration}
0 = \Vcal_0 \subset \Vcal_1 \subset \dotsb \subset \Vcal_l = \Vcal
\end{equation}
such that the successive quotients $\Vcal_i/\Vcal_{i-1}$ are semistable vector bundles with
\[\mu(\Vcal_1/\Vcal_0) > \mu(\Vcal_2/\Vcal_1) > \cdots > \mu(\Vcal_l/\Vcal_{l-1}).\]
Moreover, the filtration \eqref{HN filtration} splits into a direct sum decomposition
\begin{equation}\label{HN decomposition}
\Vcal \simeq \bigoplus_{i=1}^l \trivbundle(\lambda_i)^{\oplus m_i}
\end{equation}
where $\lambda_i = \mu(\Vcal_i/\Vcal_{i-1})$. 
\end{theorem}

\begin{remark}
The existence and uniqueness of the filtration \eqref{HN filtration} is a formal consequence of Propositions \ref{why FF curve is a curve} and \ref{completeness of FF curve}, as noted by Kedlaya \cite[\S3.4]{Kedlaya_arizona}. The existence of the direct sum decomposition \ref{HN decomposition} then follows by Proposition \ref{classification of semistable bundles} and Theorem \ref{cohomological vanishing of stable bundles}. We note that the recent work of Fargues-Scholze \cite{FS_geomLL} gives a new proof of Proposition \ref{classification of semistable bundles} (and thus Theorem \ref{existence of HN decomp}) which conceptualize the original proof of Fargues-Fontaine \cite{FF_curve}.
\end{remark}

\begin{defn}\label{def of HN filtration, decomposition and polygon}
Let $\Vcal$ be a vector bundle on $\adicff$. 
\begin{enumerate}[label = (\arabic*)]
\item We refer to the filtration \eqref{HN filtration}
 in Theorem \ref{existence of HN decomp} 
as the \emph{Harder-Narasimhan (HN) filtration} of $\Vcal$. 
\smallskip

\item We refer to the decomposition \eqref{HN decomposition} 
in Theorem \ref{existence of HN decomp} 
as a \emph{Harder-Narasimhan (HN) decomposition} of $\Vcal$. 
\smallskip

\item We define the \emph{Harder-Narasimhan (HN) polygon} of $\Vcal$ as the upper convex hull of the points $(\rk(\Vcal_i), \deg(\Vcal_i))$ where $\Vcal_i$'s are subbundles in the HN filtration of $\Vcal$. 
\smallskip

\item We refer to the slopes of $\HN(\Vcal)$ as the \emph{Harder-Narasimhan (HN) slopes} of $\Vcal$, or simply the \emph{slopes} of $\Vcal$. These are precisely the numbers $\lambda_i = \mu(\Vcal_i/\Vcal_{i-1})$ in Theorem \ref{existence of HN decomp}. 
\end{enumerate}
\end{defn}

We can restate Theorem \ref{existence of HN decomp} in terms of HN polygons as follows:

\begin{cor}\label{classification by HN polygon}
Every vector bundle $\Vcal$ on $\adicff$ is determined up to isomorphism by its HN polygon $\HN(\Vcal)$. 
\end{cor}

Let us now introduce some notations that we will frequently use. 

\begin{defn}\label{slope_leq_part}
Let $\Vcal$ be a vector bundle on $\adicff$ with Harder-Narasimhan filtration
\[0 = \Vcal_0 \subset \Vcal_1 \subset \ldots \subset \Vcal_m = \Vcal.\] 
\begin{enumerate}[label=(\arabic*)]
\item We write $\mumax(\Vcal)$ (resp. $\mumin(\Vcal)$) for the maximum (resp. minimum) slope of $\Vcal$. 
\smallskip

\item For every $\mu \in \Q$, we define $\Vcal^{\geq \mu}$ (resp. $\Vcal^{> \mu}$) to be the subbundle of $\Vcal$ given by $\Vcal_i$ for the largest value of $i$ such that $\mu(\Vcal_i/\Vcal_{i-1}) \geq \mu$ (resp. such that $\mu(\Vcal_i/\Vcal_{i-1}) > \mu$). We also define $\Vcal^{<\mu} := \Vcal / \Vcal^{\geq \mu}$ and $\Vcal^{\leq \mu} := \Vcal / \Vcal^{> \mu}$.
\end{enumerate}
\end{defn}

\begin{lemma}\label{slope_leq_part via HN decomp}
Let $\Vcal$ be a vector bundle on $\adicff$ with Harder-Narasimhan decomposition
\[\Vcal \simeq \bigoplus_{i=1}^l \trivbundle(\lambda_i)^{\oplus m_i}.\]
Then we have the following identifications:
\[\Vcal^{\geq \mu} \simeq \bigoplus_{\lambda_i \geq \mu}\trivbundle(\lambda_i)^{\oplus m_i} \quad\quad \text{ and } \quad\quad \Vcal^{> \mu} \simeq \bigoplus_{\lambda_i>\mu}\trivbundle(\lambda_i)^{\oplus m_i},\]
\[\Vcal^{\leq \mu} \simeq \bigoplus_{\lambda_i \leq \mu}\trivbundle(\lambda_i)^{\oplus m_i} \quad\quad \text{ and } \quad\quad  \Vcal^{< \mu} \simeq \bigoplus_{\lambda_i<\mu}\trivbundle(\lambda_i)^{\oplus m_i}.\]
\end{lemma}
\begin{proof}
This is an immediate consequence of Definition \ref{slope_leq_part}. 
\end{proof}

\begin{lemma}\label{rank and degree of dual bundle}
Given a vector bundle $\Vcal$ on $\adicff$, we have identities
\[ \rk(\Vcal) = \rk(\Vcal^\vee) \quad\text{ and }\quad \deg(\Vcal) = -\deg(\Vcal^\vee).\]
More generally, for every $\mu \in \Q$ we have identities
\[\rk(\Vcal^{\geq \mu}) = \rk((\Vcal^\vee)^{\leq-\mu}) \quad\text{ and }\quad \deg(\Vcal^{\geq \mu}) = -\deg((\Vcal^\vee)^{\leq-\mu}).\]
\end{lemma}
\begin{proof}
When $\Vcal$ is stable, the first statement is an immediate consequence of Lemma \ref{basic properties of stable bundles}. From this, we deduce the first statement for the general case using HN decomposition of $\Vcal$. The second statement then follows from the first statement since we have $(\Vcal^{\geq \mu})^\vee \simeq (\Vcal^\vee)^{\leq -\mu}$ by Lemma \ref{basic properties of stable bundles} and Lemma \ref{slope_leq_part via HN decomp}. 
\end{proof}

\begin{lemma}\label{zero hom for dominating slopes}
Given two vector bundles $\Vcal$ and $\Wcal$ on $\adicff$, we have
\begin{equation*}\label{zero hom for dominating slopes equation}
\mathrm{Hom}(\Vcal, \Wcal) = 0 \quad\quad \text{ if and only if } \quad\quad \mumin(\Vcal)>\mumax(\Wcal).
\end{equation*}
\end{lemma}
\begin{proof}
It suffices to consider the case when both $\Vcal$ and $\Wcal$ are stable; indeed, the general case will follow from this special case using the HN decompositions of $\Vcal$ and $\Wcal$. Let us now write $\Vcal = \trivbundle(\lambda)$ and $\Wcal = \trivbundle(\mu)$ for some $\lambda, \mu \in \Q$. Then using Lemma \ref{basic properties of stable bundles} we find
\[\mathrm{Hom}(\Vcal, \Wcal) \simeq H^0(\adicff, \Vcal^\vee \otimes \Wcal) \simeq H^0(\adicff, \trivbundle(\lambda)^\vee \otimes \trivbundle(\mu)) \simeq H^0(\adicff, \trivbundle(\mu - \lambda)^{\oplus n})\]
for some $n$. Since the condition $\mumin(\Vcal)>\mumax(\Wcal)$ is equivalent to $\lambda > \mu$, the assertion follows from Theorem \ref{cohomological vanishing of stable bundles}. 
\end{proof}

\section{Moduli of bundle maps}\label{diamond_dim}

\subsection{Definitions and key properties}\label{bundlemaps}$ $

In this section we define certain moduli spaces of bundle maps and collect some of their key properties. We refer the readers to \cite[\S3.3]{Arizona_extvb} for details. 

\begin{defn}\label{def of moduli functors of bundle maps}
Let $\Ecal$ and $\Fcal$ be vector bundles on $\adicff$. We denote by $\Perfd_{/\Spa (\algclosedperfdfield)}$ the category of perfectoid spaces over $\Spa(\algclosedperfdfield)$.
\begin{enumerate}
\item For each $S \in \Perfd_{/\Spa(\algclosedperfdfield)}$, we write $\Ecal_S$ and $\Fcal_S$ for the pullbacks of $\Ecal$ and $\Fcal$ along the natural map $\adicff_S \to \adicff$.
\smallskip

\item $\Hcal^{0}(\Ecal)$ is the functor associating to $S \in \Perfd_{/\Spa(\algclosedperfdfield)}$ the set $H^{0}(\adicff_S,\Ecal_S)$.
\smallskip

\item $\Hom(\mathcal{E},\Fcal)$ is the functor
associating to $S \in \Perfd_{/\Spa(\algclosedperfdfield)}$ the set of $\trivbundle_{\adicff_S}$-module
maps $\Ecal_S\to\Fcal_S$. 
\smallskip

\item $\Surj(\Ecal, \Fcal)$ is the subfunctor of $\Hom(\Ecal,\Fcal)$ whose $S$-points parametrize surjective $\trivbundle_{\adicff_S}$-module maps $\Ecal_S \surj \Fcal_S$. 
\smallskip

\item $\Inj(\Ecal, \Fcal)$
is the subfunctor of $\Hom(\Ecal,\Fcal)$ whose $S$-points parametrize ``fiberwise-injective'' $\trivbundle_{\adicff_S}$-module
maps. Precisely, this functor parametrizes $\trivbundle_{\adicff_S}$-module
maps $\Ecal_{S}\to\Fcal_{S}$ whose pullback along the map $\adicff_{\overline{x}} \to\adicff_S$ for any geometric point $\overline{x}
\to S$
gives an injective $\trivbundle_{\adicff_{\overline{x}}}$-module map. %


\end{enumerate}
\end{defn}

\begin{remark}
The condition defining $\Inj(\Ecal, \Fcal)$ is much stronger than the condition that $\Ecal_{S}\to\Fcal_{S}$ is injective. We impose this enhanced condition to ensure that $\Inj(\Ecal, \Fcal)$ is a sheaf; indeed, the functor associating $S \in \Perfd_{/\Spa(\algclosedperfdfield)}$ to the set of injective $\trivbundle_{\adicff_S}$-module maps $\Ecal_S \in \Fcal_S$ is not even a presheaf. The fiberwise injectivity guarantees that a short exact sequence $0 \to \Ecal_S \to \Fcal_S \to \Fcal_S/\Ecal_S \to 0$ remains exact after tensoring with $\trivbundle_{\adicff_T}$ for any perfectoid space $T$ over $S$ as $\mathrm{Tor}_1(\Fcal_S/\Ecal_S)$ vanishes for all $S \in \Perfd_{/\Spa(\algclosedperfdfield)}$; however, we don't have this property without fiberwise injectivity. 
\end{remark}


Scholze's theory of diamonds in \cite{Scholze_diamonds} provides a framework for making sense of these functors as moduli spaces, thereby allowing us to study their geometric properties. 

\begin{prop}[{\cite[Proposition 3.3.2, Proposition 3.3.5 and Proposition 3.3.6]{Arizona_extvb}}]\label{moduli of bundle maps dimension formulas}
Let $\Ecal$ and $\Fcal$ be vector bundles on $\adicff$. The functors $\Hcal^0(\Ecal)$, $\Hom(\Ecal, \Fcal)$, $\Surj(\Ecal, \Fcal)$ and $\Inj(\Ecal, \Fcal)$ are all locally spatial and partially proper diamonds 
in the sense of Scholze \cite{Scholze_diamonds}. Moreover, their dimensions are given as follows:
\begin{enumerate}[label=(\arabic*)]
\item The diamond $\Hcal^0(\Ecal)$ is equidimensional of dimension $\deg(\Ecal)^\nonneg$. 

\item The diamond $\Hom(\Ecal, \Fcal)$ is equidimensional of dimension $\deg(\Ecal^\vee \otimes \Fcal)^\nonneg$. 

\item The diamonds $\Surj(\Ecal, \Fcal)$ and $\Inj(\Ecal, \Fcal)$ are either empty or equidimensional of dimension $\deg(\Ecal^\vee \otimes \Fcal)^\nonneg$. 
\end{enumerate}
\end{prop}


\begin{remark}
For a diamond $Y$, its dimension refers to the Krull dimension of its topological space $|Y|$. For locally spatial diamonds, this notion of dimension turns out to behave pleasantly well as their topological spaces are locally spectral. 

We also note that, by the work of Le Bras \cite{LeBras_BCspaces}, the functors $\Hcal^0(\Ecal)$ and $\Hom(\Ecal, \Fcal)$ are also Banach-Colmez spaces as defined by Colmez \cite{Colmez_BCspace}. Moreover, their dimension as a diamond is equal to their \emph{principal dimension} as a Banach-Colmez space. 
\end{remark}

\begin{prop}[{\cite[Theorem 3.3.11]{Arizona_extvb}}]\label{dimension inequality for surj maps}
Let $\Ecal$ and $\Fcal$ be vector bundles on $\adicff$ satisfying the following properties:
\begin{enumerate}[label=(\roman*)]
\item\label{existence of nonzero bundle map from E to F} There exists a nonzero bundle map $\Ecal \to \Fcal$. 

\item\label{positive codim for Hom minus surj} For any $\Qcal \subsetneq \Fcal$ which also occurs as a quotient of $\Ecal$ we have an inequality
\[ \deg(\Ecal^\vee \otimes \Qcal)^\nonneg + \deg(\Qcal^\vee \otimes \Fcal)^\nonneg < \deg(\Ecal^\vee \otimes \Fcal)^\nonneg + \deg(\Qcal^\vee \otimes \Qcal)^\nonneg.\]
\end{enumerate}
Then there exists a surjective bundle map $\Ecal \surj \Fcal$. 
\end{prop}

\begin{remark}
Let us provide an interpretation of the properties \ref{existence of nonzero bundle map from E to F} and \ref{positive codim for Hom minus surj} in line with Proposition \ref{moduli of bundle maps dimension formulas}. Let $S$ be the set of isomorphism classes of subsheaves $\Qcal \subsetneq \Fcal$ which also occur as a quotient of $\Ecal$. For each $\Qcal \in S$, composition of bundle maps induces a natural map of diamonds 
\[\Surj(\Ecal,\Qcal) \times_{\Spd\,\algclosedperfdfield} \Inj(\Qcal,\Fcal) \to \Hom(\Ecal,\Fcal).\] 
Let us define $|\Hom(\Ecal, \Fcal)_{\Qcal}| \subset |\Hom(\Ecal, \Fcal)|$ to be the image of the induced map on topological spaces; indeed, $|\Hom(\Ecal, \Fcal)_{\Qcal}|$ is the underlying topological space of a subdiamond $\Hom(\Ecal, \Fcal)_{\Qcal}$ of $\Hom(\Ecal, \Fcal)$, which essentially parametrizes bundle maps $\Ecal \to \Fcal$ with image isomorphic to $\Qcal$ at all geometric points. Then $|\Hom(\Ecal, \Fcal)_{\Qcal}|$ is either empty or satisfies 
\[\dim |\Hom(\Ecal, \Fcal)_{\Qcal}| = \deg(\Ecal^\vee \otimes \Qcal)^\nonneg + \deg(\Qcal^\vee \otimes \Fcal)^\nonneg - \deg(\Qcal^\vee \otimes \Qcal)^\nonneg,\]
as shown in \cite[Lemma 3.3.10]{Arizona_extvb}. Hence the properties \ref{existence of nonzero bundle map from E to F} and \ref{positive codim for Hom minus surj} together imply that $\Hom(\Ecal, \Fcal)$ admits an $F$-point and satisfies
\[ \dim |\Hom(\Ecal, \Fcal)_{\Qcal}| < \dim |\Hom(\Ecal, \Fcal)| \quad \text{ for } \Qcal \in S.\]
\end{remark}

\subsection{Dimension counting by Harder-Narasimhan polygons}\label{Geometric interpretation of degrees}$ $

Our discussion in \S\ref{bundlemaps} suggests that we will have to understand quantities of the form $\deg(\Vcal^\vee \otimes \Wcal)^\nonneg$ for fairly arbitrary vector  bundles $\Vcal$ and $\Wcal$ on $\adicff$. In this subsection, we prove some useful lemmas for this purpose following the strategy developed in \cite[\S2.3]{Arizona_extvb}.

    
\begin{defn}\label{order_relation} Let $v$ and $w$ be arbitrary vectors in $\R^2$. 
\begin{enumerate}[label = (\arabic*)]
\item We denote by $v_x$ (resp. $v_y$) the $x$-coordinate (resp. $y$-coordinate) of $v$. 
\smallskip

\item If $v_x \neq 0$, we write $\mu(v):= v_y/v_x$ for the slope of $v$. 
\smallskip

\item If both $v$ and $w$ have nonzero $x$-coordinates, we will often write $v \prec w$ (resp. $v \preceq w$) in lieu of $\mu(v) < \mu(w)$ (resp. $\mu(v) \leq \mu(w)$). 
\smallskip

\item We write $v \times w$ for the (two-dimensional) cross product of $v$ and $w$, regarded as a scalar by the formula $v \times w = v_x w_y - v_y w_x.$
\end{enumerate}
\end{defn}
	
In this paper, we are exclusively interested in vectors with a positive $x$-coordinate. For such vectors, we can characterize the relations $\preceq$ and $\prec$ in terms of the two dimensional cross product as follows:

\begin{lemma}\label{relation_area} 
Let $v$ and $w$ be vectors in $\R^2$ with $v_x, w_x >0$. Then we have $v \prec w$ (resp. $v \preceq w$) if and only if $v \times w >0$ (resp. $v \times w \geq 0$). 
\end{lemma}
	
\begin{proof}
This is straightforward to check using Definition \ref{order_relation}. 
\end{proof}

We will make use of Lemma \ref{relation_area} by expressing HN polygons in terms of vectors. 
\begin{defn} \label{HN vector notation}

		Let $\Vcal$ be a vector bundle on $\adicff$ with Harder-Narasimhan decomposition
	\[ \Vcal = \bigoplus_{i=1}^l \trivbundle(\lambda_i)^{m_i}\]
	where $\lambda_1 > \lambda_2 > \cdots > \lambda_l$. We define the \emph{HN vectors of $\Vcal$} by
\[\HNvec(\Vcal) := (v_i)_{1\leq i \leq l}\]
where $v_i$ is the vector representing the $i$-th line segment in $\HN(\Vcal)$; more precisely, writing $\lambda_i = r_i/s_i$ in lowest terms with $s_i>0$, we set $v_i :=(m_i s_i, m_i r_i)$.  
\end{defn}

The following simple lemma is pivotal to our discussion in this section. 

	\begin{lemma}\label{diamond lemma for nonnegative degree}\label{degree in terms of HN vectors}
Let $\Vcal$ and $\Wcal$ be vector bundles on $\adicff$ with $\HNvec(\Vcal) = (v_i)$ and $\HNvec(\Wcal) = (w_j)$. Then we have
\begin{equation*}\label{degree formula in terms of HN vectors}
\deg(\Vcal^\vee \otimes \Wcal) = \sum_{i, j} v_i \times w_j \quad\quad \text{ and } \quad\quad \deg(\Vcal^\vee \otimes \Wcal)^\nonneg = \sum_{v_i \preceq w_j} v_i \times w_j
\end{equation*}
	\end{lemma}

	\begin{proof}
When $\Vcal$ and $\Wcal$ are both semistable, we quickly verify both identities in \eqref{degree formula in terms of HN vectors} using Lemma \ref{basic properties of stable bundles} and Lemma \ref{relation_area}.
Then we deduce the general case using the HN decompositions of $\Vcal$ and $\Wcal$.
\end{proof}
		

\begin{cor}\label{zero degree for completely dominating slopes}
For arbitrary vector bundles $\Vcal$ and $\Wcal$ on $\adicff$, we have
\[\dim \Hom(\Vcal, \Wcal) = 0 \quad\quad \text{ if and only if } \quad\quad \mumin(\Vcal) \geq \mumax(\Wcal). \]
\end{cor}
\begin{proof}
This is an immediate consequence of Proposition \ref{moduli of bundle maps dimension formulas} and Lemma \ref{degree in terms of HN vectors}
\end{proof}

\begin{remark}
This is not a consequence of Lemma \ref{zero hom for dominating slopes}; in fact, 
When $\mumin(\Vcal) = \mumax(\Wcal)$, Lemma \ref{zero hom for dominating slopes} and Corollary \ref{zero degree for completely dominating slopes} respectively yield $\mathrm{Hom}(\Vcal, \Wcal) \neq 0$ and $\dim \Hom(\Vcal, \Wcal) = 0$. 
\end{remark}

\begin{defn}
Given a vector bundle $\Vcal$ on $\adicff$, we write $\Vcal(\lambda) := \Vcal \otimes \trivbundle(\lambda)$ for any $\lambda \in \Q$. 
\end{defn}

\begin{lemma}\label{degree after shear}
Let $\Vcal$ and $\Wcal$ be vector bundles on $\adicff$. For any $\lambda \in \Q$ we have
\[\deg(\Vcal(\lambda)^\vee \otimes \Wcal(\lambda))^\nonneg = \rk(\trivbundle(\lambda))^2 \cdot \deg(\Vcal^\vee \otimes \Wcal)^\nonneg.\]
\end{lemma}

\begin{proof}
By Lemma \ref{basic properties of stable bundles} we find
\[ \Vcal(\lambda)^\vee \otimes \Wcal(\lambda) \simeq \Vcal^\vee \otimes \Wcal \otimes \trivbundle(\lambda) \otimes \trivbundle(\lambda)^\vee \simeq \Vcal^\vee \otimes \Wcal \otimes \trivbundle^{\rk(\trivbundle(\lambda))^2}\]
and consequently obtain the desired assertion. 
\end{proof}

\begin{lemma}\label{degree after stretch}
Let $\Vcal$ and $\Wcal$ be vector bundles on $\adicff$. Take $\vertstretch{\Vcal}$ and $\vertstretch{\Wcal}$ to be vector bundles on $\adicff$ whose HN polygons are obtained by vertically stretching $\HN(\Vcal)$ and $\HN(\Wcal)$ by a positive integer factor $C$. Then we have
\[ \deg(\vertstretch{\Vcal}^\vee \otimes \vertstretch{\Wcal})^\nonneg = C \cdot \deg(\Vcal^\vee \otimes \Wcal)^\nonneg.\]
\end{lemma}
\begin{proof}
Let us consider the HN vectors 
\[\HNvec(\Vcal) = (v_i), \quad\quad \HNvec(\Wcal) = (w_j), \quad\quad \HNvec(\vertstretch{\Vcal}) = (\vertstretch{v}_i), \quad\quad \HNvec(\vertstretch{\Wcal}) = (\vertstretch{w}_j).\] 
By construction, we have the following relations between these HN vectors. 
%
\[ \vertstretch{v}_{i, x} = v_{i, x}, \quad\quad \vertstretch{w}_{j, x} = w_{j, x}, \quad\quad \vertstretch{v}_{i, y} = C v_{i, y}, \quad\quad \vertstretch{w}_{j, y} = C w_{j, y}.\]
Now for each $i$ and $j$ we have
\begin{equation}\label{cross product of HN vectors after stretch}
\begin{aligned}
\vertstretch{v}_i \times \vertstretch{w}_j &= \vertstretch{v}_{i, x} \vertstretch{w}_{j, y} - \vertstretch{v}_{i, y} \vertstretch{w}_{j, x}
= C (v_{i, x} w_{j, y} - v_{i, y} w_{j, x}) = C \cdot (v_i \times w_j)
\end{aligned}
\end{equation}
We thus use Lemma \ref{relation_area}, Lemma \ref{degree in terms of HN vectors} and \eqref{cross product of HN vectors after stretch} to find
\begin{align*}
\deg(\vertstretch{\Vcal}^\vee \otimes \vertstretch{\Wcal})^\nonneg &= \sum_{\vertstretch{v}_i \preceq \vertstretch{w}_j} \vertstretch{v}_i \times \vertstretch{w}_j = \sum_{v_i \preceq w_j} C (v_i \times w_j) \\
&= C \sum_{v_i \preceq w_j} v_i \times w_j = C \cdot \deg(\Vcal^\vee \otimes \Wcal)^\nonneg,
\end{align*}
completing the proof. 
\end{proof}

\begin{remark}
The equation \eqref{cross product of HN vectors after stretch} represents the fact that the vertical stretch by a factor $C$ scales the area of an arbitrary parallelogram by the same factor. 
\end{remark}

\section{Classification of quotient bundles}
  
\subsection{The main statement and its consequences}\label{statement and consequences}$ $

Let us state our main theorem, which gives a complete classification of all quotient bundles of a given vector bundle on $\adicff$. 

\begin{theorem}\label{classification of quotient bundles}
Let $\Ecal$ be a vector bundle on $\adicff$. Then a vector bundle $\Fcal$ on $\adicff$ is a quotient bundle of $\Ecal$ if and only if the following equivalent conditions are satisfied:
\begin{enumerate}[label=(\roman*)]
\item\label{rank inequalities for quotients} For every $\mu \in \Q$, we have $\rk(\Ecal^{\leq \mu}) \geq \rk(\Fcal^{\leq \mu})$ with equality if and only if $\Ecal^{\leq \mu}$ and $\Fcal^{\leq \mu}$ are isomorphic.
\smallskip

\item\label{dual strong slopewise dominance for quotients} If we align $\HN(\Ecal)$ and $\HN(\Fcal)$ so that their right endpoints lie at the origin, then for each $i = 1, \cdots, \rk(\Fcal)$, the slope of $\HN(\Fcal)$ on $[-i, -i+1]$ is greater than or equal to the slope of $\HN(\Ecal)$ on $[-i-1, -i]$ unless $\HN(\Ecal)$ and $\HN(\Fcal)$ agree on $[-i, 0]$. 
\end{enumerate}

\smallskip
\begin{figure}[H]
\begin{tikzpicture}[scale=1]	

		\coordinate (right) at (0, 0);
		\coordinate (q0) at (-1,2);
		\coordinate (q1) at (-2.5, 3.4);
		\coordinate (q2) at (-6, 4.8);
		\coordinate (q3) at (-9, 4);
		

		\coordinate (p0) at (-2, 1.5);
		\coordinate (p1) at (-4.5, 2);
		\coordinate (p2) at (-6, 1.3);
		\coordinate (p3) at (-7, 0.1);
				
		\draw[step=1cm,thick] (right) -- (q0) --  (q1) -- (q2) -- (q3);
		\draw[step=1cm,thick] (right) -- (p0) --  (p1) -- (p2) -- (p3);
		
		\draw [fill] (q0) circle [radius=0.05];		
		\draw [fill] (q1) circle [radius=0.05];		
		\draw [fill] (q2) circle [radius=0.05];		
		\draw [fill] (q3) circle [radius=0.05];		
		\draw [fill] (right) circle [radius=0.05];
		
		\draw [fill] (p0) circle [radius=0.05];		
		\draw [fill] (p1) circle [radius=0.05];		
		\draw [fill] (p2) circle [radius=0.05];		
		\draw [fill] (p3) circle [radius=0.05];		
		
		\draw[step=1cm,dotted] (-6.5, -0.4) -- (-6.5, 5);
       		\draw[step=1cm,dotted] (-5.5, -0.4) -- (-5.5, 5);
       		\draw[step=1cm,dotted] (-6, -0.4) -- (-6, 5);

		\node at (-6.8,-0.8) {\scriptsize $-i-1$};
		\node at (-5.2,-0.8) {\scriptsize $-i+1$};
		\node at (-6,-0.8) {\scriptsize $-i$};
		
		\path (q3) ++(-0.8, 0.05) node {$\HN(\Ecal)$};
		\path (p3) ++(-0.8, 0.05) node {$\HN(\Fcal)$};
		\path (right) ++(0.3, -0.05) node {$O$};

\end{tikzpicture}
\caption{Illustration of the condition \ref{dual strong slopewise dominance for quotients} in Theorem \ref{classification of quotient bundles}.}
\end{figure}
\end{theorem}
We will discuss our proof of Theorem \ref{classification of quotient bundles} in the subsequent subsections. 
In this subsection we explain some classification results as consequences of Theorem \ref{classification of quotient bundles}. 



Our first corollary of Theorem \ref{classification of quotient bundles} dualizes the statement of Theorem \ref{classification of quotient bundles} to classify \emph{almost} all subsheaves of a given vector bundle on $\adicff$.

\begin{cor}\label{almost classification of subbundles}
Let $\Ecal$ be a vector bundle on $\adicff$. Then a vector bundle $\Dcal$ on $\adicff$ is 
a subsheaf of $\Ecal$ if the following equivalent conditions are satisfied:
\begin{enumerate}[label=(\roman*)]
\item\label{rank inequalities for subbundles} For every $\mu \in \Q$, we have $\rk(\Ecal^{\geq \mu}) \geq \rk(\Dcal^{\geq \mu})$ with equality if and only if $\Ecal^{\geq \mu}$ and $\Dcal^{\geq \mu}$ are isomorphic. 

\smallskip




\item\label{strong slopewise dominance for subbundles} If we align $\HN(\Dcal)$ and $\HN(\Ecal)$ so that their left endpoints lie at the origin, then for each $i = 1, \cdots, \rk(\Dcal)$, the slope of $\HN(\Dcal)$ on $[i-1, i]$ is less than or equal to the slope of $\HN(\Ecal)$ on $[i, i+1]$ unless $\HN(\Dcal)$ and $\HN(\Ecal)$ agree on $[0, i]$. 
\end{enumerate}

\smallskip
\begin{figure}[H]
\begin{tikzpicture}[scale=1]

		\coordinate (left) at (0, 0);
		\coordinate (q0) at (1,2);
		\coordinate (q1) at (2.5, 3.4);
		\coordinate (q2) at (6, 4.8);
		\coordinate (q3) at (9, 4);
		

		\coordinate (p0) at (2, 1.5);
		\coordinate (p1) at (4.5, 2);
		\coordinate (p2) at (6, 1.3);
		\coordinate (p3) at (7, 0.1);
				
		\draw[step=1cm,thick] (left) -- (q0) --  (q1) -- (q2) -- (q3);
		\draw[step=1cm,thick] (left) -- (p0) --  (p1) -- (p2) -- (p3);
		
		\draw [fill] (q0) circle [radius=0.05];		
		\draw [fill] (q1) circle [radius=0.05];		
		\draw [fill] (q2) circle [radius=0.05];		
		\draw [fill] (q3) circle [radius=0.05];		
		\draw [fill] (left) circle [radius=0.05];
		
		\draw [fill] (p0) circle [radius=0.05];		
		\draw [fill] (p1) circle [radius=0.05];		
		\draw [fill] (p2) circle [radius=0.05];		
		\draw [fill] (p3) circle [radius=0.05];		
		
		\draw[step=1cm,dotted] (5.5, -0.4) -- (5.5, 5);
       		\draw[step=1cm,dotted] (6.5, -0.4) -- (6.5, 5);
       		\draw[step=1cm,dotted] (6, -0.4) -- (6, 5);

		\node at (5.2,-0.8) {\scriptsize $i-1$};
		\node at (6.8,-0.8) {\scriptsize $i+1$};
		\node at (6,-0.8) {\scriptsize $i$};
		
		\path (q3) ++(0.8, 0.05) node {$\HN(\Ecal)$};
		\path (p3) ++(0.8, 0.05) node {$\HN(\Dcal)$};
		\path (right) ++(-0.3, -0.05) node {$O$};

\end{tikzpicture}
\caption{Illustration of the condition \ref{strong slopewise dominance for subbundles} in Corollary \ref{almost classification of subbundles}.}
\end{figure}
\end{cor}

\begin{proof}
Let us first to show that there exists an injective bundle map $\Dcal \inj \Ecal$ if $\Dcal$ satisfies the condition \ref{rank inequalities for subbundles}. By means of dualizing, it suffices to show that there exists a surjective bundle map $\Ecal^\vee \surj \Dcal^\vee$, or equivalently that $\Dcal^\vee$ is a quotient bundle of $\Ecal^\vee$. This follows from Theorem \ref{classification of quotient bundles} since by Lemma \ref{rank and degree of dual bundle} we can rewrite the condition \ref{rank inequalities for subbundles} as follows:
\begin{enumerate}[label=(\roman*)$^\vee$]
\item\label{rank inequalities for dual quotients} For every $\mu \in \Q$, we haev $\rk((\Ecal^\vee)^{\leq -\mu}) \geq \rk((\Dcal^\vee)^{\leq -\mu})$ with equality if and only if $(\Ecal^\vee)^{\leq -\mu}$ and $(\Dcal^\vee)^{\leq -\mu}$ are isomorphic.
\end{enumerate}

It remains to verify that the conditions \ref{rank inequalities for subbundles} and \ref{strong slopewise dominance for subbundles} are equivalent. By reflecting the HN polygons $\HN(\Dcal)$ and $\HN(\Ecal)$ about the $y$-axis, we obtain the HN polygons $\HN(\Dcal^\vee)$ and $\HN(\Ecal^\vee)$ with their right endpoints at the origin. We thus find that the condition \ref{strong slopewise dominance for subbundles} is equivalent to the following condition:
\begin{enumerate}[label=(\roman*)$^\vee$, start=2]
\item\label{dual slopewise dominance for dual quotients} For each $i = 1, \cdots, \rk(\Dcal^\vee)$, the slope of $\HN(\Dcal^\vee)$ on the interval $[-i, -i+1]$ is greater than or equal to the slope of $\HN(\Ecal^\vee)$ on $[-i-1, -i]$ unless $\HN(\Dcal^\vee)$ and $\HN(\Ecal^\vee)$ agree on $[-i, 0]$. 
\end{enumerate}
Moreover, the conditions \ref{rank inequalities for dual quotients} and \ref{dual slopewise dominance for dual quotients} are equivalent by Theorem \ref{classification of quotient bundles}. Hence we obtain the equivalence between the conditions \ref{rank inequalities for subbundles} and \ref{strong slopewise dominance for subbundles} as desired, thereby completing the proof. 
\end{proof}

\begin{remark} Corollary \ref{almost classification of subbundles} does not give a complete classification of subsheaves since the equality part of the condition \ref{rank inequalities for subbundles} is not necessary. The main underlying issue is that the cokernel of an injective bundle map is not necessarily a vector bundle. It turns out that the inequality in the condition \ref{rank inequalities for subbundles} alone gives a complete classification of subsheaves of $\Ecal$, as proved by the author in the sequel paper \cite{Hong_subvb}.
\end{remark}




As another consequence of Theorem \ref{classification of quotient bundles}, we have a complete classification of finitely globally generated vector bundles on $\adicff$. 

\begin{cor}
A vector bundle $\Ecal$ on $\adicff$ is generated by $n$ global sections if and only if the following conditions are satisfied:
\begin{enumerate}[label=(\roman*)]
\item\label{nonpositivity of slopes for globally generated bundles} $\HN(\Ecal)$ has only nonnegative slopes, i.e., $\Ecal^{<0} = 0$. 
\smallskip

\item\label{rank bound for globally generated bundles} We have $\rk(\Ecal) \leq n$ with equality if and only if $\Ecal$ is trivial (i.e., isomorphic to $\trivbundle^{\oplus n}$). 
\end{enumerate}
\end{cor}

\begin{proof}
A vector bundle $\Ecal$ on $\adicff$ is generated by $n$ global sections if and only if there is a surjective bundle map $\trivbundle^{\oplus n} \surj \Ecal$, 
which amounts to saying that $\Ecal$ is a quotient bundle of $\trivbundle^{\oplus n}$. By Theorem \ref{classification of quotient bundles}, this precisely means that we have $\rk((\trivbundle^{\oplus n})^{\leq \mu}) \geq \rk(\Ecal^{\leq \mu})$ for every $\mu \in \Q$ with equality if and only if $(\trivbundle^{\oplus n})^{\leq \mu}$ and  $\Ecal^{\leq \mu}$ are isomorphic. Since $(\trivbundle^{\oplus n})^{\leq \mu}$ is equal to $\trivbundle^{\oplus n}$ for $\mu \geq 0$ and zero for $\mu < 0$, 
we find that $\Ecal$ on $\adicff$ is generated by $n$ global sections if and only if the following conditions are satisfied:

\begin{enumerate}[label=(\roman*)']
\item\label{rank inequalities for globally generated bundles} For every $\mu < 0$, we have $\Ecal^{\leq \mu} = 0$. 

\smallskip

\item\label{equal rank condition for globally generated bundles} For every $\mu \geq 0$, we have $n \geq \rk(\Ecal^{\leq \mu})$ with equality if and only if $\Ecal^{\leq \mu}$ is trivial. 
\end{enumerate}
The conditions \ref{nonpositivity of slopes for globally generated bundles} and \ref{rank inequalities for globally generated bundles} are evidently equivalent. 
In addition, the condition \ref{equal rank condition for globally generated bundles} implies the condition \ref{rank bound for globally generated bundles} as we have $\Ecal^{\leq \mumax(\Ecal)} = \Ecal$. Hence it remains to prove that the condition \ref{rank bound for globally generated bundles} implies the condition \ref{equal rank condition for globally generated bundles}. 



The inequality in the condition \ref{rank bound for globally generated bundles} implies the inequality in the condition \ref{equal rank condition for globally generated bundles} as we have $\rk(\Ecal^{\leq \mu}) \leq \rk(\Ecal)$ for every $\mu \in \Q$. Therefore we only need to consider the equality part in the condition \ref{equal rank condition for globally generated bundles}. Let us now suppose that we have $\rk(\Ecal^{\leq \mu}) = n$ for some $\mu \in \Q$. 
The condition \ref{rank bound for globally generated bundles} yields $\rk(\Ecal^{\leq \mu}) \leq \rk(\Ecal) \leq n$, which in turn implies $\rk(\Ecal^{\leq \mu}) = \rk(\Ecal) = n$. The first equality means that $\Ecal^{\leq \mu}$ is isomorphic to $\Ecal$, while the second equality implies that $\Ecal$ is trivial by the condition \ref{rank bound for globally generated bundles}. We thus deduce that $\Ecal^{\leq \mu}$ is trivial as desired. 
\end{proof}

\subsection{Slopewise dominance of vector bundles}\label{slopewise dominance}$ $

The rest of this paper will be devoted to proving Theorem \ref{classification of quotient bundles}. 
In this section, we introduce and study the notion of \emph{slopewise dominance} which will be crucial for our proof. 


\begin{defn}\label{def of slopewise dominance}
Let $\Vcal$ and $\Wcal$ be vector bundles on $\adicff$. Assume that their HN polygons $\HN(\Vcal)$ and $\HN(\Wcal)$ are aligned as usual so that their \textbf{left endpoints} lie at the origin. We say that $\Vcal$ \emph{slopewise dominates} $\Wcal$ if for $i = 1, \cdots, \rk(\Wcal)$, the slope of $\HN(\Wcal)$ on the interval $[i-1, i]$ is less than or equal to the slope of $\HN(\Vcal)$ on this interval. 
\begin{figure}[H]
\begin{tikzpicture}	

		\coordinate (left) at (0, 0);
		\coordinate (q0) at (1,2);
		\coordinate (q1) at (2.5, 3.4);
		\coordinate (q2) at (6, 4.8);
		\coordinate (q3) at (9, 4);
		

		\coordinate (p0) at (2, 1.5);
		\coordinate (p1) at (4.5, 2);
		\coordinate (p2) at (7.5, 0.7);
				
		\draw[step=1cm,thick] (left) -- (q0) --  (q1) -- (q2) -- (q3);
		\draw[step=1cm,thick] (left) -- (p0) --  (p1) -- (p2);
		
		\draw [fill] (q0) circle [radius=0.05];		
		\draw [fill] (q1) circle [radius=0.05];		
		\draw [fill] (q2) circle [radius=0.05];		
		\draw [fill] (q3) circle [radius=0.05];		
		\draw [fill] (left) circle [radius=0.05];
		
		\draw [fill] (p0) circle [radius=0.05];		
		\draw [fill] (p1) circle [radius=0.05];		
		\draw [fill] (p2) circle [radius=0.05];		
		
		\draw[step=1cm,dotted] (3.5, -0.4) -- (3.5, 4.8);
       		\draw[step=1cm,dotted] (4, -0.4) -- (4, 4.8);

		\node at (3.4,-0.8) {\scriptsize $i-1$};
		\node at (4.1,-0.8) {\scriptsize $i$};
		
		\path (q3) ++(0.8, 0.05) node {$\HN(\Vcal)$};
		\path (p2) ++(0.8, 0.05) node {$\HN(\Wcal)$};
		\path (left) ++(-0.3, -0.05) node {$O$};

\end{tikzpicture}
\caption{Illustration of Definition \ref{def of slopewise dominance}}\label{slopes_cond_fig}
\end{figure}
\end{defn}

\begin{remark}
Slopewise dominance of $\Vcal$ on $\Wcal$ implies that $\rk(\Vcal) \geq \rk(\Wcal)$. 
\end{remark}



The notion of slopewise dominance 
gives us a characterization of the condition \ref{rank inequalities for quotients} in Theorem \ref{classification of quotient bundles} (and the condition \ref{rank inequalities for subbundles} in Corollary \ref{almost classification of subbundles}). 

\begin{lemma}[cf. \cite{Arizona_extvb} Corollary 4.2.2]\label{slopewise dominance and rank inequalities}
Let $\Vcal$ and $\Wcal$ be vector bundles on $\adicff$. 
\begin{enumerate}[label=(\arabic*)]
\item\label{subbundle rank inequalities and slopewise dominance} $\Vcal$ slopewise dominates $\Wcal$ if and only if $\rk(\Vcal^{\geq \mu}) \geq \rk(\Wcal^{\geq \mu})$ for every $\mu \in \Q$. 

\smallskip

\item\label{quotient rank inequantlies and dual slopewise dominance} $\Vcal^\vee$ slopewise dominates $\Wcal^\vee$ if and only if $\rk(\Vcal^{\leq \mu}) \geq \rk(\Wcal^{\leq \mu})$ for every $\mu \in \Q$. 
\end{enumerate}
\end{lemma}
\begin{proof}
We first note that the statement \ref{quotient rank inequantlies and dual slopewise dominance} follows from the statement \ref{subbundle rank inequalities and slopewise dominance} as a dual statement. In fact, by Lemma \ref{rank and degree of dual bundle} we can rewrite the inequality $\rk(\Vcal^{\leq \mu}) \geq \rk(\Wcal^{\leq \mu})$ as $\rk((\Vcal^\vee)^{\geq -\mu}) \geq \rk((\Wcal^\vee)^{\geq -\mu})$. Hence we only need to prove the statement \ref{subbundle rank inequalities and slopewise dominance}.

We now assume the inequality $\rk(\Vcal^{\geq \mu}) \geq \rk(\Wcal^{\geq \mu})$ for every $\mu \in \Q$ and assert that $\Vcal$ slopewise dominates $\Wcal$. 
For each $i = 1, \cdots, \rk(\Wcal)$, we let $\mu_i$ be the slope of $\HN(\Wcal)$ on the interval $[i-1, i]$. If some $\mu_i$ is greater than the slope of $\HN(\Vcal)$ on $[i-1, i]$, convexity of HN polygons yields $\rk(\Vcal^{\geq \mu_i}) < i \leq \rk(\Wcal^{\geq \mu_i})$ which contradicts the inequality we assumed. 
We thus deduce that $\Vcal$ slopewise dominates $\Wcal$ as desired.

Conversely, we claim the inequality $\rk(\Vcal^{\geq \mu}) \geq \rk(\Wcal^{\geq \mu})$ for every $\mu \in \Q$ assuming that $\Vcal$ slopewise dominates $\Wcal$. Suppose for contradiction that $\rk(\Vcal^{\geq \mu}) < \rk(\Wcal^{\geq \mu})$ for some $\mu$. Then for $i = \rk(\Wcal^{\geq \mu})$, the slope of $\HN(\Wcal)$ on the interval $[i-1, i]$ is at least $\mu$ whereas the slope of $\HN(\Vcal)$ on this interval is less than $\mu$. In particular, the slope of $\HN(\Wcal)$ on $[i-1, i]$ is greater than the slope of $\HN(\Vcal)$ on this interval, yielding a desired contradiction. 
\end{proof}


The notion of slopewise dominance also yields an interesting inequality on degrees which will be useful to us. 

\begin{lemma}\label{nonnegative degree for slopewise dominant pairs}
Let $\Vcal$ and $\Wcal$ be vector bundles on $\adicff$ such that $\Vcal$ slopewise dominates $\Wcal$. We have an inequality
\[\deg(\Vcal)^\nonneg \geq \deg(\Wcal)^\nonneg.\]
\end{lemma}

\begin{proof}
We align $\HN(\Vcal)$ and $\HN(\Wcal)$ as in Definition \ref{def of slopewise dominance} so that their left endpoints lie at the origin. 
\begin{figure}[H]
\begin{tikzpicture}	

		\coordinate (left) at (0, 0);
		\coordinate (q0) at (1,2);
		\coordinate (q1) at (2.5, 3.4);
		\coordinate (q2) at (6, 4.8);
		\coordinate (q3) at (9, 4);
		

		\coordinate (p0) at (2, 1.5);
		\coordinate (p1) at (4.5, 2);
		\coordinate (p2) at (7.5, 0.7);

		\coordinate (aux) at (4.5, 4.2);
				
		\draw[step=1cm,thick] (left) -- (q0) --  (q1) -- (q2) -- (q3);
		\draw[step=1cm,thick] (left) -- (p0) --  (p1) -- (p2);
		
		\draw [fill] (q0) circle [radius=0.05];		
		\draw [fill] (q1) circle [radius=0.05];		
		\draw [fill] (q2) circle [radius=0.05];		
		\draw [fill] (q3) circle [radius=0.05];		
		\draw [fill] (left) circle [radius=0.05];
		
		\draw [fill] (p0) circle [radius=0.05];		
		\draw [fill] (p1) circle [radius=0.05];		
		\draw [fill] (p2) circle [radius=0.05];		

		
		\draw[step=1cm,dotted] (4.5, -0.4) -- (4.5, 4.2);
       		\draw[step=1cm,dotted] (6, -0.4) -- (6, 4.8);

		\node at (4.5,-0.8) {\scriptsize $\rk(\Wcal^\nonneg)$};
		\node at (6,-0.8) {\scriptsize $\rk(\Vcal^\nonneg)$};
		
		\path (q3) ++(0.8, 0.05) node {$\HN(\Vcal)$};
		\path (p2) ++(0.8, 0.05) node {$\HN(\Wcal)$};
		\path (left) ++(-0.3, -0.05) node {$O$};

		\path (p1) ++(-0.05, 0.25) node {\scriptsize $(\rk(\Wcal^\nonneg), \deg(\Wcal^\nonneg)$};
		\path (q2) ++(0.3, 0.3) node {\scriptsize $(\rk(\Vcal^\nonneg), \deg(\Vcal^\nonneg)$};

		\path (aux) ++(-0.5, 0.25) node {\scriptsize $(\rk(\Wcal^\nonneg), d)$};

\end{tikzpicture}
\caption{Comparison of nonnegative parts using slopewise dominance}
\end{figure}
We denote by $d$ the $y$-value of $\HN(\Vcal)$ at $\rk(\Wcal^\nonneg)$. Since $\HN(\Vcal)$ lies above $\HN(\Wcal)$ by slopewise dominance,  
we compare the $y$-values of $\HN(\Vcal)$ and $\HN(\Wcal)$ at $\rk(\Wcal^\nonneg)$ and obtain
\[\deg(\Wcal^\nonneg) \leq d.\]
Moreover, we observe that the $y$-value of $\HN(\Vcal)$ increases on the interval $[0, \rk(\Vcal^\nonneg)]$. Since $\rk(\Wcal^\nonneg) \leq \rk(\Vcal^\nonneg)$ by Lemma \ref{slopewise dominance and rank inequalities}, we compare the $y$-values of $\HN(\Vcal)$ at $\rk(\Wcal^\nonneg)$ and $\rk(\Vcal^\nonneg)$ to find
\[ d \leq \deg(\Vcal^\nonneg).\]
We thus combine the two inequalities to obtain the desired inequality. 
\end{proof}

\begin{remark}
By the same argument we can prove the inequality $\deg(\Vcal)^{\geq \mu} \geq \deg(\Wcal)^{\geq \mu}$ for all $\mu >0$. However, this inequality does not necessarily hold for $\mu <0$. In fact, when $\mu$ is sufficiently small the inequality is equivalent to $\deg(\Vcal) \geq \deg(\Wcal)$, which doesn't necessarily hold as shown by $\Vcal = \trivbundle(1)^{\oplus 2} \oplus \trivbundle(-2)$ and $\Wcal = \trivbundle(1/2)$.

\end{remark}

A number of our reduction arguments will use the following decomposition lemma regarding slopewise dominance. 

\begin{lemma}\label{existence of maximal common factor decomp}
Let $\Vcal$ and $\Wcal$ be vector bundles on $\adicff$ such that $\Vcal$ slopewise dominates $\Wcal$. Then we have decompositions
\begin{equation}\label{max common factor decomp}
 \Vcal \simeq \Ucal \oplus \Vcal' \quad\quad \text{ and } \quad\quad \Wcal \simeq \Ucal \oplus \Wcal'
\end{equation}
satisfying the following properties:
\begin{enumerate}[label=(\roman*)]
\item\label{slopewise dominance for complement part} $\Vcal'$ slopewise dominates $\Wcal'$. 
\smallskip

\item\label{inequality for max slopes of complement parts} If $\Wcal' \neq 0$, we have $\mumax(\Vcal')>\mumax(\Wcal')$. 
\smallskip

\item\label{ineqaulities for min slope of common factor} If $\Ucal \neq 0$ and $\Wcal' \neq 0$, we have $\mumin(\Ucal) \geq \mumax(\Vcal') > \mumax(\Wcal')$.
\end{enumerate}
\end{lemma}

\begin{proof}
We assume that $\HN(\Vcal)$ and $\HN(\Wcal)$ are aligned as in Definition \ref{def of slopewise dominance}. For each $x \in [0, \rk(\Wcal)]$, we define $d(x)$ to be the vertical distance between $\HN(\Vcal)$ and $\HN(\Wcal)$ at $x$. Note that $d(x)$ is nonnegative and increasing by slopewise dominance of $\Vcal$ on $\Wcal$. 

Let us take the maximum $r$ with $d(r) = 0$. The interval $[0, r]$ corresponds to the common part of $\HN(\Vcal)$ and $\HN(\Wcal)$. Moreover, unless we have $r = \rk(\Wcal)$ the polygon $\HN(\Wcal)$ changes its slope at $r$ so that $d(x)$ becomes positive after this point. Hence $r$ must be an integer. 

We take $\Ucal$ to be the vector bundle on $\adicff$ whose HN polygon is given by the common part of $\HN(\Vcal)$ and $\HN(\Wcal)$, as illustrated by the red polygon in the figure below. We also take $\Vcal'$ and $\Wcal'$ to be vector bundles on $\adicff$ whose HN polygons are given by the complement subpolygons of $\HN(\Vcal)$ and $\HN(\Wcal)$, as illustrated by the blue and green polygons in the figure below. Note that these definitions are valid since $r$ is an integer. 
\begin{figure}[H]
\begin{tikzpicture}	

		\coordinate (left) at (0, 0);
		\coordinate (q0) at (1,2.5);
		\coordinate (q1) at (4, 4.5);
		\coordinate (q2) at (6.5, 4.5);
		\coordinate (q3) at (9, 2.5);
		

		\coordinate (p0) at (q0);
		\coordinate (p1) at (2.5, 3.5);
		\coordinate (p2) at (5.5, 3);
		\coordinate (p3) at (7.5, 0.5);
				
		\draw[step=1cm,thick, color=red] (left) -- (p0) --  (p1);
		\draw[step=1cm,thick, color=blue] (p1) -- (q1) -- (q2) -- (q3);
		\draw[step=1cm,thick, color=green] (p1) -- (p2) -- (p3);
		
		\draw [fill] (q0) circle [radius=0.05];		
		\draw [fill] (q1) circle [radius=0.05];		
		\draw [fill] (q2) circle [radius=0.05];		
		\draw [fill] (q3) circle [radius=0.05];		
		\draw [fill] (left) circle [radius=0.05];
		
		\draw [fill] (p0) circle [radius=0.05];		
		\draw [fill] (p1) circle [radius=0.05];		
		\draw [fill] (p2) circle [radius=0.05];		
		\draw [fill] (p3) circle [radius=0.05];		
		
		\draw[step=1cm,dotted] (2.5, -0.4) -- (2.5, 3.5);

		\node at (2.5,-0.8) {\scriptsize $r$};
		
		\path (q3) ++(0.8, 0.05) node {$\HN(\Vcal)$};
		\path (p3) ++(0.8, 0.05) node {$\HN(\Wcal)$};
		\path (left) ++(-0.3, -0.05) node {$O$};

		\path (p0) ++(-0.5, -0.6) node {\color{red}$\Ucal$};
		\path (q2) ++(1.5, -0.7) node {\color{blue}$\Vcal'$};
		\path (p2) ++(1, -0.6) node {\color{green}$\Wcal'$};



\end{tikzpicture}
\caption{Illustration of the decompositions \eqref{max common factor decomp} in terms of HN polygons.}
\end{figure}

It remains to check the desired properties for $\Ucal$, $\Vcal'$ and $\Wcal'$. By construction we have decompositions 
\[ \Vcal \simeq \Ucal \oplus \Vcal' \quad\quad \text{ and } \quad\quad \Wcal \simeq \Ucal \oplus \Wcal'.\]
Moreover, we obtain slopewise dominance of $\Vcal'$ on $\Wcal'$ from slopewise dominance of $\Vcal$ on $\Wcal$. If $\Wcal' \neq 0$, we have $r < \rk(\Wcal)$ and therefore deduce the strict inequality $\mumax(\Vcal')>\mumax(\Wcal')$ from the fact that $d(x)$ becomes positive after $r$. If $\Ucal \neq 0$ and $\Wcal' \neq 0$, we also have $\mumin(\Ucal) \geq \mumax(\Vcal')$ by convexity of $\HN(\Vcal)$, thereby obtaining a combined inequality $\mumin(\Ucal) \geq \mumax(\Vcal') > \mumax(\Wcal')$. 
\end{proof}

We will also need the following duality of slopewise dominance for vector bundles of equal ranks. 

\begin{lemma}\label{duality of slopewise dominance for equal rank case}
Let $\Vcal$ and $\Wcal$ be vector bundles on $\adicff$ with $\rk(\Vcal) = \rk(\Wcal)$. Then $\Vcal$ slopewise dominates $\Wcal$ if and only if $\Wcal^\vee$ slopewise dominates $\Vcal^\vee$. 
\end{lemma}

\begin{proof}

We align the polygons $\HN(\Vcal)$ and $\HN(\Wcal)$ so that their left points lie at the origin. By reflecting $\HN(\Vcal)$ and $\HN(\Wcal)$ about the $y$-axis, we obtain the polygons $\HN(\Vcal^\vee)$ and $\HN(\Wcal^\vee)$ with their right points at the origin. 
\begin{figure}[H]
\begin{tikzpicture}	

		\coordinate (left) at (0, 0);
		\coordinate (q0) at (1,2.5);
		\coordinate (q1) at (2.5, 4);
		\coordinate (q2) at (5, 4.5);
		

		\coordinate (p0) at (2, 1.2);
		\coordinate (p1) at (4, 1.2);
		\coordinate (p2) at (5, 0.5);
				
		\draw[step=1cm,thick] (left) -- (q0) --  (q1) -- (q2);
		\draw[step=1cm,thick] (left) -- (p0) --  (p1) -- (p2);
		
		\draw [fill] (q0) circle [radius=0.05];		
		\draw [fill] (q1) circle [radius=0.05];		
		\draw [fill] (q2) circle [radius=0.05];		
		\draw [fill] (left) circle [radius=0.05];
		
		\draw [fill] (p0) circle [radius=0.05];		
		\draw [fill] (p1) circle [radius=0.05];		
		\draw [fill] (p2) circle [radius=0.05];		

		
		\draw[step=1cm,dotted] (3, -0.4) -- (3, 4.8);
       		\draw[step=1cm,dotted] (3.5, -0.4) -- (3.5, 4.8);
       		\draw[step=1cm,dotted] (5, -0.4) -- (5, 4.8);

		\node at (2.9,-0.8) {\scriptsize $i-1$};
		\node at (3.6,-0.8) {\scriptsize $i$};
		\node at (5,-0.8) {\scriptsize $r$};
		
		\path (q2) ++(0.8, 0.05) node {$\HN(\Vcal)$};
		\path (p2) ++(0.8, 0.05) node {$\HN(\Wcal)$};
		\path (left) ++(0, -0.3) node {$O$};

		\coordinate (-q0) at (-1,2.5);
		\coordinate (-q1) at (-2.5, 4);
		\coordinate (-q2) at (-5, 4.5);
		

		\coordinate (-p0) at (-2, 1.2);
		\coordinate (-p1) at (-4, 1.2);
		\coordinate (-p2) at (-5, 0.5);
				
		\draw[step=1cm,thick] (left) -- (-q0) --  (-q1) -- (-q2);
		\draw[step=1cm,thick] (left) -- (-p0) --  (-p1) -- (-p2);
		
		\draw [fill] (-q0) circle [radius=0.05];		
		\draw [fill] (-q1) circle [radius=0.05];		
		\draw [fill] (-q2) circle [radius=0.05];		
		
		\draw [fill] (-p0) circle [radius=0.05];		
		\draw [fill] (-p1) circle [radius=0.05];		
		\draw [fill] (-p2) circle [radius=0.05];		

		
		\draw[step=1cm,dotted] (-3, -0.4) -- (-3, 4.8);
       		\draw[step=1cm,dotted] (-3.5, -0.4) -- (-3.5, 4.8);
       		\draw[step=1cm,dotted] (-5, -0.4) -- (-5, 4.8);

		\node at (-2.8,-0.8) {\scriptsize $-i+1$};
		\node at (-3.7,-0.8) {\scriptsize $-i$};
		\node at (-5,-0.8) {\scriptsize $-r$};
		
		\path (-q2) ++(-0.8, 0.05) node {$\HN(\Vcal^\vee)$};
		\path (-p2) ++(-0.8, 0.05) node {$\HN(\Wcal^\vee)$};

\end{tikzpicture}
\caption{Duality of slopewise dominance for vector bundles of equal ranks}
\end{figure}
\noindent Note that $\Vcal, \Wcal, \Vcal^\vee$ and $\Wcal^\vee$ all have equal rank by our assumption $\rk(\Vcal) = \rk(\Wcal)$. We let $r$ denote this common rank of $\Vcal, \Wcal, \Vcal^\vee$ and $\Wcal^\vee$. 

With this setup, we can establish our assertion by proving equivalence of the following statements:
\begin{enumerate}[label=(\alph*)]
\item\label{slopewise dominace of W dual on V dual} $\Wcal^\vee$ slopewise dominates $\Vcal^\vee$. 

\item\label{slope comparison of dual polygons} For each $i = 1, 2, \cdots, r$, the slope of $\HN(\Vcal^\vee)$ on the interval $[-i, -i+1]$ is less than or equal to the slope of $\HN(\Wcal^\vee)$ on this interval.

\item\label{slope comparison of origianl polygons} For each $i = 1, 2, \cdots, r$, the slope of $\HN(\Wcal)$ on the interval $[i-1, i]$ is less than or equal to the slope of $\HN(\Vcal)$ on this interval. 

\item\label{slopewise dominance of V on W} $\Vcal$ slopewise dominates $\Wcal$. 
\end{enumerate}
Equivalence between \ref{slopewise dominace of W dual on V dual} and \ref{slope comparison of dual polygons} is a consequence of the fact that the left points of $\HN(\Vcal^\vee)$ and $\HN(\Wcal^\vee)$ have the same $x$-values of $-r$ in our alignment; in fact, to compare the slopes as per Definition \ref{def of slopewise dominance} we only have to align the left points at the same $x$-values. Equivalence between \ref{slope comparison of dual polygons} and \ref{slope comparison of origianl polygons} is immediate since the slope of $\HN(\Vcal^\vee)$ (resp. $\HN(\Wcal^\vee)$) on $[-i, -i+1]$ is the negative of the slope of $\HN(\Vcal)$ (resp. $(\HN(\Wcal)$) on $[i-1, i]$. Equivalence between \ref{slope comparison of origianl polygons} and \ref{slopewise dominance of V on W} is evident by
Definition \ref{def of slopewise dominance}. 
\end{proof}

\subsection{Formulation of the key inequality}\label{reformulation of statement}$ $

Our primary goal in this subsection is to reduce the statement of Theorem \ref{classification of quotient bundles} to a quantitative statement which we can prove using the results from \S\ref{Geometric interpretation of degrees}.

We begin by establishing the equivalence of the two characterizations of quotient bundles in the statement of Theorem \ref{classification of quotient bundles}. 
\begin{prop}\label{classification of quotient bundles by slopewise dominance}
For arbitrary vector bundles $\Ecal$ and $\Fcal$ on $\adicff$, the conditions \ref{rank inequalities for quotients} and \ref{dual strong slopewise dominance for quotients} in Theorem \ref{classification of quotient bundles} are equivalent.
\end{prop}

\begin{proof}
As in the statement of Theorem \ref{classification of quotient bundles}, we align $\HN(\Ecal)$ and $\HN(\Fcal)$ so that their right endpoints lie at the origin. By reflecting $\HN(\Ecal)$ and $\HN(\Fcal)$ about the $y$-axis, we obtain the HN polygons $\HN(\Ecal^\vee)$ and $\HN(\Fcal^\vee)$ with their left endpoints at the origin. 


Let us now assert that the condition \ref{rank inequalities for quotients} implies the condition \ref{dual strong slopewise dominance for quotients}. 
Suppose that for some positive integer $i \leq \rk(\Fcal)$ the slope of $\HN(\Fcal)$ on $[-i, -i+1]$ is less than the slope of $\HN(\Ecal)$ on $[-i-1, -i]$. By Lemma \ref{slopewise dominance and rank inequalities} the inequality in the condition \ref{rank inequalities for quotients} is equivalent to slopewise dominance of $\Ecal^\vee$ on $\Fcal^\vee$. Then our observation in the preceding paragraph shows that the slope of $\HN(\Fcal)$ on $[-i, -i+1]$ must be greater than or equal to the slope of $\HN(\Ecal)$ on $[-i, -i+1]$. Therefore our assumption implies that $\HN(\Ecal)$ has a breakpoint at $-i$. Taking $\mu$ to be the slope of $\HN(\Fcal)$ on $[-i, -i+1]$, we find
\[ \rk(\Ecal^{\leq \mu}) \leq i = \rk(\Fcal^{\leq \mu}).\]
Now the condition \ref{rank inequalities for quotients} yields $\rk(\Ecal^{\leq \mu}) = \rk(\Fcal^{\leq \mu}) = j$ and $\Ecal^{\leq \mu} \simeq \Fcal^{\leq \mu}$, thereby implying that $\HN(\Ecal)$ and $\HN(\Fcal)$ must agree on $[-i, 0]$ as the condition \ref{dual strong slopewise dominance for quotients} states.

It remains to prove that the condition \ref{dual strong slopewise dominance for quotients} implies the condition \ref{rank inequalities for quotients}. By our observation in the first paragraph and convexity of HN polygons, the condition \ref{dual strong slopewise dominance for quotients} implies that $\Ecal^\vee$ slopewise dominates $\Fcal^\vee$, and consequently yields the inequality in the condition \ref{rank inequalities for quotients} by Lemma \ref{slopewise dominance and rank inequalities}. Let us now suppose that we have $\rk(\Ecal^{\leq \mu}) = \rk(\Fcal^{\leq \mu})$ for some $\mu \in \Q$. Taking $i = \rk(\Ecal^{\leq \mu}) = \rk(\Fcal^{\leq \mu})$ we make the following observations:
\begin{enumerate}[label=(\alph*)]
\item Both $\HN(\Ecal)$ and $\HN(\Fcal)$ have vertices at $-i$. 
\smallskip

\item The slope of $\HN(\Fcal)$ on $[-i, -i+1]$ is less than or equal to $\mu$. 
\smallskip

\item The slope of $\HN(\Ecal)$ on $[-i-1, -i]$ is greater than $\mu$ unless $i = \rk(\Ecal) = \rk(\Fcal)$. 
\end{enumerate}
Hence the condition \ref{dual strong slopewise dominance for quotients} implies that $\HN(\Ecal)$ and $\HN(\Fcal)$ must agree on $[-i, 0]$, and consequently yields an isomorphism $\Ecal^{\leq \mu} \simeq \Fcal^{\leq \mu}$ as desired. 
\end{proof}


As our next step, we verify the necessity part of Theorem \ref{classification of quotient bundles}. 

\begin{prop}\label{quotient bundles necessary condition}
Given a vector bundle $\Ecal$ on $\adicff$, every quotient bundle $\Fcal$ of $\Ecal$ satisfies the condition \ref{rank inequalities for quotients} of Theorem \ref{classification of quotient bundles}. 
\end{prop}

\begin{proof}
Let $\mu$ be an arbitrary rational number, and consider the decomposition $\Ecal \simeq \Ecal^{\leq \mu} \oplus \Ecal^{>\mu}$. Since any bundle map from $\Ecal^{>\mu}$ to $\Fcal^{\leq \mu}$ must be zero by Lemma \ref{zero hom for dominating slopes}, the composite surjective map $\Ecal \surj \Fcal \surj \Fcal^{\leq \mu}$ should factor through $\Ecal^{\leq \mu}$. We thus find $\rk(\Ecal^{\leq \mu}) \geq \rk(\Fcal^{\leq \mu})$.

Let us now assume that $\rk(\Ecal^{\leq \mu}) = \rk(\Fcal^{\leq \mu})$ for some $\mu \in \Q$. Then the kernel of the surjective map $\Ecal^{\leq \mu} \surj \Fcal^{\leq \mu}$ must be zero since it is a subbundle of $\Ecal^{\leq \mu}$ whose rank is equal to $\rk(\Ecal^{\leq \mu}) - \rk(\Fcal^{\leq \mu}) = 0$. Hence we obtain an isomorphism $\Ecal^{\leq \mu} \simeq \Fcal^{\leq \mu}$, thereby verifying the condition \ref{rank inequalities for quotients}.
\end{proof}

We also note that Proposition \ref{quotient bundles necessary condition} has the following dual statement:
\begin{prop}\label{subbundles necessary condition}
Given a vector bundle $\Ecal$ on $\adicff$, every subsheaf $\Dcal$ of $\Ecal$ satisfies the inequality $\rk(\Ecal^{\geq \mu}) \leq \rk(\Dcal^{\geq \mu})$ for every $\mu \in \Q$. 
\end{prop}

\begin{proof}
Let $\mu$ be an arbitrary rational number, and consider the decomposition $\Ecal = \Ecal^{< \mu} \oplus \Ecal^{\geq \mu}$. Since any bundle map from $\Dcal^{\geq \mu}$ to $\Ecal^{<\mu}$ must be zero by Lemma \ref{zero hom for dominating slopes}, the composite injective map $\Dcal^{\geq \mu} \inj \Dcal \inj \Ecal$ should factor through $\Ecal^{\geq \mu}$. We thus find $\rk(\Ecal^{\geq \mu}) \leq \rk(\Dcal^{\geq \mu})$, which is equivalent to the desired inequality by Lemma \ref{rank and degree of dual bundle}. 
\end{proof}

By Proposition \ref{classification of quotient bundles by slopewise dominance} and Proposition \ref{quotient bundles necessary condition}, it remains to prove the sufficiency of the condition \ref{rank inequalities for quotients} in Theorem \ref{classification of quotient bundles}. 
For this, we have the following elementary and important reduction:
\begin{lemma}\label{reduction on common slopes}
We may prove the sufficiency of the condition \ref{rank inequalities for quotients} in Theorem \ref{classification of quotient bundles} under the assumption that $\Ecal$ and $\Fcal$ have no common slopes. 
\end{lemma}

\begin{proof}
Let $\Ecal$ and $\Fcal$ be vector bundles on $\adicff$ satisfying the condition \ref{rank inequalities for quotients} in Theorem \ref{classification of quotient bundles}. By their HN decompositions, we may write
\begin{equation}\label{common slopes decomp}
\Ecal \simeq \Gcal \oplus \commonslopered{\Ecal} \quad\quad \text{ and } \quad\quad \Fcal \simeq \Gcal \oplus \commonslopered{\Fcal}
\end{equation}
for some vector bundles $\commonslopered{\Ecal}$ and $\commonslopered{\Fcal}$ with no common slopes. 

Let $\mu$ be an arbitrary rational number. We assert that $\rk(\commonslopered{\Ecal}^{\leq \mu}) \geq \rk(\commonslopered{\Fcal}^{\leq \mu})$ with equality if only if $\commonslopered{\Ecal}^{\leq \mu}$ is isomorphic to $\commonslopered{\Fcal}^{\leq \mu}$. By the decompositions in \eqref{common slopes decomp} we have
\begin{equation}\label{common slopes decomp with bounded slopes}
\Ecal^{\leq \mu} \simeq \Gcal^{\leq \mu} \oplus \commonslopered{\Ecal}^{\leq \mu} \quad\quad \text{ and } \quad\quad \Fcal^{\leq \mu} \simeq \Gcal^{\leq \mu} \oplus \commonslopered{\Fcal}^{\leq \mu},
\end{equation}
which consequently yield
\begin{equation}\label{rank relations for common slopes decomp with bounded slopes}
\rk(\Ecal^{\leq \mu}) = \rk(\Gcal^{\leq \mu}) + \rk(\commonslopered{\Ecal}^{\leq \mu}) \quad\quad \text{ and } \quad\quad \rk(\Fcal^{\leq \mu}) = \rk(\Gcal^{\leq \mu}) + \rk(\commonslopered{\Fcal}^{\leq \mu}).
\end{equation}
We thus deduce the desired inequality $\rk(\commonslopered{\Ecal}^{\leq \mu}) \geq \rk(\commonslopered{\Fcal}^{\leq \mu})$ from the corresponding inequality $\rk({\Ecal}^{\leq \mu}) \geq \rk({\Fcal}^{\leq \mu})$ for $\Ecal$ and $\Fcal$. Moreover, if we have $\rk(\commonslopered{\Ecal}^{\leq \mu}) = \rk(\commonslopered{\Fcal}^{\leq \mu})$, then by \eqref{rank relations for common slopes decomp with bounded slopes} we find $\rk(\Ecal^{\leq \mu}) = \rk(\Fcal^{\leq \mu})$, which in turn yields an isomorphism $\Ecal^{\leq \mu} \simeq \Fcal^{\leq \mu}$ and consequently implies by \eqref{common slopes decomp with bounded slopes} that $\commonslopered{\Ecal}^{\leq \mu}$ is isomorphic to $\commonslopered{\Fcal}^{\leq \mu}$. 

Now we observe from \eqref{common slopes decomp} that a surjective bundle map $\commonslopered{\Ecal} \surj \commonslopered{\Fcal}$ gives rise to a surjective bundle map $\Ecal \surj \Fcal$ by direct summing with the identity map for $\Gcal$. Hence our discussion in the preceding paragraph implies that we can prove the sufficiency of the condition \ref{rank inequalities for quotients} in Theorem \ref{classification of quotient bundles} after replacing $\Ecal$ and $\Fcal$ by $\commonslopered{\Ecal}$ and $\commonslopered{\Fcal}$, yielding our desired reduction as $\commonslopered{\Ecal}$ and $\commonslopered{\Fcal}$ have no common slopes by construction. 
\end{proof}

\begin{remark}
After our reduction in Lemma \ref{reduction on common slopes}, the rank inequality of the condition \ref{rank inequalities for quotients} in Theorem \ref{classification of quotient bundles} becomes essentially strict. In fact, if $\Ecal$ and $\Fcal$ have no common slopes, $\Ecal^{\leq \mu}$ is isomorphic to  $\Fcal^{\leq \mu}$ only holds for $\mu \in \Q$ with $\Ecal^{\leq \mu} = \Fcal^{\leq \mu} = 0$. 
\end{remark}

We now state our key inequality for proving the sufficiency of the condition \ref{rank inequalities for quotients} in Theorem \ref{classification of quotient bundles}. 

\begin{prop}\label{key inequality}
Let $\Ecal$, $\Fcal$ and $\Qcal$ be vector bundles on $\adicff$ with the following properties:
\begin{enumerate}[label=(\roman*)]
\item\label{slope condition on E and F} $\rk(\Ecal^{\leq \mu}) \geq \rk(\Fcal^{\leq \mu})$ for every $\mu \in \Q$ with equality only when $\Ecal^{\leq \mu} \simeq \Fcal^{\leq \mu}$. 

\item\label{slope condition on E and Q} $\rk(\Ecal^{\leq \mu}) \geq \rk(\Qcal^{\leq \mu})$ for every $\mu \in \Q$ with equality only when $\Ecal^{\leq \mu} \simeq \Qcal^{\leq \mu}$. 

\item\label{slope condition on F and Q} $\rk(\Fcal^{\geq \mu}) \geq \rk(\Qcal^{\geq \mu})$ for every $\mu \in \Q$. 

\item\label{no common slopes for E and F} $\Ecal$ and $\Fcal$ have no common slopes. 
\end{enumerate}
Then we have an inequality
\begin{equation}\label{deg inequality for surj}
\deg(\Ecal^\vee \otimes \Qcal)^\nonneg + \deg(\Qcal^\vee \otimes \Fcal)^\nonneg \leq \deg(\Ecal^\vee \otimes \Fcal)^\nonneg + \deg(\Qcal^\vee \otimes \Qcal)^\nonneg
\end{equation}
with equality if and only if $\Qcal = \Fcal$. 
\end{prop}

\begin{example}\label{example showing necessity of no common slopes condition}
We present an example which shows that our reduction in Lemma \ref{reduction on common slopes} is crucial for the formulation of Proposition \ref{key inequality}. Take $\Ecal = \trivbundle^{\oplus 3}, \Fcal = \trivbundle^{\oplus 2}$ and $\Qcal = \trivbundle$. Note that our choice does not satisfy the property \ref{no common slopes for E and F}. However, we check the other properties \ref{slope condition on E and F}, \ref{slope condition on E and Q} and \ref{slope condition on F and Q} by Proposition \ref{quotient bundles necessary condition} and Proposition \ref{subbundles necessary condition} after observing that $\Fcal$ and $\Qcal$ are quotient bundles of $\Ecal$ while $\Qcal$ is a subbundle of $\Fcal$. We now observe that all terms in \eqref{deg inequality for surj} are zero, thereby obtaining an equality even though $\Qcal \neq \Fcal$. 
Hence the equality condition in Proposition \ref{key inequality} can be broken when $\Ecal$ and $\Fcal$ have common slopes. 
\end{example}

We will prove Proposition \ref{key inequality} in \S\ref{proof of key inequality}. Here we explain why establishing Proposition \ref{key inequality} finishes the proof of Theorem \ref{classification of quotient bundles}. 

\begin{prop}
Proposition \ref{key inequality} implies the sufficiency of the condition \ref{rank inequalities for quotients} in Theorem \ref{classification of quotient bundles}. 
\end{prop}

\begin{proof}
Let $\Ecal$ and $\Fcal$ be vector bundles on $\adicff$ satisfying the condition \ref{rank inequalities for quotients} in Theorem \ref{classification of quotient bundles}. We further assume that $\Ecal$ and $\Fcal$ have no common slopes in light of Lemma \ref{reduction on common slopes}. We wish to prove existence of a surjective bundle map $\Ecal \surj \Fcal$ assuming Proposition \ref{key inequality}. For this, it suffices to check that $\Ecal$ and $\Fcal$ satisfy the assumptions \ref{existence of nonzero bundle map from E to F} and \ref{positive codim for Hom minus surj} in Proposition \ref{dimension inequality for surj maps}. 

We first check the assumption \ref{existence of nonzero bundle map from E to F} in Proposition \ref{dimension inequality for surj maps} for $\Ecal$ and $\Fcal$. By Lemma \ref{zero hom for dominating slopes}, it is enough to prove $\mumin(\Ecal) \leq \mumax(\Fcal)$.
We verify this by observing
\[\rk(\Ecal^{\leq \mumax(\Fcal)}) \geq \rk(\Fcal^{\leq \mumax(\Fcal)}) = \rk(\Fcal) >0.\]



It remains to check the assumption \ref{positive codim for Hom minus surj} in Proposition \ref{dimension inequality for surj maps} for $\Ecal$ and $\Fcal$. 
Let $\Qcal$ be an arbitrary proper subsheaf of $\Fcal$ which also occurs as a quotient of $\Ecal$. Then $\Ecal, \Fcal$ and $\Qcal$ satisfy the assumptions of Proposition \ref{key inequality} by Proposition \ref{quotient bundles necessary condition}, Proposition \ref{subbundles necessary condition} and our assumption on $\Ecal$ and $\Fcal$. 
Since we have $\Qcal \neq \Fcal$, Proposition \ref{key inequality} now yields a strict inequality 
\[ \deg(\Ecal^\vee \otimes \Qcal)^\nonneg + \deg(\Qcal^\vee \otimes \Fcal)^\nonneg < \deg(\Ecal^\vee \otimes \Fcal)^\nonneg + \deg(\Qcal^\vee \otimes \Qcal)^\nonneg\]
as required by the assumption \ref{positive codim for Hom minus surj} in Proposition \ref{dimension inequality for surj maps}.
\end{proof}

\subsection{Proof of the key inequality}\label{proof of key inequality}$ $

We now aim to establish Proposition \ref{key inequality}. 
\begin{defn}\label{definition of cEF(Q)}
For arbitrary vector bundles $\Ecal, \Fcal$ and $\Qcal$ on $\adicff$, we define
\[c_{\Ecal, \Fcal}(\Qcal) := \deg(\Ecal^\vee \otimes \Fcal)^\nonneg + \deg(\Qcal^\vee \otimes \Qcal)^\nonneg - \deg(\Ecal^\vee \otimes \Qcal)^\nonneg -\deg(\Qcal^\vee \otimes \Fcal)^\nonneg.\]
\end{defn}

\begin{remark}
The inequality \eqref{deg inequality for surj} in Proposition \ref{key inequality} can be stated as $c_{\Ecal, \Fcal}(\Qcal) \geq 0$.
\end{remark}

Our proof of Proposition \ref{key inequality} will consist of a series of reduction steps as follows:
\begin{enumerate}[label=Step \arabic*., leftmargin=5.3em]
\item\label{reduction to integer slopes} We reduce the proof to the case where all slopes of $\Ecal, \Fcal$ and $\Qcal$ are integers. 

\item\label{reduction to equal ranks} We further reduce the proof to the case $\rk(\Qcal) = \rk(\Fcal)$. 

\item\label{reduction to equal slopes} After these reductions, we complete the proof by gradually ``reducing" the slopes of $\Fcal$ to the slopes of $\Qcal$.  
\end{enumerate}
Throughout these reduction steps, we will establish the following key facts:
\begin{enumerate}[label=(\arabic*)]
\item\label{decreasing codimension over the entire reduction process} The quantity $c_{\Ecal, \Fcal}(\Qcal)$ monotone decreases to $0$ as we reduce $\rk(\Fcal)$ to $\rk(\Qcal)$ and the slopes of $\Fcal$ to the slopes of $\Qcal$. 

\item\label{equality condition for nonequal rank case} When $\rk(\Qcal)<\rk(\Fcal)$, the equality $c_{\Ecal, \Fcal}(\Qcal) = 0$ never holds. 

\item\label{equality condition for equal rank case} When $\rk(\Qcal) = \rk(\Fcal)$, the equality $c_{\Ecal, \Fcal}(\Qcal) = 0$ holds only when $\Qcal = \Fcal$. 
\end{enumerate}
We will then obtain the desired inequality $c_{\Ecal, \Fcal}(\Qcal) \geq 0$ from the first fact and the equality condition $\Qcal = \Fcal$ from the second and the third facts.

\begin{remark} For curious readers, we 
briefly describe how each assumption in Proposition \ref{key inequality} will be used to establish the key facts \ref{decreasing codimension over the entire reduction process}, \ref{equality condition for nonequal rank case} and \ref{equality condition for equal rank case} above. 


The fact \ref{decreasing codimension over the entire reduction process} relies on the rank inequalities 
from the assumptions \ref{slope condition on E and F}, \ref{slope condition on E and Q} and \ref{slope condition on F and Q} of Proposition \ref{key inequality}. As we will see in Lemma \ref{slopewise dominance relations for key inequality}, these rank inequalities can be interpreted as slopewise dominance relations between the vector bundles $\Ecal, \Fcal$ and $\Qcal$. We will use those relations to ``gradually reduce" $\HN(\Fcal)$ to $\HN(\Qcal)$ in a way that $c_{\Ecal, \Fcal}(\Qcal)$ always decreases. 

The fact \ref{equality condition for nonequal rank case} is essentially a consequence of the assumption \ref{no common slopes for E and F} of Proposition \ref{key inequality}. In Proposition \ref{key inequality reduction for equal rank}, we will use this assumption to prove that $c_{\Ecal, \Fcal}(\Qcal)$ strictly decreases while we reduce the rank of $\Fcal$ by cutting down $\HN(\Fcal)$ from the right  (cf. Example \ref{example showing necessity of no common slopes condition}). 

The fact \ref{equality condition for equal rank case} comes from the assumption \ref{no common slopes for E and F} along with the equality conditions 
in the assumptions \ref{slope condition on E and F} and \ref{slope condition on E and Q} of Proposition \ref{key inequality}. As we will see in Lemma \ref{decreasing c_EF(Q) after max reduction}, these assumptions ensure that $c_{\Ecal, \Fcal}(\Qcal)$ strictly decreases during the first reduction cycle in Step 3. 
\end{remark}




Before proceeding to our reduction steps, let us make some useful observations about the assumptions of Proposition \ref{key inequality}. 

\begin{lemma}\label{slopewise dominance relations for key inequality}
Let $\Ecal, \Fcal$ and $\Qcal$ be as in the statement of Proposition \ref{key inequality}. Then we have the following slopewise dominance relations:
\begin{enumerate}[label=(\arabic*)]
\item $\Ecal^\vee$ slopewise dominates $\Fcal^\vee$.

\item $\Ecal^\vee$ slopewise dominates $\Qcal^\vee$.

\item $\Fcal$ slopewise dominates $\Qcal$. 
\end{enumerate}
\end{lemma}

\begin{proof}
By Lemma \ref{slopewise dominance and rank inequalities}, these slopewise dominance relations follow from the rank inequalities in the assumptions \ref{slope condition on E and F}, \ref{slope condition on E and Q} and \ref{slope condition on F and Q} of Proposition \ref{key inequality}.
\end{proof}

\begin{remark} Based on Lemma \ref{slopewise dominance relations for key inequality}, we can give an intuitive explanation on how our reduction steps work. Let us rewrite $c_{\Ecal, \Fcal}(\Qcal)$ as 
\[ c_{\Ecal, \Fcal}(\Qcal):= \deg((\Fcal^\vee)^\vee \otimes \Ecal^\vee)^\nonneg + \deg(\Qcal^\vee \otimes \Qcal)^\nonneg - \deg((\Qcal^\vee)^\vee \otimes \Ecal^\vee)^\nonneg - \deg(\Qcal^\vee \otimes \Fcal)^\nonneg.\]
Then by Lemma \ref{slopewise dominance relations for key inequality} every term on the right side is of the form $\deg(\Vcal^\vee \otimes \Wcal)^\nonneg$ where $\Wcal$ slopewise dominates $\Vcal$. 
During our reduction steps, we will utilize the slopewise dominance relations as stated in Lemma \ref{slopewise dominance relations for key inequality} so that all terms in $c_{\Ecal, \Fcal}(\Qcal)$ behave as we want.  
\end{remark}


\begin{lemma}\label{assumptions of key inequality after vertical stretch}
Let $\Ecal, \Fcal$ and $\Qcal$ be as in the statement of Proposition \ref{key inequality}. Choose a positive integer $C$, and let $\vertstretch{\Ecal}, \vertstretch{\Fcal}$ and $\vertstretch{\Qcal}$ be vector bundles on $\adicff$ whose HN polygons are obtained by vertically stretching $\HN(\Ecal), \HN(\Fcal)$ and $\HN(\Qcal)$ by a factor $C$. Then we have the following properties of $\vertstretch{\Ecal}, \vertstretch{\Fcal}$ and $\vertstretch{\Qcal}$. 
\begin{enumerate}[label=(\roman*)]
\item\label{slope condition on E and F after vertical stretch} $\rk(\vertstretch{\Ecal}^{\leq \mu}) \geq \rk(\vertstretch{\Fcal}^{\leq \mu})$ for every $\mu \in \Q$ with equality only when $\vertstretch{\Ecal}^{\leq \mu} \simeq \vertstretch{\Fcal}^{\leq \mu}$. 
\smallskip

\item\label{slope condition on E and Q after vertical stretch} $\rk(\vertstretch{\Ecal}^{\leq \mu}) \geq \rk(\vertstretch{\Qcal}^{\leq \mu})$ for every $\mu \in \Q$ with equality only when $\vertstretch{\Ecal}^{\leq \mu} \simeq \vertstretch{\Qcal}^{\leq \mu}$. 
\smallskip

\item\label{slope condition on F and Q after vertical stretch} $\rk(\vertstretch{\Fcal}^{\geq \mu}) \geq \rk(\vertstretch{\Qcal}^{\geq \mu})$ for every $\mu \in \Q$. 
\smallskip

\item\label{no common slopes for E and F after vertical stretch} $\vertstretch{\Ecal}$ and $\vertstretch{\Fcal}$ have no common slopes. 
\smallskip
\end{enumerate}
\end{lemma}

\begin{proof}
By construction, we have the following facts:
\begin{enumerate}[label = (\alph*)]
\item\label{vertical stretch slopewise rank} For $\Vcal = \Ecal, \Fcal$ and $\Qcal$, we have $\rk(\vertstretch{\Vcal}^{\leq \mu}) = \rk(\Vcal^{\leq \mu/C})$ and $\rk(\vertstretch{\Vcal}^{\geq \mu}) = \rk(\Vcal^{\geq \mu/C})$ for every $\mu \in \Q$. 
\smallskip

\item\label{vertical stretch equal rank condition} For $\Wcal = \Fcal$ and $\Qcal$, we have $\vertstretch{\Ecal}^{\leq \mu} \simeq \vertstretch{\Wcal^{\leq \mu}}$
if $\Ecal^{\leq \mu/C} \simeq \Wcal^{\leq \mu/C}$. 
\smallskip

\item\label{vertical stretch slopes} The slopes of $\vertstretch{\Ecal}$ and $\vertstretch{\Fcal}$ are given by multiplying the slopes of $\Ecal$ and $\Fcal$ by $C$. 
\end{enumerate}
Hence we deduce the properties \ref{slope condition on E and F after vertical stretch} - \ref{no common slopes for E and F after vertical stretch} from the corresponding properties of $\Ecal, \Fcal$ and $\Qcal$. 
\end{proof}

\begin{lemma}\label{assumptions of key inequality after shear}
Let $\Ecal, \Fcal$ and $\Qcal$ be as in the statement of Proposition \ref{key inequality}. For any integer $\lambda$, the vector bundles $\Ecal(\lambda), \Fcal(\lambda)$ and $\Qcal(\lambda)$ satisfy the following properties:
\begin{enumerate}[label=(\roman*)]
\item\label{slope condition on E and F after shear} $\rk(\Ecal(\lambda)^{\leq \mu}) \geq \rk(\Fcal(\lambda)^{\leq \mu})$ for every $\mu \in \Q$ with equality only when $\Ecal(\lambda)^{\leq \mu} \simeq \Fcal(\lambda)^{\leq \mu}$. 
\smallskip

\item\label{slope condition on E and Q after shear} $\rk(\Ecal(\lambda)^{\leq \mu}) \geq \rk(\Qcal(\lambda)^{\leq \mu})$ for every $\mu \in \Q$ with equality only when $\Ecal(\lambda)^{\leq \mu} \simeq \Qcal(\lambda)^{\leq \mu}$. 
\smallskip

\item\label{slope condition on F and Q after vertical stretch} $\rk(\Fcal(\lambda)^{\geq \mu}) \geq \rk(\Qcal(\lambda)^{\geq \mu})$ for every $\mu \in \Q$. 
\smallskip

\item\label{no common slopes for E and F after shear} $\Ecal(\lambda)$ and $\Fcal(\lambda)$ have no common slopes 
\smallskip
\end{enumerate}
\end{lemma}

\begin{proof}
Since the vector bundle $\trivbundle(\lambda)$ has rank $1$ and degree $\lambda$, tensoring a vector bundle with $\trivbundle(\lambda)$ is the same as adding $\lambda$ to all slopes. Therefore we have the following observations:
\begin{enumerate}[label = (\alph*)]

\item\label{shear slopewise rank} For $\Vcal = \Ecal, \Fcal$ or $\Qcal$, we have $\rk(\Vcal(\lambda)^{\leq \mu}) = \rk(\Vcal^{\leq \mu-\lambda})$ and $\rk(\Vcal(\lambda)^{\geq \mu}) = \rk(\Vcal^{\geq \mu-\lambda})$ for every $\mu \in \Q$. 
\smallskip

\item\label{shear equal rank condition} For $\Wcal = \Fcal$ and $\Qcal$, we have $\Ecal(\lambda)^{\leq \mu} \simeq \Wcal(\lambda)^{\leq \mu}$ if $\Ecal^{\leq \mu-\lambda} \simeq \Wcal^{\leq \mu-\lambda}$. 
\smallskip

\item\label{shear slopes} The slopes of $\Ecal(\lambda)$ and $\Fcal(\lambda)$ are given by adding $\lambda$ to the slopes of $\Ecal$ and $\Fcal$. 
\end{enumerate}
We thus deduce the properties \ref{slope condition on E and F after shear} - \ref{no common slopes for E and F after vertical stretch} from the corresponding properties of $\Ecal, \Fcal$ and $\Qcal$. 
\end{proof}

We are now ready to carry out Step 1 and Step 2. 
\begin{prop}\label{key inequality reduction to integer slopes}
To prove Proposition \ref{key inequality}, we may assume that all slopes of $\Ecal$, $\Fcal$ and $\Qcal$ are integers. 
\end{prop}

\begin{proof}
Let $\Ecal, \Fcal$ and $\Qcal$ be as in the statement of Proposition \ref{key inequality}. Take $C$ to be the least common multiple of all denominators of the slopes of $\Ecal, \Fcal$ and $\Qcal$, and let $\vertstretch{\Ecal}, \vertstretch{\Fcal}$ and $\vertstretch{\Qcal}$ be vector bundles on $\adicff$ whose HN polygons are obtained by vertically stretching $\HN(\Ecal), \HN(\Fcal)$ and $\HN(\Qcal)$ by a factor $C$. Note that all slopes of $\vertstretch{\Ecal}, \vertstretch{\Fcal}$ and $\vertstretch{\Qcal}$ are integers by construction. We now use Lemma \ref{degree after stretch} to obtain an identity 
\[c_{\vertstretch{\Ecal}, \vertstretch{\Fcal}}(\vertstretch{\Qcal}) = C \cdot c_{\Ecal, \Fcal}(\Qcal)\] 
which implies that the inequality \eqref{deg inequality for surj} for $\Ecal, \Fcal$ and $\Qcal$ follows from the corresponding inequality for $\vertstretch{\Ecal}, \vertstretch{\Fcal}$ and $\vertstretch{\Qcal}$.  In addition, our construction translates the equality condition $\Qcal = \Fcal$ for the former inequality to the equality condition $\vertstretch{\Qcal} = \vertstretch{\Fcal}$ for the latter inequality. Now Lemma \ref{assumptions of key inequality after vertical stretch} implies that we may prove Proposition \ref{key inequality} after replacing $\Ecal, \Fcal$ and $\Qcal$ by $\vertstretch{\Ecal}, \vertstretch{\Fcal}$ and $\vertstretch{\Qcal}$, thereby yielding our desired reduction. 
\end{proof}

\begin{prop}\label{key inequality reduction for equal rank}
It suffices to prove Proposition \ref{key inequality} under the additional assumptions that $\rk(\Qcal) = \rk(\Fcal)$ and that all slopes of $\Ecal, \Fcal$ and $\Qcal$ are integers. 
\end{prop}

\begin{proof}
Suppose that Proposition \ref{key inequality} holds in the special case where the additional assumptions are satisfied. We assert that the general case of Proposition \ref{key inequality} follows from this special case by induction on $\rk(\Fcal) - \rk(\Qcal)$. We assume that all slopes of $\Ecal, \Fcal$ and $\Qcal$ are integers in light of Proposition \ref{key inequality reduction to integer slopes}. 


We first reduce our induction step to the case $\mumin(\Fcal) = 0$. For this, we take $\lambda = -\mumin(\Fcal)$ 
so that we have $\mumin(\Fcal(\lambda)) = \mumin(\Fcal) + \lambda = 0$. 
Our assumption implies that $\lambda$ is an integer, and consequently that all slopes of $\Ecal(\lambda), \Fcal(\lambda)$ and $\Qcal(\lambda)$ are integers as well. We now apply Lemma \ref{degree after shear} to get an identity
\[c_{\Ecal(\lambda), \Fcal(\lambda)}(\Qcal(\lambda)) = c_{\Ecal, \Fcal}(\Qcal),\] 
which implies that the inequality \eqref{deg inequality for surj} for $\Ecal, \Fcal$ and $\Qcal$ is equivalent to the corresponding inequality for $\Ecal(\lambda), \Fcal(\lambda)$ and $\Qcal(\lambda)$. In addition, we translate the equality condition $\Qcal = \Fcal$ for the former inequality to the equality condition $\Qcal(\lambda) = \Fcal(\lambda)$ for the latter inequality. We also have $\rk(\Fcal) - \rk(\Qcal) = \rk(\Fcal(\lambda)) - \rk(\Qcal(\lambda))$ as tensoring with $\trivbundle(\lambda)$ does not change ranks. Now Lemma \ref{assumptions of key inequality after shear} implies that we may proceed to the induction step after replacing $\Ecal, \Fcal$ and $\Qcal$ by $\Ecal(\lambda), \Fcal(\lambda)$ and $\Qcal(\lambda)$, thereby yielding our desired reduction. 

Let us now assume that $\mumin(\Fcal) = 0$. For our induction step we have $\rk(\Fcal) - \rk(\Qcal) >0$, or equivalently $\rk(\Fcal) > \rk(\Qcal)$. Hence we can write 
\[\Fcal = \rankred{\Fcal} \oplus \trivbundle\] 
for some vector bundle $\rankred{\Fcal}$ with $\mumin(\rankred{\Fcal}) \geq 0$ and $\rk(\rankred{\Fcal})  \geq \rk(\Qcal)$.  
\begin{figure}[H]
\begin{tikzpicture}	

		\coordinate (left) at (0, 0);
		\coordinate (q0) at (1,2);
		\coordinate (q1) at (2, 3);
		\coordinate (q2) at (3.5, 3.5);
		\coordinate (q3) at (5, 3.5);
		

		\coordinate (p0) at (1.5, 1);
		\coordinate (p1) at (3, 1.3);
		\coordinate (p2) at (4, 0.7);
				
		\draw[step=1cm,thick] (left) -- (q0) --  (q1) -- (q2) -- (q3);
		\draw[step=1cm,thick] (left) -- (p0) --  (p1) -- (p2);
		
		\draw [fill] (q0) circle [radius=0.05];		
		\draw [fill] (q1) circle [radius=0.05];		
		\draw [fill] (q2) circle [radius=0.05];		
		\draw [fill] (q3) circle [radius=0.05];		
		\draw [fill] (left) circle [radius=0.05];
		
		\draw [fill] (p0) circle [radius=0.05];		
		\draw [fill] (p1) circle [radius=0.05];		
		\draw [fill] (p2) circle [radius=0.05];		

		\draw[step=1cm,dotted] (4.5, -0.4) -- (4.5, 3.6);

		\node at (4.4,-0.8) {\scriptsize $\rk(\Fcal)-1$};
		
		\path (q3) ++(0.2, 0.3) node {$\HN(\Fcal)$};
		\path (p2) ++(-0.2, -0.3) node {$\HN(\Qcal)$};
		\path (left) ++(-0.3, -0.05) node {$O$};

\end{tikzpicture}
\begin{tikzpicture}[scale=0.4]
        \pgfmathsetmacro{\textycoordinate}{8}
		\draw[->, line width=0.6pt] (0, \textycoordinate) -- (1.5,\textycoordinate);
		\draw (0,0) circle [radius=0.00];	
        \hspace{0.2cm}
\end{tikzpicture}
\begin{tikzpicture}	

		\coordinate (left) at (0, 0);
		\coordinate (q0) at (1,2);
		\coordinate (q1) at (2, 3);
		\coordinate (q2) at (3.5, 3.5);
		\coordinate (q3) at (4.5, 3.5);
		

		\coordinate (p0) at (1.5, 1);
		\coordinate (p1) at (3, 1.3);
		\coordinate (p2) at (4, 0.7);
				
		\draw[step=1cm,thick] (left) -- (q0) --  (q1) -- (q2) -- (q3);
		\draw[step=1cm,thick] (left) -- (p0) --  (p1) -- (p2);
		
		\draw [fill] (q0) circle [radius=0.05];		
		\draw [fill] (q1) circle [radius=0.05];		
		\draw [fill] (q2) circle [radius=0.05];		
		\draw [fill] (q3) circle [radius=0.05];		
		\draw [fill] (left) circle [radius=0.05];
		
		\draw [fill] (p0) circle [radius=0.05];		
		\draw [fill] (p1) circle [radius=0.05];		
		\draw [fill] (p2) circle [radius=0.05];		
		
		\draw[step=1cm,dotted] (4.5, -0.4) -- (4.5, 3.6);

		\draw[step=1cm,dashed] (q3) -- (5, 3.5);

		\node at (4.4,-0.8) {\scriptsize $\rk(\Fcal)-1$};
		
		\path (q3) ++(0.2, 0.3) node {$\HN(\rankred{\Fcal})$};
		\path (p2) ++(-0.2, -0.3) node {$\HN(\Qcal)$};
		\path (left) ++(-0.3, -0.05) node {$O$};

\end{tikzpicture}
\caption{Illustration of the induction step}\label{slopes_cond_fig}
\end{figure}

We assert that our assumptions on $\Ecal, \Fcal$ and $\Qcal$ yield the corresponding conditions on $\Ecal, \rankred{\Fcal}$ and $\Qcal$. In other words, we claim that $\Ecal, \rankred{\Fcal}$ and $\Qcal$ satisfy the following properties:
\begin{enumerate}[label=(\roman*)]
\item\label{slope condition on E and F for induction step} $\rk(\Ecal^{\leq \mu}) \geq \rk(\rankred{\Fcal}^{\leq \mu})$ for every $\mu \in \Q$ with equality only when $\Ecal^{\leq \mu} \simeq \rankred{\Fcal}^{\leq \mu}$. 
\smallskip

\item\label{slope condition on E and Q for induction step} $\rk(\Ecal^{\leq \mu}) \geq \rk(\Qcal^{\leq \mu})$ for every $\mu \in \Q$ with equality only when $\Ecal^{\leq \mu} \simeq \Qcal^{\leq \mu}$. 
\smallskip

\item\label{slope condition on F and Q for induction step} $\rk(\rankred{\Fcal}^{\geq \mu}) \geq \rk(\Qcal^{\geq \mu})$ for every $\mu \in \Q$. 
\smallskip

\item\label{no common slopes for induction step} $\Ecal$ and $\rankred{\Fcal}$ have no common slopes. 
\smallskip

\item\label{integer slopes for induction step} the slopes of $\Ecal, \rankred{\Fcal}$ and $\Qcal$ are integers.
\end{enumerate}
By construction, the properties \ref{slope condition on E and Q for induction step}, \ref{no common slopes for induction step} and \ref{integer slopes for induction step} immediately follow from the corresponding assumptions on $\Ecal, \Fcal$ and $\Qcal$. The inequality $\rk(\Ecal^{\leq \mu}) \geq \rk(\rankred{\Fcal}^{\leq \mu})$ in \ref{slope condition on E and F for induction step} follows from the corresponding inequality $\rk(\Ecal^{\leq \mu}) \geq \rk(\Fcal^{\leq \mu})$ after observing
\[\rk(\rankred{\Fcal}^{\leq \mu}) = \begin{cases} \rk(\Fcal^{\leq \mu}) - 1 &\text { if } \mu \geq 0,\\ \rk(\Fcal^{\leq \mu}) & \text{ if } \mu < 0.\end{cases}\]
This observation further shows that equality in $\rk(\Ecal^{\leq \mu}) \geq \rk(\rankred{\Fcal}^{\leq \mu})$ never holds for $\mu \geq 0$. Moreover, for $\mu < 0$ we have $\rankred{\Fcal}^{\leq \mu} = 0$ by the fact $\mumin(\rankred{\Fcal}) \geq 0$. Hence we deduce that equality in $\rk(\Ecal^{\leq \mu}) \geq \rk(\rankred{\Fcal}^{\leq \mu})$ can hold only if $\Ecal^{\leq \mu} = \rankred{\Fcal}^{\leq \mu} = 0$, thereby verifying the property \ref{slope condition on E and F for induction step}. The remaining property \ref{slope condition on F and Q for induction step} is equivalent to slopewise dominance of $\rankred{\Fcal}$ on $\Qcal$ by Lemma \ref{slopewise dominance and rank inequalities}, and thus follows from the following observations:
\begin{enumerate}[label=(\alph*)]
\item\label{slopewise dominance of F on Q for induction step} $\Fcal$ slopewise dominates $\Qcal$ by Lemma \ref{slopewise dominance relations for key inequality}. 
\smallskip

\item\label{relation between HN polygons for induction step} $\HN(\rankred{\Fcal})$ is obtained from $\HN(\Fcal)$ by removing the line segment over the interval $(\rk(\Fcal)-1, \rk(\Fcal)]$ (and thereby having the same left endpoint). 
\smallskip

\item Since $\rk(\Fcal)>\rk(\Qcal)$, the removal process in \ref{relation between HN polygons for induction step} does not affect slopewise dominance. 
\end{enumerate}

Now since $\rk(\rankred{\Fcal}) - \rk(\Qcal) < \rk(\Fcal) - \rk(\Qcal)$, our induction hypothesis yields
\begin{equation}\label{step 1 codim inequality for E, F', Q}
c_{\Ecal, \rankred{\Fcal}}(\Qcal) \geq 0
\end{equation}
with equality if and only if $\Qcal \simeq \rankred{\Fcal}$. For the desired inequality $c_{\Ecal, \Fcal}(\Qcal) \geq 0$ we compute
\begin{align*}
\deg(\Ecal^\vee \otimes \Fcal)^\nonneg &= \deg(\Ecal^\vee \otimes (\rankred{\Fcal} \oplus \trivbundle))^\nonneg\\
&= \deg(\Ecal^\vee \otimes \rankred{\Fcal})^\nonneg + \deg(\Ecal^\vee \otimes \trivbundle)^\nonneg\\
&= \deg(\Ecal^\vee \otimes \rankred{\Fcal})^\nonneg + \deg(\Ecal^\vee)^\nonneg,\\
\deg(\Qcal^\vee \otimes \Fcal)^\nonneg &= \deg(\Qcal^\vee \otimes (\rankred{\Fcal} \oplus \trivbundle))^\nonneg\\
&= \deg(\Qcal^\vee \otimes \rankred{\Fcal})^\nonneg + \deg(\Qcal^\vee \otimes \trivbundle)^\nonneg\\
&= \deg(\Qcal^\vee \otimes \rankred{\Fcal})^\nonneg + \deg(\Qcal^\vee)^\nonneg.
\end{align*}
Then we have
\[c_{\Ecal, \Fcal}(\Qcal) = c_{\Ecal, \rankred{\Fcal}}(\Qcal) - \deg(\Ecal^\vee)^\nonneg + \deg(\Qcal^\vee)^\nonneg.\]
Since $\Ecal^\vee$ slopewise dominates $\Qcal^\vee$ as noted in Lemma \ref{slopewise dominance relations for key inequality}, we use Lemma \ref{nonnegative degree for slopewise dominant pairs} to find
\begin{equation}\label{step 1 inequality for codim E, F, Q and codim E, F', Q}
c_{\Ecal, \Fcal}(\Qcal) \geq c_{\Ecal, \rankred{\Fcal}}(\Qcal)
\end{equation}
with equality if and only if $\deg(\Ecal^\vee)^\nonneg = \deg(\Qcal^\vee)^\nonneg$ or equivalently $\deg(\Ecal)^\nonpos = \deg(\Qcal)^\nonpos$. Combining \eqref{step 1 codim inequality for E, F', Q} and \eqref{step 1 inequality for codim E, F, Q and codim E, F', Q} we obtain the desired inequality
\begin{equation}\label{step 1 codim inequality for E, F, Q}
c_{\Ecal, \Fcal}(\Qcal) \geq 0.
\end{equation}

It remains to check the equality condition for \eqref{step 1 codim inequality for E, F, Q}. Note that the equality condition $\Qcal = \Fcal$ as asserted in Proposition \ref{key inequality} never holds in our induction step as we have $\rk(\Fcal)>\rk(\Qcal)$. We thus have to prove that the inequality \eqref{step 1 codim inequality for E, F, Q} is strict. Suppose for contradiction that we have equality in \eqref{step 1 codim inequality for E, F, Q}. From the equality conditions for \eqref{step 1 codim inequality for E, F', Q} and \eqref{step 1 inequality for codim E, F, Q and codim E, F', Q} we get $\Qcal \simeq \rankred{\Fcal}$ and $\deg(\Ecal)^\nonpos = \deg(\Qcal)^\nonpos$. Since $\mumin(\rankred{\Fcal}) \geq 0$ by construction, the condition $\Qcal \simeq \rankred{\Fcal}$ implies $\deg(\Qcal)^\nonpos = 0$. Hence we must have $\deg(\Ecal)^\nonpos = 0$, which implies $\mumin(\Ecal) \geq 0$. Furthermore, since we assume $\mumin(\Fcal) = 0$, the assumption \ref{no common slopes for E and F} of Proposition \ref{key inequality} yields
\[ \mumin(\Ecal) > 0 = \mumin(\Fcal).\]
We thus find
\[ \rk(\Ecal^{\leq \mumin(\Fcal)}) = 0 < \rk(\Fcal^{\leq \mumin(\Fcal)}),\]
yielding a contradiction to the assumption \ref{slope condition on E and F} of Proposition \ref{key inequality} as desired. 
\end{proof}



We now proceed to the final reduction step. Here we aim to reduce the slopes of $\Fcal$ to the slopes of $\Qcal$ in a certain way that the quantity $c_{\Ecal, \Fcal}(\Qcal)$ can only decrease throughout the procedure. For a precise description of our procedure, we introduce the following construction:

\begin{defn}\label{max slope reduction definition}
Let $\Vcal$ and $\Wcal$ be nonzero vector bundles on $\adicff$ with integer slopes such that $\Vcal$ slopewise dominates $\Wcal$. Let $\maxslopered{\Vcal}$ be the vector bundle on $\adicff$ obtained from $\Vcal$ by reducing all slopes of $\Vcal^{>\mumax(\Wcal)}$ to $\mumax(\Wcal)$. More precisely, we set
\[\maxslopered{\Vcal} := \trivbundle(\mumax(\Wcal))^{\oplus \rk(\Vcal^{>\mumax(\Wcal)})} \oplus \Vcal^{\leq \mumax(\Wcal)}.\]
We say that $\maxslopered{\Vcal}$ is the \emph{maximal slope reduction} of $\Vcal$ to $\Wcal$.  
\end{defn}

\begin{figure}[H]
\begin{tikzpicture}	

		\coordinate (left) at (0, 0);
		\coordinate (q0) at (1,2);
		\coordinate (q1) at (2, 3);
		\coordinate (q2) at (3.5, 3.5);
		\coordinate (q3) at (5, 2.8);
		

		\coordinate (p0) at (1.5, 1.2);
		\coordinate (p1) at (3, 1);
		\coordinate (p2) at (4, 0.2);
				
		\draw[step=1cm,thick, color=red] (left) -- (q0) --  (q1);
		\draw[step=1cm,thick, color=blue] (q1) -- (q2) -- (q3);
		\draw[step=1cm,thick, color=black] (left) -- (p0) -- (p1) -- (p2);

		\draw [fill] (q0) circle [radius=0.05];		
		\draw [fill] (q1) circle [radius=0.05];		
		\draw [fill] (q2) circle [radius=0.05];		
		\draw [fill] (q3) circle [radius=0.05];		
		\draw [fill] (left) circle [radius=0.05];
		
		\draw [fill] (p0) circle [radius=0.05];		
		\draw [fill] (p1) circle [radius=0.05];		
		\draw [fill] (p2) circle [radius=0.05];		


		
		\path (q3) ++(0.2, -0.4) node {$\HN(\Vcal)$};
		\path (p2) ++(-0.2, -0.4) node {$\HN(\Wcal)$};
		\path (left) ++(-0.3, -0.05) node {$O$};

		\path (q0) ++(-0.5, 0.6) node {\color{red}$\Vcal^{>\mumax(\Wcal)}$};
		\path (q2) ++(0, 0.3) node {\color{blue}$\Vcal^{\leq \mumax(\Wcal)}$};

\end{tikzpicture}
\begin{tikzpicture}[scale=0.4]
        \pgfmathsetmacro{\textycoordinate}{5}
		\draw[->, line width=0.6pt] (0, \textycoordinate) -- (1.5,\textycoordinate);
		\draw (0,0) circle [radius=0.00];	
        \hspace{0.2cm}
\end{tikzpicture}
\begin{tikzpicture}
		\pgfmathsetmacro{\reducedycoordinate}{1.6}
	

		\coordinate (left) at (0, 0);
		\coordinate (q0) at (1,2);
		\coordinate (q1) at (2, 3);
		\coordinate (q2) at (3.5, 3.5);
		\coordinate (q3) at (5, 2.8);

		\coordinate (left) at (0, 0);
		\coordinate (q1') at (2, \reducedycoordinate);
		\coordinate (q2') at (3.5, \reducedycoordinate+0.5);
		\coordinate (q3') at (5, \reducedycoordinate-0.2);
		

		\coordinate (p0) at (1.5, 1.2);
		\coordinate (p1) at (3, 1);
		\coordinate (p2) at (4, 0.2);

		\draw[step=1cm,thick,dashed, color=red] (left) -- (q0) --  (q1);
		\draw[step=1cm,thick,dashed, color=blue] (q1) -- (q2) -- (q3);
		\draw[step=1cm,thick, color=red] (left) --(q1');
		\draw[step=1cm,thick, color=blue] (q1') -- (q2') -- (q3');
		\draw[step=1cm,thick] (p0) --  (p1) -- (p2);
		
		\draw [fill] (q0) circle [radius=0.05];		
		\draw [fill] (q1) circle [radius=0.05];		
		\draw [fill] (q2) circle [radius=0.05];		
		\draw [fill] (q3) circle [radius=0.05];		
		\draw [fill] (left) circle [radius=0.05];

		\draw [fill] (q1') circle [radius=0.05];			
		\draw [fill] (q2') circle [radius=0.05];		
		\draw [fill] (q3') circle [radius=0.05];		
		
		\draw [fill] (p0) circle [radius=0.05];		
		\draw [fill] (p1) circle [radius=0.05];		
		\draw [fill] (p2) circle [radius=0.05];		

		\draw[->, line width=0.6pt, color=red] (1, 1.5) -- (1,1);
		\draw[->, line width=0.6pt, color=blue] (3.5, 2.8) -- (3.5,2.3);
		


		
		\path (q3') ++(0.2, -0.4) node {$\HN(\maxslopered{\Vcal})$};
		\path (p2) ++(-0.2, -0.4) node {$\HN(\Wcal)$};
		\path (left) ++(-0.3, -0.05) node {$O$};

\end{tikzpicture}
\caption{Illustration of the maximal slope reduction}
\end{figure}

\begin{remark}
The assumption that $\Vcal$ and $\Wcal$ have integer slopes is crucial in Definition \ref{max slope reduction definition}. In fact, if $\mumax(\Wcal)$ is not an integer, reducing all slopes of $\Vcal^{>\mumax(\Wcal)}$ to $\mumax(\Wcal)$ may not make sense.
For example, if we consider $\Vcal = \trivbundle(1)^{\oplus 3}$ and $\Wcal = \trivbundle\left(\frac{1}{2}\right)$, reducing all slopes of $\Vcal$ to $\mumax(\Wcal) = \frac{1}{2}$ should yield a semistable vector bundle of slope $\frac{1}{2}$ and rank $3$, which does not exist. 

On the other hand, slopewise dominance of $\Vcal$ on $\Wcal$ is not essential for the definition to make sense. However, there are a couple of reasons that we don't consider the case when $\Vcal$ does not slopewise dominates $\Wcal$. First, our terminology doesn't quite make sense in this case as $\Vcal$ may have no slopes to reduce down to $\Wcal$, for example when $\mumax(\Vcal)<\mumin(\Wcal)$. Second, we won't need this case for our purpose; indeed, we will only apply the notion of maximal slope reduction to (some direct summands of) $\Fcal$ and $\Qcal$ for which we have a slopewise dominance relation given by Lemma \ref{slopewise dominance relations for key inequality}. 
\end{remark}

We note some basic properties of the maximal slope reduction. 

\begin{lemma}\label{basic properties of max slope reduction}
Let $\Vcal$ and $\Wcal$ be nonzero vector bundles on $\adicff$ with integer slopes such that $\Vcal$ slopewise dominates $\Wcal$. Let $\maxslopered{\Vcal}$ denote the maximal slope reduction of $\Vcal$ to $\Wcal$. Then $\maxslopered{\Vcal}$ satisfies the following properties:
\begin{enumerate}[label=(\roman*)]
\item $\mumax(\maxslopered{\Vcal}) = \mumax(\Wcal)$. 
\smallskip

\item $\rk(\maxslopered{\Vcal}) = \rk(\Vcal)$. 
\smallskip

\item $\Vcal = \maxslopered{\Vcal}$ if and only if $\mumax(\Vcal) = \mumax(\Wcal)$. 
\smallskip

\item $\maxslopered{\Vcal}$ slopewise dominates $\Wcal$. 
\smallskip

\item all slopes of $\maxslopered{\Vcal}$ are integers. 
\end{enumerate}
\end{lemma}
\begin{proof}
All properties are immediate consequences of Definition \ref{max slope reduction definition} 
\end{proof}

We can now recursively define our procedure for reducing the slopes of $\Fcal$ to the slopes of $\Qcal$ as follows:
\begin{enumerate}[label = (\Roman*)]
\item\label{slope reduction decomposition step} Since $\Fcal$ slopewise dominates $\Qcal$ as noted in Lemma \ref{slopewise dominance relations for key inequality}, we use Lemma \ref{existence of maximal common factor decomp} to obtain decompositions 
\begin{equation*}\label{max common factor decomps for F and Q}
\Fcal \simeq \Ucal \oplus \Fcal' \quad\quad \text{ and } \quad\quad \Qcal \simeq \Ucal \oplus \Qcal'
\end{equation*}
satisfying the following properties:
\smallskip
\begin{enumerate}[label=(\roman*)]
\item\label{slopewise dominance for F' and Q'} $\Fcal'$ slopewise dominates $\Qcal'$. 
\smallskip

\item\label{inequality for max slopes of F' and Q'} If $\Qcal' \neq 0$, we have $\mumax(\Fcal')>\mumax(\Qcal')$. 
\smallskip

\item\label{ineqaulities for min slope of U} If $\Ucal \neq 0$ and $\Qcal' \neq 0$, we have $\mumin(\Ucal) \geq \mumax(\Fcal') > \mumax(\Qcal')$.
\end{enumerate}
\smallskip

\item If $\Qcal' = 0$, we terminate the process. Otherwise, we go back to \ref{slope reduction decomposition step} after replacing $\Fcal$ by $\Ucal \oplus \maxslopered{\Fcal}'$, where $\maxslopered{\Fcal}'$ denotes the maximal slope reduction of $\Fcal'$ to $\Qcal'$.
\end{enumerate}

\begin{figure}[H]
\begin{tikzpicture}	
		\coordinate (left) at (0, 0);
		\coordinate (q1) at (2.5, 3.5);
		\coordinate (q2) at (4, 4);
		\coordinate (q3) at (5, 3.7);

		\coordinate (p0) at (0.5, 1.5);
		\coordinate (p1) at (1.5, 2.5);
		\coordinate (p2) at (3.5, 3);
		\coordinate (p3) at (4.5, 2.5);
		\coordinate (p4) at (5, 1.5);
				
		\draw[step=1cm,thick, color=red] (left) -- (p0) --  (p1);
		\draw[step=1cm,thick, color=blue] (p1) -- (q1) -- (q2) -- (q3);
		\draw[step=1cm,thick, color=green] (p1) -- (p2) -- (p3) -- (p4);

		\draw [fill] (q1) circle [radius=0.05];		
		\draw [fill] (q2) circle [radius=0.05];		
		\draw [fill] (q3) circle [radius=0.05];		
		\draw [fill] (left) circle [radius=0.05];
		
		\draw [fill] (p0) circle [radius=0.05];		
		\draw [fill] (p1) circle [radius=0.05];		
		\draw [fill] (p2) circle [radius=0.05];		
		\draw [fill] (p3) circle [radius=0.05];		
		\draw [fill] (p4) circle [radius=0.05];		


		
		\path (q3) ++(0.2, -0.4) node {$\HN(\Fcal)$};
		\path (p4) ++(-0.2, -0.4) node {$\HN(\Qcal)$};
		\path (left) ++(-0.3, -0.05) node {$O$};

		\path (p0) ++(-0.2, 0.3) node {\color{red}$\Ucal$};
		\path (q1) ++(0, 0.3) node {\color{blue}$\Fcal'$};
		\path (p2) ++(0, -0.4) node {\color{green}$\Qcal'$};

\end{tikzpicture}
\begin{tikzpicture}[scale=0.4]
        \pgfmathsetmacro{\textycoordinate}{6}
		\draw[->, line width=0.6pt] (0, \textycoordinate) -- (1.5,\textycoordinate);
		\draw (0,0) circle [radius=0.00];	
        \hspace{0.2cm}
\end{tikzpicture}
\begin{tikzpicture}
		\pgfmathsetmacro{\reducedycoordinate}{2.5/4+2.5}
	
		\coordinate (left) at (0, 0);
		\coordinate (q1) at (2.5, 3.5);
		\coordinate (q2) at (4, 4);
		\coordinate (q3) at (5, 3.7);

		\coordinate (q2') at (4, \reducedycoordinate);
		\coordinate (q3') at (5, \reducedycoordinate-0.3);

		\coordinate (p0) at (0.5, 1.5);
		\coordinate (p1) at (1.5, 2.5);
		\coordinate (p2) at (3.5, 3);
		\coordinate (p3) at (4.5, 2.5);
		\coordinate (p4) at (5, 1.5);
				
		\draw[step=1cm,thick, color=red] (left) -- (p0) --  (p1);
		\draw[step=1cm,thick,dashed, color=blue] (p1) -- (q1) -- (q2) -- (q3);
		\draw[step=1cm,thick, color=blue] (p1) -- (q2') -- (q3');
		\draw[step=1cm,thick, color=green] (p2) -- (p3) -- (p4);

		\draw [fill] (q1) circle [radius=0.05];		
		\draw [fill] (q2) circle [radius=0.05];		
		\draw [fill] (q3) circle [radius=0.05];		
		\draw [fill] (left) circle [radius=0.05];

		\draw [fill] (q2') circle [radius=0.05];		
		\draw [fill] (q3') circle [radius=0.05];	
		
		\draw [fill] (p0) circle [radius=0.05];		
		\draw [fill] (p1) circle [radius=0.05];		
		\draw [fill] (p2) circle [radius=0.05];		
		\draw [fill] (p3) circle [radius=0.05];		
		\draw [fill] (p4) circle [radius=0.05];	

		\draw[->, line width=0.6pt, color=blue] (3, 3.5) -- (3,3);	


		
		\path (q3') ++(0.6, -0.4) node {$\HN(\Fcal)$};
		\path (p4) ++(-0.2, -0.4) node {$\HN(\Qcal)$};
		\path (left) ++(-0.3, -0.05) node {$O$};

		\path (p0) ++(-0.2, 0.3) node {\color{red}$\Ucal$};
		\path (q2') ++(0, 0.3) node {\color{blue}$\maxslopered{\Fcal}'$};

\end{tikzpicture}
\caption{Illustration of the slope reduction process}
\end{figure}

With this procedure, we will obtain the desired inequality $c_{\Ecal, \Fcal}(\Qcal) \geq 0$ by establishing the following facts:
\begin{enumerate}[label=(\Alph*)]
\item\label{decreasing cEF(Q) during slope reduction} The quantity $c_{\Ecal, \Fcal}(\Qcal)$ never increases throughout the process.

\item\label{terminal condition for slope reduction} The process eventually terminates with the condition $\Qcal = \Fcal$ and $c_{\Ecal, \Fcal}(\Qcal) = 0$.
\end{enumerate}
For the equality condition, we will further show that
\begin{enumerate}[label=(\Alph*), start=3]
\item\label{strict increase after first slope reduction} $c_{\Ecal, \Fcal}(\Qcal)$ strictly decreases after the first cycle of the process,
\end{enumerate}
thereby deducing that the equality $c_{\Ecal, \Fcal}(\Qcal)= 0$ can be only achieved by starting with the terminal state $\Qcal = \Fcal$.

The main subtlety for our procedure arises from the fact that some of our assumptions on $\Ecal, \Fcal$ and $\Qcal$ may be lost during our process. In the following lemma, we give a list of all assumptions that are maintained throughout the entire process. 

\begin{lemma}\label{assumptions of key inequality after max slope reduction}
Let $\Ecal, \Fcal$ and $\Qcal$ be nonzero vector bundles on $\adicff$ with the following properties: 
\begin{enumerate}[label=(\roman*)]
\item\label{slope condition on E and F without equality condition} $\rk(\Ecal^{\leq \mu}) \geq \rk(\Fcal^{\leq \mu})$ for every $\mu \in \Q$ 
\smallskip

\item\label{slope condition on E and Q} $\rk(\Ecal^{\leq \mu}) \geq \rk(\Qcal^{\leq \mu})$ for every $\mu \in \Q$ with equality only when $\Ecal^{\leq \mu} \simeq \Qcal^{\leq \mu}$. 
\smallskip

\item\label{slope condition on F and Q} $\rk(\Fcal^{\geq \mu}) \geq \rk(\Qcal^{\geq \mu})$ for every $\mu \in \Q$. 
\smallskip

\item\label{integer slopes assumption for E, F and Q} all slopes of $\Ecal, \Fcal$ and $\Qcal$ are integers. 
\smallskip

\item\label{equal rank assumption for F and Q} $\rk(\Qcal) = \rk(\Fcal)$. 

\end{enumerate}
Then all properties \ref{slope condition on E and F without equality condition} - \ref{equal rank assumption for F and Q} are invariant under replacing $\Fcal$ by $\maxslopered{\Fcal}$, the maximal slope reduction of $\Fcal$ to $\Qcal$.






\end{lemma}

\begin{proof}
Let us first remark that the maximal slope reduction of $\Fcal$ to $\Qcal$ makes sense. Indeed, the property \ref{slope condition on F and Q} implies slopewise dominance of $\Fcal$ on $\Qcal$ by Lemma \ref{slopewise dominance relations for key inequality} while the property \ref{integer slopes assumption for E, F and Q} says that $\Fcal$ and $\Qcal$ have integer slopes. 

We now assert that the property \ref{slope condition on E and F without equality condition} is a formal consequence of the other properties. Note that the property \ref{slope condition on F and Q} is equivalent to slopewise dominance of $\Fcal$ on $\Qcal$ by Lemma \ref{slopewise dominance and rank inequalities}. Combining this with the property \ref{equal rank assumption for F and Q}, we obtain slopewise dominance of $\Qcal^\vee$ on $\Fcal^\vee$ by Lemma \ref{duality of slopewise dominance for equal rank case}. Hence Lemma \ref{slopewise dominance and rank inequalities} now yields an inequality
\[\rk(\Qcal^{\leq \mu}) \geq \rk(\Fcal^{\leq \mu}) \quad\quad \text{ for every } \mu \in \Q.\]
We then deduce the desired inequality $\rk(\Ecal^{\leq \mu}) \geq \rk(\Fcal^{\leq \mu})$ for every $\mu \in \Q$ by combining the above inequality with the inequality in the property \ref{slope condition on E and Q}. 

Hence we only need to check the invariance of the other properties \ref{slope condition on E and Q} - \ref{equal rank assumption for F and Q}. 
The invariance of the property \ref{slope condition on E and Q} is obvious since $\Ecal$ and $\Qcal$ remain unchanged. The property \ref{slope condition on F and Q} is equivalent to slopewise dominance of $\Fcal$ on $\Qcal$ by Lemma \ref{slopewise dominance and rank inequalities}, so its invariance under replacing $\Fcal$ by $\maxslopered{\Fcal}$ follows from Lemma \ref{basic properties of max slope reduction}. The invariance of the properties \ref{integer slopes assumption for E, F and Q} and \ref{equal rank assumption for F and Q} also follow immediately from Lemma \ref{basic properties of max slope reduction}.
\end{proof}

Lemma \ref{assumptions of key inequality after max slope reduction} suggests that during our process we may lose the following assumptions:
\begin{itemize}
\item $\Ecal$ and $\Fcal$ have no common slopes. 
\smallskip

\item the equality in $\rk(\Ecal^{\leq \mu}) \geq \rk(\Fcal^{\leq \mu})$ holds only when $\Ecal^{\leq \mu} \simeq \Fcal^{\leq \mu}$. 
\end{itemize}
Fortunately, losing either of these assumptions during our procedure will do no harm to our proof. In fact, these assumptions will be only necessary for establishing the fact that $c_{\Ecal, \Fcal}(\Qcal)$ strictly decreases after the first cycle of our procedure. In other words, our proof will be valid as long as we begin our procedure with all assumptions in Proposition \ref{key inequality}.



Let us now prove the key inequality for Step 3. 

\begin{prop}\label{decreasing c_EF(Q) after max reduction}
Let $\Ecal, \Fcal$ and $\Qcal$ be nonzero vector bundles on $\adicff$ with the following properties: 
\begin{enumerate}[label=(\roman*)]
\item\label{slope condition on E and F without equality condition, decreasing c_EF(Q) after max redunction} $\rk(\Ecal^{\leq \mu}) \geq \rk(\Fcal^{\leq \mu})$ for every $\mu \in \Q$ 
\smallskip

\item\label{slope condition on E and Q, decreasing c_EF(Q) after max redunction} $\rk(\Ecal^{\leq \mu}) \geq \rk(\Qcal^{\leq \mu})$ for every $\mu \in \Q$ with equality only when $\Ecal^{\leq \mu} \simeq \Qcal^{\leq \mu}$. 
\smallskip

\item\label{slope condition on F and Q, decreasing c_EF(Q) after max redunction} $\rk(\Fcal^{\geq \mu}) \geq \rk(\Qcal^{\geq \mu})$ for every $\mu \in \Q$.
\smallskip 

\item\label{integer slopes assumption for E, F and Q, decreasing c_EF(Q) after max redunction} all slopes of $\Ecal, \Fcal$ and $\Qcal$ are integers. 
\smallskip

\item\label{equal rank assumption for F and Q, decreasing c_EF(Q) after max redunction} $\rk(\Qcal) = \rk(\Fcal)$. 

\end{enumerate}
Let $\maxslopered{\Fcal}$ be the maximal slope reduction of $\Fcal$ to $\Qcal$. Then we have an inequality
\begin{equation}\label{max slope reduction inequality for cEF(Q)}
c_{\Ecal, \Fcal}(\Qcal) \geq c_{\Ecal, \maxslopered{\Fcal}}(\Qcal). 
\end{equation}
Moreover, equality in \eqref{max slope reduction inequality for cEF(Q)} never holds if $\Ecal, \Fcal$ and $\Qcal$ satisfy the following additional properties:
\begin{enumerate}[label=(\roman*), start = 6]
\item\label{no common slopes for E and F, decreasing c_EF(Q) after max reduction} $\Ecal$ and $\Fcal$ have no common slopes. 
\smallskip

\item\label{equality condition for rank inequality on E and F, decreasing c_EF(Q) after max reduction} For every $\mu \in \Q$, an equality $\rk(\Ecal^{\leq \mu}) = \rk(\Fcal^{\leq \mu})$ holds only when $\Ecal^{\leq \mu} \simeq \Fcal^{\leq \mu}$. 
\smallskip

\item\label{max slope condition for F and Q, decreasing c_EF(Q) after max reduction} $\mumax(\Fcal)>\mumax(\Qcal)$
\end{enumerate}
\end{prop}

\begin{proof}
Set $\lambda = \mumax(\Qcal)$ and $r = \rk(\Fcal^{>\lambda})$. By definition, we may write
\begin{equation}\label{max slope decomps for max slope reduction} 
\Fcal = \Fcal^{> \lambda} \oplus \Fcal^{\leq \lambda} \quad\quad \text{ and } \quad\quad \maxslopered{\Fcal} = \trivbundle(\lambda)^{\oplus r} \oplus \Fcal^{\leq \lambda}.
\end{equation}
Then we have
\begin{align*}
\deg(\Ecal^\vee \otimes \Fcal)^\nonneg &= \deg(\Ecal^\vee \otimes (\Fcal^{> \lambda} \oplus \Fcal^{\leq \lambda}))^\nonneg\\
&= \deg(\Ecal^\vee \otimes \Fcal^{>\lambda})^\nonneg + \deg(\Ecal^\vee \otimes \Fcal^{\leq \lambda})^\nonneg,\\
\deg(\Ecal^\vee \otimes \maxslopered{\Fcal})^\nonneg &= \deg(\Ecal^\vee \otimes (\trivbundle(\lambda)^{\oplus r} \oplus \Fcal^{\leq \lambda}))^\nonneg\\
&= \deg(\Ecal^\vee \otimes \trivbundle(\lambda)^{\oplus r})^\nonneg + \deg(\Ecal^\vee \otimes \Fcal^{\leq \lambda})^\nonneg,\\
\deg(\Qcal^\vee \otimes \Fcal)^\nonneg &= \deg(\Qcal^\vee \otimes (\Fcal^{> \lambda} \oplus \Fcal^{\leq \lambda}))^\nonneg\\
&= \deg(\Qcal^\vee \otimes \Fcal^{> \lambda})^\nonneg + \deg(\Qcal^\vee \otimes \Fcal^{\leq \lambda})^\nonneg, \\
\deg(\Qcal^\vee \otimes \maxslopered{\Fcal})^\nonneg &= \deg(\Qcal^\vee \otimes (\trivbundle(\lambda)^{\oplus r} \oplus \Fcal^{\leq \lambda}))^\nonneg\\
&= \deg(\Qcal^\vee \otimes \trivbundle(\lambda)^{\oplus r})^\nonneg + \deg(\Qcal^\vee \otimes \Fcal^{\leq \lambda})^\nonneg. 
\end{align*}
Thus we obtain
\begin{align}
\deg(\Ecal^\vee \otimes \Fcal)^\nonneg - \deg(\Ecal^\vee \otimes \maxslopered{\Fcal})^\nonneg &= \deg(\Ecal^\vee \otimes \Fcal^{>\lambda})^\nonneg - \deg(\Ecal^\vee \otimes \trivbundle(\lambda)^{\oplus r})^\nonneg, \label{max slope reduction difference of e,f terms}\\
\deg(\Qcal^\vee \otimes \Fcal)^\nonneg - \deg(\Qcal^\vee \otimes \maxslopered{\Fcal})^\nonneg &= \deg(\Qcal^\vee \otimes \Fcal^{> \lambda})^\nonneg - \deg(\Qcal^\vee \otimes \trivbundle(\lambda)^{\oplus r})^\nonneg. \label{max slope reduction difference of f,q terms}
\end{align}

Let us write $\HNvec(\Ecal) = (e_i), \HNvec(\Fcal^{> \lambda}) = (f_j), \HNvec(\Qcal) = (q_k), \HNvec(\trivbundle(\lambda)^{\oplus r}) = (\maxslopered{f})$, and set $f := \sum f_j$. Note that we can write $\maxslopered{f} = \sum \maxslopered{f}_j$ where $\maxslopered{f}_j$ denotes the vector obtained by reducing the slope of $f_j$ to $\lambda$.
By construction, we have the following observations:
\begin{enumerate}[label=(\arabic*)]
\item\label{equal x coordinates for max slope reduction HN vectors} $f_x = \rk(\Fcal^{> \lambda}) = r = \maxslopered{f}_x$
\smallskip

\item\label{comparison of y coordinates for max slope reduction HN vectors} $f_y \geq \maxslopered{f}_y$ with equality if and only if $f = \maxslopered{f} = 0$. 
\end{enumerate}

We now use Lemma \ref{degree in terms of HN vectors} to write the right side of \eqref{max slope reduction difference of e,f terms} as
\begin{equation}\label{max slope reduction difference of e,f terms HN vectors}
\deg(\Ecal^\vee \otimes \Fcal^{>\lambda})^\nonneg - \deg(\Ecal^\vee \otimes \trivbundle(\lambda)^{\oplus r})^\nonneg 
= \sum_{e_i \preceq f_j} e_i \times f_j - \sum_{\mu(e_i) \leq \lambda} e_i \times \maxslopered{f}.
\end{equation}
Note that each $e_i$ with $\mu(e_i) \leq \lambda$ satisfies $e_i \preceq f_j$ for all $j$ since by construction we have $\mu(f_j) > \lambda$ for all $j$. We then find
\begin{equation}\label{max slope reduction e,f terms fundamental inequality}
\sum_{\mu(e_i) \leq \lambda} e_i \times f_j \leq \sum_{e_i \preceq f_j} e_i \times f_j
\end{equation}
as each term on the right hand side is nonnegative. Now \eqref{max slope reduction difference of e,f terms HN vectors} yields
\begin{align*}
\deg(\Ecal^\vee \otimes \Fcal^{>\lambda})^\nonneg - \deg(\Ecal^\vee \otimes \trivbundle(\lambda)^{\oplus r})^\nonneg &\geq  \sum_{\mu(e_i) \leq \lambda} e_i \times f_j - \sum_{\mu(e_i) \leq \lambda} e_i \times \maxslopered{f}\\
&= \sum_{\mu(e_i) \leq \lambda} e_i \times \sum f_j  - \sum_{\mu(e_i) \leq \lambda} e_i \times \maxslopered{f} \\
&= \sum_{\mu(e_i) \leq \lambda} e_i \times (f - \maxslopered{f}).
\end{align*}
Since $(f - \maxslopered{f})_x = 0$ as noted in \ref{equal x coordinates for max slope reduction HN vectors}, we have
\[ \sum_{\mu(e_i) \leq \lambda} e_i \times (f - \maxslopered{f}) = \left( \sum_{\mu(e_i) \leq \lambda} e_i\right)_x \cdot (f - \maxslopered{f})_y = \rk(\Ecal^{\leq \lambda}) \cdot (f - \maxslopered{f})_y.\]
We thus obtain an inequality
\begin{equation*}
\deg(\Ecal^\vee \otimes \Fcal^{>\lambda})^\nonneg - \deg(\Ecal^\vee \otimes \trivbundle(\lambda)^{\oplus r})^\nonneg \geq \rk(\Ecal^{\leq \lambda}) \cdot (f - \maxslopered{f})_y
\end{equation*}
which is equivalent by \eqref{max slope reduction difference of e,f terms} to an inequality
\begin{equation}\label{max slope reduction difference of e,f terms lower bound}
\deg(\Ecal^\vee \otimes \Fcal)^\nonneg - \deg(\Ecal^\vee \otimes \maxslopered{\Fcal})^\nonneg \geq \rk(\Ecal^{\leq \lambda}) \cdot (f - \maxslopered{f})_y.
\end{equation}

Let us now use Lemma \ref{degree in terms of HN vectors} to write the right side of \eqref{max slope reduction difference of f,q terms} as
\begin{equation*}\label{max slope reduction difference of f,q terms HN vectors}
\deg(\Qcal^\vee \otimes \Fcal^{> \lambda})^\nonneg - \deg(\Qcal^\vee \otimes \trivbundle(\lambda)^{\oplus r})^\nonneg = \sum_{q_k \preceq f_j} q_k \times f_j - \sum_{\mu(q_k) \leq \lambda} q_k \times \maxslopered{f}.
\end{equation*}
Note that the conditions $q_k \preceq f_j$ and $\mu(q_k) \leq \lambda$ hold for all $j$ and $k$; indeed, by construction we have $\mu(q_k) \leq \mumax(\Qcal) = \lambda < \mu(f_j)$ for all $j$ and $k$. Hence we can simplify the above equation as
\begin{align*}
\deg(\Qcal^\vee \otimes \Fcal^{> \lambda})^\nonneg - \deg(\Qcal^\vee \otimes \trivbundle(\lambda)^{\oplus r})^\nonneg &= \sum q_k \times f_j - \sum q_k \times \maxslopered{f} \\
&= \sum q_k \times \sum f_j - \sum q_k \times \maxslopered{f}\\
&= \sum q_k \times (f - \maxslopered{f}).
\end{align*}
Now, as in the previous paragraph, we use the fact $(f_j - f'_j)_x = 0$ from \ref{equal x coordinates for max slope reduction HN vectors} to write
\[ \sum q_k \times (f - \maxslopered{f}) = \left(\sum q_k\right)_x - (f - \maxslopered{f})_y = \rk(\Qcal) \cdot (f - \maxslopered{f})_y\]
and consequently obtain an equation
\[\deg(\Qcal^\vee \otimes \Fcal^{> \lambda})^\nonneg - \deg(\Qcal^\vee \otimes \trivbundle(\lambda)^{\oplus r})^\nonneg = \rk(\Qcal) \cdot (f - \maxslopered{f})_y.\]
By \eqref{max slope reduction difference of f,q terms}, this equation is equivalent to 
\begin{equation}\label{max slope reduction difference of f,q terms simple form}
\deg(\Qcal^\vee \otimes \Fcal)^\nonneg - \deg(\Qcal^\vee \otimes \maxslopered{\Fcal})^\nonneg = \rk(\Qcal) \cdot (f - \maxslopered{f})_y.
\end{equation}

Note that we have
\[c_{\Ecal, \Fcal}(\Qcal) - c_{\Ecal, \maxslopered{\Fcal}}(\Qcal) = \big(\deg(\Ecal^\vee \otimes \Fcal)^\nonneg - \deg(\Ecal^\vee \otimes \maxslopered{\Fcal})^\nonneg\big) - \big(\deg(\Qcal^\vee \otimes \Fcal)^\nonneg - \deg(\Qcal^\vee \otimes \maxslopered{\Fcal})^\nonneg\big).\]
Hence \eqref{max slope reduction difference of e,f terms lower bound} and \eqref{max slope reduction difference of f,q terms simple form} together yields an inequality
\begin{equation}\label{max slope reduction difference of cEF(Q) lower bound}
c_{\Ecal, \Fcal}(\Qcal) - c_{\Ecal, \maxslopered{\Fcal}}(\Qcal)  \geq (\rk(\Ecal^{\leq \lambda}) - \rk(\Qcal)) \cdot (f - \maxslopered{f})_y.
\end{equation}
Now we observe $\Qcal = \Qcal^{\leq \mumax(\Qcal)} = \Qcal^{\leq \lambda}$ and find
\[\rk(\Ecal^{\leq \lambda}) - \rk(\Qcal) = \rk(\Ecal^{\leq \lambda}) - \rk(\Qcal^{\leq \lambda}) \geq 0\]
where the inequality follows from the assumption \ref{slope condition on E and Q, decreasing c_EF(Q) after max redunction}. 
Since we also have $(f - \maxslopered{f})_y \geq 0$ as noted in \ref{comparison of y coordinates for max slope reduction HN vectors}, we obtain
\begin{equation}\label{max slope redunction difference of cEF(Q) lower bound nonnegativity}
 (\rk(\Ecal^{\leq \lambda}) - \rk(\Qcal)) \cdot (f - \maxslopered{f})_y \geq 0
\end{equation}
We thus deduce the desired inequality \eqref{max slope reduction inequality for cEF(Q)}
from \eqref{max slope reduction difference of cEF(Q) lower bound} and \eqref{max slope redunction difference of cEF(Q) lower bound nonnegativity}.

It remains to prove the last statement of Proposition \ref{decreasing c_EF(Q) after max reduction}. For the rest of the proof, we therefore assume that $\Ecal, \Fcal$ and $\Qcal$ satisfy the properties \ref{no common slopes for E and F, decreasing c_EF(Q) after max reduction}, \ref{equality condition for rank inequality on E and F, decreasing c_EF(Q) after max reduction} and \ref{max slope condition for F and Q, decreasing c_EF(Q) after max reduction}. We also suppose for contradiction that equality in \eqref{max slope reduction inequality for cEF(Q)} holds. Then we must have equality in both \eqref{max slope reduction difference of cEF(Q) lower bound} and \eqref{max slope redunction difference of cEF(Q) lower bound nonnegativity}. The equality in \eqref{max slope redunction difference of cEF(Q) lower bound nonnegativity} gives us two cases to consider, namely
\begin{enumerate}[label=(\alph*)]
\item\label{case of equal ranks upto max slope} when $\rk(\Ecal^{\leq \lambda}) = \rk(\Qcal)$, 

\item\label{case of trivial max slope reduction} when $(f-\maxslopered{f})_y = 0$.
\end{enumerate}

We first investigate the case \ref{case of trivial max slope reduction}. The defining condition $(f-\maxslopered{f})_y = 0$ yields $f = 0$ by \ref{comparison of y coordinates for max slope reduction HN vectors}, which implies $\rk(\Fcal^{> \lambda}) = r = 0$ by \ref{equal x coordinates for max slope reduction HN vectors}. The decompositions \eqref{max slope decomps for max slope reduction} then yield $\Fcal = \maxslopered{\Fcal}$, implying $\mumax(\Fcal) = \mumax(\Qcal)$ by Lemma \ref{basic properties of max slope reduction}. 
We thus have a contradiction to the property \ref{max slope condition for F and Q, decreasing c_EF(Q) after max reduction}.

Let us now consider the case \ref{case of equal ranks upto max slope}. We may assume that $\Fcal^{> \lambda} \neq 0$, since otherwise we can argue as in the preceding paragraph to obtain a contradiction. 
The equality in \eqref{max slope reduction difference of cEF(Q) lower bound} implies equality in \eqref{max slope reduction difference of e,f terms lower bound}, and consequently equality in \eqref{max slope reduction e,f terms fundamental inequality}. Then every term on the right side of \eqref{max slope reduction e,f terms fundamental inequality} appears on the left side of \eqref{max slope reduction e,f terms fundamental inequality}, since every term in \eqref{max slope reduction e,f terms fundamental inequality} is positive by the property \ref{no common slopes for E and F, decreasing c_EF(Q) after max reduction}. Therefore every $e_i$ that satisfies $e_i \preceq f_j$ for some $j$ also satisfies $\mu(e_i) \leq \lambda$. In particular, we obtain $\mu(e_i) \leq \lambda$ for all $e_i$ with $e_i \preceq f_1$. Since $\mu(f_1) = \mumax(\Fcal^{>\lambda}) = \mumax(\Fcal)$, 
we deduce
\begin{equation}\label{E has no intermediate slopes for max slopep reduction}
\Ecal^{\leq \mumax(\Fcal)} = \Ecal^{\leq \lambda}.
\end{equation}
We then use the defining condition $\rk(\Ecal^{\leq \lambda}) = \rk(\Qcal)$ and the assumption \ref{equal rank assumption for F and Q, decreasing c_EF(Q) after max redunction} to find
\[ \rk(\Ecal^{\leq \mumax(\Fcal)}) = \rk(\Ecal^{\leq \lambda}) = \rk(\Qcal) = \rk(\Fcal) = \rk(\Fcal^{\leq \mumax(\Fcal)}),\]
which yields an isomorphism
\begin{equation}\label{isomorphism between E and F up to max slope of F}
\Ecal^{\leq \mumax(\Fcal)} \simeq \Fcal^{\leq \mumax(\Fcal)} = \Fcal.
\end{equation}
by the property \ref{equality condition for rank inequality on E and F, decreasing c_EF(Q) after max reduction}. Moreover, we observe $\Qcal = \Qcal^{\leq \mumax(\Qcal)} = \Qcal^{\leq \lambda}$ to rewrite the defining condition $\rk(\Ecal^{\leq \lambda}) = \rk(\Qcal)$ as $\rk(\Ecal^{\leq \lambda}) = \rk(\Qcal^{\leq \lambda})$, thereby obtaining another isomorphism 
\begin{equation}\label{isomorphism between E and Q up to max slope of Q}
\Ecal^{\leq \lambda} \simeq \Qcal^{\leq \lambda} = \Qcal
\end{equation}
by the assumption \ref{slope condition on E and Q, decreasing c_EF(Q) after max redunction}. Now \eqref{E has no intermediate slopes for max slopep reduction}, \eqref{isomorphism between E and F up to max slope of F} and \eqref{isomorphism between E and Q up to max slope of Q} together imply that $\Fcal \simeq \Qcal$, thereby yielding a contradiction to the property \ref{max slope condition for F and Q, decreasing c_EF(Q) after max reduction}. 
\end{proof}

The following proposition translates the results of Lemma \ref{assumptions of key inequality after max slope reduction} and Proposition \ref{decreasing c_EF(Q) after max reduction} in the setting of our slope reduction procedure. 

\begin{prop}\label{slope reduction single step}
Let $\Ecal, \Fcal$ and $\Qcal$ be nonzero vector bundles on $\adicff$ with the following properties: 
\begin{enumerate}[label=(\roman*)]
\item\label{slope condition on E and F without equality condition, slope reduction single step} $\rk(\Ecal^{\leq \mu}) \geq \rk(\Fcal^{\leq \mu})$ for every $\mu \in \Q$ 
\smallskip

\item\label{slope condition on E and Q, slope reduction single step} $\rk(\Ecal^{\leq \mu}) \geq \rk(\Qcal^{\leq \mu})$ for every $\mu \in \Q$ with equality only when $\Ecal^{\leq \mu} \simeq \Qcal^{\leq \mu}$. 
\smallskip

\item\label{slope condition on F and Q, slope reduction single step} $\rk(\Fcal^{\geq \mu}) \geq \rk(\Qcal^{\geq \mu})$ for every $\mu \in \Q$. 
\smallskip

\item\label{integer slopes assumption for E, F and Q, slope reduction single step} all slopes of $\Ecal, \Fcal$ and $\Qcal$ are integers. 
\smallskip

\item\label{equal rank assumption for F and Q, slope reduction single step} $\rk(\Qcal) = \rk(\Fcal)$. 

\end{enumerate}
Consider the decompositions 
\begin{equation}\label{max common factor decomps for F and Q in slope reduction step}
\Fcal \simeq \Ucal \oplus \Fcal' \quad\quad \text{ and } \quad\quad \Qcal \simeq \Ucal \oplus \Qcal'
\end{equation}
given by Lemma \ref{existence of maximal common factor decomp}. Assume that $\Qcal' \neq 0$, and let $\maxslopered{\Fcal}'$ denote the maximal slope reduction of $\Fcal'$ to $\Qcal'$.

\begin{enumerate}[label=(\arabic*)]
\item\label{invariance of assumptions in slope reduction steps} The properties \ref{slope condition on E and F without equality condition, slope reduction single step} - \ref{equal rank assumption for F and Q, slope reduction single step} are invariant under replacing $\Fcal$ by $\genslopered{\Fcal}:=\Ucal \oplus \maxslopered{\Fcal}'$. 
\smallskip

\item\label{statement of decreasing cEF(Q) in slope reduction steps} We have an inequality
\begin{equation}\label{slope reduction single step inequality for cEF(Q)}
c_{\Ecal, \Fcal}(\Qcal) \geq c_{\Ecal, \genslopered{\Fcal}}(\Qcal).
\end{equation}

\item\label{equality condition for cEF(Q) in slope reduction steps} 
The inequality \eqref{slope reduction single step inequality for cEF(Q)} becomes strict if $\Ecal$ and $\Fcal$ satisfy the following additional properties:
\begin{enumerate}[label = (\roman*), start = 6]
\smallskip
\item\label{no common slopes for E and F, slope reduction single step} $\Ecal$ and $\Fcal$ have no common slopes. 
\smallskip

\item\label{equality condition for rank inequality on E and F, slope reduction single step} For every $\mu \in \Q$, an equality $\rk(\Ecal^{\leq \mu}) = \rk(\Fcal^{\leq \mu})$ holds only when $\Ecal^{\leq \mu} \simeq \Fcal^{\leq \mu}$. 
\end{enumerate}
\end{enumerate}
\end{prop}

\begin{figure}[H]
\begin{tikzpicture}	
		\coordinate (left) at (0, 0);
		\coordinate (q1) at (2.5, 3.5);
		\coordinate (q2) at (4, 4);
		\coordinate (q3) at (5, 3.7);

		\coordinate (p0) at (0.5, 1.5);
		\coordinate (p1) at (1.5, 2.5);
		\coordinate (p2) at (3.5, 3);
		\coordinate (p3) at (4.5, 2.5);
		\coordinate (p4) at (5, 1.5);
				
		\draw[step=1cm,thick, color=red] (left) -- (p0) --  (p1);
		\draw[step=1cm,thick, color=blue] (p1) -- (q1) -- (q2) -- (q3);
		\draw[step=1cm,thick, color=green] (p1) -- (p2) -- (p3) -- (p4);

		\draw [fill] (q1) circle [radius=0.05];		
		\draw [fill] (q2) circle [radius=0.05];		
		\draw [fill] (q3) circle [radius=0.05];		
		\draw [fill] (left) circle [radius=0.05];
		
		\draw [fill] (p0) circle [radius=0.05];		
		\draw [fill] (p1) circle [radius=0.05];		
		\draw [fill] (p2) circle [radius=0.05];		
		\draw [fill] (p3) circle [radius=0.05];		
		\draw [fill] (p4) circle [radius=0.05];		


		
		\path (q3) ++(0.2, -0.4) node {$\HN(\Fcal)$};
		\path (p4) ++(-0.2, -0.4) node {$\HN(\Qcal)$};
		\path (left) ++(-0.3, -0.05) node {$O$};

		\path (p0) ++(-0.2, 0.3) node {\color{red}$\Ucal$};
		\path (q1) ++(0, 0.3) node {\color{blue}$\Fcal'$};
		\path (p2) ++(0, -0.4) node {\color{green}$\Qcal'$};

\end{tikzpicture}
\begin{tikzpicture}[scale=0.4]
        \pgfmathsetmacro{\textycoordinate}{6}
		\draw[->, line width=0.6pt] (0, \textycoordinate) -- (1.5,\textycoordinate);
		\draw (0,0) circle [radius=0.00];	
        \hspace{0.2cm}
\end{tikzpicture}
\begin{tikzpicture}
		\pgfmathsetmacro{\reducedycoordinate}{2.5/4+2.5}
	
		\coordinate (left) at (0, 0);
		\coordinate (q1) at (2.5, 3.5);
		\coordinate (q2) at (4, 4);
		\coordinate (q3) at (5, 3.7);

		\coordinate (q2') at (4, \reducedycoordinate);
		\coordinate (q3') at (5, \reducedycoordinate-0.3);

		\coordinate (p0) at (0.5, 1.5);
		\coordinate (p1) at (1.5, 2.5);
		\coordinate (p2) at (3.5, 3);
		\coordinate (p3) at (4.5, 2.5);
		\coordinate (p4) at (5, 1.5);
				
		\draw[step=1cm,thick, color=red] (left) -- (p0) --  (p1);
		\draw[step=1cm,thick,dashed, color=blue] (p1) -- (q1) -- (q2) -- (q3);
		\draw[step=1cm,thick, color=blue] (p1) -- (q2') -- (q3');
		\draw[step=1cm,thick, color=green] (p2) -- (p3) -- (p4);

		\draw [fill] (q1) circle [radius=0.05];		
		\draw [fill] (q2) circle [radius=0.05];		
		\draw [fill] (q3) circle [radius=0.05];		
		\draw [fill] (left) circle [radius=0.05];

		\draw [fill] (q2') circle [radius=0.05];		
		\draw [fill] (q3') circle [radius=0.05];	
		
		\draw [fill] (p0) circle [radius=0.05];		
		\draw [fill] (p1) circle [radius=0.05];		
		\draw [fill] (p2) circle [radius=0.05];		
		\draw [fill] (p3) circle [radius=0.05];		
		\draw [fill] (p4) circle [radius=0.05];	

		\draw[->, line width=0.6pt, color=blue] (3, 3.5) -- (3,3);	


		
		\path (q3') ++(0.6, -0.4) node {$\HN(\genslopered{\Fcal})$};
		\path (p4) ++(-0.2, -0.4) node {$\HN(\Qcal)$};
		\path (left) ++(-0.3, -0.05) node {$O$};

		\path (p0) ++(-0.2, 0.3) node {\color{red}$\Ucal$};
		\path (q2') ++(0, 0.3) node {\color{blue}$\maxslopered{\Fcal}'$};

\end{tikzpicture}
\caption{Illustration of the constructions in Proposition \ref{slope reduction single step}}
\end{figure}

\begin{proof}
Let us first verify that all constructions in the statement are valid. The validity of the decompositions \eqref{max common factor decomps for F and Q in slope reduction step} relies on slopewise dominance of $\Fcal$ on $\Qcal$, which follows from the property \ref{slope condition on F and Q, slope reduction single step} by Lemma \ref{slopewise dominance and rank inequalities}. For the validity of the maximal slope reduction of $\Fcal'$ to $\Qcal'$, we verify slopewise dominance of $\Fcal'$ on $\Qcal'$ by Lemma \ref{existence of maximal common factor decomp} and integer slopes of $\Fcal'$ and $\Qcal'$ by the property  \ref{integer slopes assumption for E, F and Q, slope reduction single step}.

We assert that the properties \ref{slope condition on E and F without equality condition, slope reduction single step} - \ref{equal rank assumption for F and Q, slope reduction single step} yield the corresponding properties for $\Ecal, \Fcal'$ and $\Qcal'$ as follows:
\begin{enumerate}[label=(\roman*)']
\item\label{slope condition on E and F' without equality condition} $\rk(\Ecal^{\leq \mu}) \geq \rk(\Fcal'^{\leq \mu})$ for every $\mu \in \Q$.
\smallskip

\item\label{slope condition on E and Q'} $\rk(\Ecal^{\leq \mu}) \geq \rk(\Qcal'^{\leq \mu})$ for every $\mu \in \Q$ with equality only when $\Ecal^{\leq \mu} \simeq \Qcal'^{\leq \mu}$. 
\smallskip

\item\label{slope condition on F' and Q'} $\rk(\Fcal'^{\geq \mu}) \geq \rk(\Qcal'^{\geq \mu})$ for every $\mu \in \Q$. 
\smallskip

\item\label{integer slopes assumption for E, F' and Q'} all slopes of $\Ecal, \Fcal'$ and $\Qcal'$ are integers. 
\smallskip

\item\label{equal rank assumption for F' and Q'} $\rk(\Qcal') = \rk(\Fcal')$. 

\end{enumerate}
We only need to check the properties \ref{slope condition on E and Q'} - \ref{equal rank assumption for F' and Q'} since the property \ref{slope condition on E and F' without equality condition} will then follow as a formal consequence of these properties as in the proof of Lemma \ref{assumptions of key inequality after max slope reduction}. The property \ref{slope condition on E and Q'} follows from the property \ref{slope condition on E and Q, slope reduction single step} since $\Qcal' = \Qcal^{\leq \lambda}$ with $\lambda = \mumax(\Qcal')$. The property \ref{slope condition on F' and Q'} is equivalent to slopewise dominance of $\Fcal'$ on $\Qcal'$ which follows from Lemma \ref{existence of maximal common factor decomp}. The properties \ref{integer slopes assumption for E, F' and Q'} and \ref{equal rank assumption for F' and Q'} follow immediately from the corresponding properties \ref{integer slopes assumption for E, F and Q, slope reduction single step} and \ref{equal rank assumption for F and Q, slope reduction single step} by construction. 

With the properties \ref{slope condition on E and F' without equality condition} - \ref{equal rank assumption for F' and Q'} established, Lemma \ref{assumptions of key inequality after max slope reduction} and Proposition \ref{decreasing c_EF(Q) after max reduction} now yield the following facts:
\begin{enumerate}[label=(\arabic*)']
\item\label{invariance of assumptions for F' and Q' in slope reduction steps} The properties \ref{slope condition on E and F' without equality condition} - \ref{equal rank assumption for F' and Q'} are invariant under replacing $\Fcal'$ by $\maxslopered{\Fcal}'$. 

\item\label{statement of decreasing cEF'(Q') in slope reduction steps} We have an inequality
\begin{equation}\label{slope reduction step inequality for cEF'(Q')}
c_{\Ecal, \Fcal'}(\Qcal') \geq c_{\Ecal, \maxslopered{\Fcal}'}(\Qcal').
\end{equation}

\item\label{equality condition for cEF'(Q') in slope reduction steps} 
The inequality \eqref{slope reduction step inequality for cEF'(Q')} becomes strict if $\Ecal, \Fcal'$ and $\Qcal'$ satisfy the following additional properties:
\smallskip
\begin{enumerate}[label=(\roman*)', start = 6]
\item\label{no common slopes for E and F'} $\Ecal$ and $\Fcal'$ have no common slopes. 
\smallskip

\item\label{equality condition for rank inequality on E and F'} For every $\mu \in \Q$, an equality $\rk(\Ecal^{\leq \mu}) = \rk(\Fcal'^{\leq \mu})$ holds only when $\Ecal^{\leq \mu} \simeq \Fcal^{\leq \mu}$. 
\smallskip

\item\label{max slope condition on F' and Q'} $\mumax(\Fcal') > \mumax(\Qcal')$. 
\end{enumerate}
\end{enumerate}
We wish to deduce the statements \ref{invariance of assumptions in slope reduction steps}, \ref{statement of decreasing cEF(Q) in slope reduction steps} and \ref{equality condition for cEF(Q) in slope reduction steps} respectively from the above facts \ref{invariance of assumptions for F' and Q' in slope reduction steps}, \ref{statement of decreasing cEF'(Q') in slope reduction steps} and \ref{equality condition for cEF'(Q') in slope reduction steps}.

Let us now prove the statement \ref{invariance of assumptions in slope reduction steps}. Note that, as in the proof of Lemma \ref{assumptions of key inequality after max slope reduction}, we only need to show the invariance of the properties \ref{slope condition on E and Q, slope reduction single step} - \ref{equal rank assumption for F and Q, slope reduction single step}.
The invariance of the property \ref{slope condition on E and Q, slope reduction single step} is evident since $\Ecal$ and $\Qcal$ remain unchanged. For the invariance of the remaining properties \ref{slope condition on F and Q, slope reduction single step}, \ref{integer slopes assumption for E, F and Q, slope reduction single step} and \ref{equal rank assumption for F and Q, slope reduction single step}, we have to show that $\Ecal, \genslopered{\Fcal}$ and $\Qcal$ satisfy the following properties: 
\begin{enumerate}[label=$\genslopered{\text{(\roman*)}}$, start=3]
\item\label{slope condition on F and Q after slope reduction} $\rk(\genslopered{\Fcal}^{\geq \mu}) \geq \rk(\Qcal^{\geq \mu})$ for every $\mu \in \Q$. 
\smallskip

\item\label{integer slopes assumption for E, F and Q, after slope reduction} all slopes of $\Ecal, \genslopered{\Fcal}$ and $\Qcal$ are integers. 
\smallskip

\item\label{equal rank assumption for F and Q, after slope reduction} $\rk(\Qcal) = \rk(\genslopered{\Fcal})$. 

\end{enumerate}
After writing $\genslopered{\Fcal} = \Ucal \oplus \maxslopered{\Fcal}'$ by definition and also $\Qcal = \Ucal \oplus \Qcal'$ as in \eqref{max common factor decomps for F and Q in slope reduction step}, we deduce all of these properties from the invariance of the properties \ref{slope condition on F' and Q'}, \ref{integer slopes assumption for E, F' and Q'}  and \ref{equal rank assumption for F' and Q'} noted in \ref{invariance of assumptions for F' and Q' in slope reduction steps}

We move on to the statement \ref{statement of decreasing cEF(Q) in slope reduction steps}. Since $\Qcal' \neq 0$ by our assumption, Lemma \ref{existence of maximal common factor decomp} yields 
\begin{equation}\label{order of slopes in max common factor decomps for slope reduction}
\mumin(\Ucal) \geq \mumax(\Fcal') > \mumax(\Qcal') = \mumax(\maxslopered{\Fcal}') \quad\quad \text{ if } \Ucal \neq 0.
\end{equation}
Then by Corollary \ref{zero degree for completely dominating slopes} we obtain
\begin{equation}\label{zero degree formula for max common factor decomps}
\deg(\Ucal^\vee \otimes \Fcal')^\nonneg = \deg(\Ucal^\vee \otimes \maxslopered{\Fcal}')^\nonneg = 0.
\end{equation}
Note that \eqref{zero degree formula for max common factor decomps} does not require the condition $\Ucal \neq 0$ from \eqref{order of slopes in max common factor decomps for slope reduction} since it evidently holds when $\Ucal = 0$. Now we use \eqref{zero degree formula for max common factor decomps} and the decompositions in \eqref{max common factor decomps for F and Q in slope reduction step} to find
\begin{equation*}
\begin{aligned}
\deg(\Ecal^\vee \otimes \Fcal)^\nonneg &= \deg(\Ecal^\vee \otimes (\Ucal \oplus \Fcal'))^\nonneg\\
&= \deg(\Ecal^\vee \otimes \Ucal)^\nonneg + \deg(\Ecal^\vee \otimes \Fcal')^\nonneg,\\
\deg(\Ecal^\vee \otimes \genslopered{\Fcal})^\nonneg &= \deg(\Ecal^\vee \otimes (\Ucal \oplus \maxslopered{\Fcal}'))^\nonneg\\
&= \deg(\Ecal^\vee \otimes \Ucal)^\nonneg + \deg(\Ecal^\vee \oplus \maxslopered{\Fcal}')^\nonneg,\\
\deg(\Qcal^\vee \otimes \Fcal)^\nonneg &= \deg((\Ucal \oplus \Qcal')^\vee \otimes (\Ucal \oplus \Fcal'))^\nonneg\\
&=  \deg(\Ucal^\vee \otimes \Ucal)^\nonneg + \deg(\Qcal'^\vee \otimes \Ucal)^\nonneg + \deg(\Ucal^\vee \otimes \Fcal')^\nonneg + \deg(\Qcal'^\vee \otimes \Fcal')^\nonneg \\
&=  \deg(\Ucal^\vee \otimes \Ucal)^\nonneg + \deg(\Qcal'^\vee \otimes \Ucal)^\nonneg + \deg(\Qcal'^\vee \otimes \Fcal')^\nonneg, \\
\deg(\Qcal^\vee \otimes \genslopered{\Fcal})^\nonneg &= \deg((\Ucal \oplus \Qcal')^\vee \otimes (\Ucal \oplus \maxslopered{\Fcal}'))^\nonneg\\
&=  \deg(\Ucal^\vee \otimes \Ucal)^\nonneg + \deg(\Qcal'^\vee \otimes \Ucal)^\nonneg + \deg(\Ucal^\vee \otimes \maxslopered{\Fcal}')^\nonneg + \deg(\Qcal'^\vee \otimes \maxslopered{\Fcal}')^\nonneg\\
&=  \deg(\Ucal^\vee \otimes \Ucal)^\nonneg + \deg(\Qcal'^\vee \otimes \Ucal)^\nonneg +  \deg(\Qcal'^\vee \otimes \maxslopered{\Fcal}')^\nonneg.
\end{aligned}
\end{equation*}
Therefore we have
\begin{align}
c_{\Ecal, \Fcal}(\Qcal) - c_{\Ecal, \genslopered{\Fcal}}(\Qcal) &= \big(\deg(\Ecal^\vee \otimes \Fcal)^\nonneg - \deg(\Ecal^\vee \otimes \genslopered{\Fcal})^\nonneg\big) - \big(\deg(\Qcal^\vee \otimes \Fcal)^\nonneg - \deg(\Qcal^\vee \otimes \genslopered{\Fcal})^\nonneg\big) \nonumber\\
&= \big(\deg(\Ecal^\vee \otimes \Fcal')^\nonneg - \deg(\Ecal^\vee \otimes \maxslopered{\Fcal}')^\nonneg\big) - \big(\deg(\Qcal'^\vee \otimes \Fcal')^\nonneg - \deg(\Qcal'^\vee \otimes \maxslopered{\Fcal}')^\nonneg \big) \nonumber\\
&= c_{\Ecal, \Fcal'}(\Qcal') - c_{\Ecal, \maxslopered{\Fcal}'}(\Qcal'). \label{slope reduction single step cEF(Q)}
\end{align}
Hence the statement \ref{statement of decreasing cEF(Q) in slope reduction steps} now follows directly from the fact \ref{statement of decreasing cEF'(Q') in slope reduction steps}.

We now consider the final statement \ref{equality condition for cEF(Q) in slope reduction steps}. In accordance with the statement, we assume that $\Ecal$ and $\Fcal$ satisfy the properties \ref{no common slopes for E and F, slope reduction single step} and \ref{equality condition for rank inequality on E and F, slope reduction single step}. By \ref{slope reduction single step cEF(Q)}, the inequality \eqref{slope reduction single step inequality for cEF(Q)} becomes strict if and only if the inequality \eqref{slope reduction step inequality for cEF'(Q')} is strict. Therefore we can prove the statement \ref{equality condition for cEF(Q) in slope reduction steps} by verifying that $\Ecal, \Fcal'$ and $\Qcal'$ satisfy the properties \ref{no common slopes for E and F'}, \ref{equality condition for rank inequality on E and F'} and \ref{max slope condition on F' and Q'} as stated in the fact \ref{equality condition for cEF'(Q') in slope reduction steps}. The property \ref{no common slopes for E and F'} follows from the property \ref{no common slopes for E and F, slope reduction single step} since $\Fcal'$ is a direct summand of $\Fcal$ by construction. Moreover, by Lemma \ref{existence of maximal common factor decomp} the property \ref{max slope condition on F' and Q'} follows from our assumption $\Qcal' \neq 0$. Hence it remains to verify the property \ref{equality condition for rank inequality on E and F'}. Suppose that we have $\rk(\Ecal^{\leq \mu}) = \rk (\Fcal'^{\leq \mu})$ for some $\mu$. We wish to prove that $\Ecal^{\leq \mu} \simeq \Fcal'^{\leq \mu}$. Since $\Fcal'$ is a direct summand of $\Fcal$, we find
\begin{equation}\label{slope condition on F and F' in slope reduction single step}
\rk(\Fcal'^{\leq \mu}) \leq \rk(\Fcal^{\leq \mu}) \quad\quad \text{ for every } \mu \in \Q
\end{equation}
with equality if and only if $\Fcal'^{\leq \mu} = \Fcal^{\leq \mu}$. We thus obtain a series of inequalities 
\begin{equation}\label{rank inequality on E, F and F'}
\rk(\Fcal'^{\leq \mu}) \leq \rk(\Fcal^{\leq \mu}) \leq \rk(\Ecal^{\leq \mu})  \quad\quad \text{ for every } \mu \in \Q
\end{equation}
combining \eqref{slope condition on F and F' in slope reduction single step} and the property \ref{slope condition on E and F without equality condition, slope reduction single step}. Now the equality $\rk(\Ecal^{\leq \mu}) = \rk (\Fcal'^{\leq \mu})$ implies that both equalities should hold in \eqref{rank inequality on E, F and F'}. Hence the equality condition for \eqref{slope condition on F and F' in slope reduction single step} and the property \ref{equality condition for rank inequality on E and F, slope reduction single step} together yield
\[\Fcal'^{\leq \mu} = \Fcal^{\leq \mu} \simeq \Ecal^{\leq \mu}.\]
We thus complete the proof. 
\end{proof}

We are finally ready to complete Step 3. 

\begin{prop}
Proposition \ref{key inequality} holds under the additional assumptions that $\rk(\Qcal) = \rk(\Fcal)$ and that all slopes of $\Ecal, \Fcal$ and $\Qcal$ are integers. 
\end{prop}

\begin{proof}
Let $\Ecal, \Fcal$ and $\Qcal$ be vector bundles on $\adicff$ satisfying the following properties:
\begin{enumerate}[label=(\roman*)]
\item\label{slope condition on E and F, slope reduction initial state} $\rk(\Ecal^{\leq \mu}) \geq \rk(\Fcal^{\leq \mu})$ for every $\mu \in \Q$ with equality only when $\Ecal^{\leq \mu} \simeq \Fcal^{\leq \mu}$. 
\smallskip

\item\label{slope condition on E and Q, slope reduction initial state} $\rk(\Ecal^{\leq \mu}) \geq \rk(\Qcal^{\leq \mu})$ for every $\mu \in \Q$ with equality only when $\Ecal^{\leq \mu} \simeq \Qcal^{\leq \mu}$. 
\smallskip

\item\label{slope condition on F and Q, slope reduction initial state} $\rk(\Fcal^{\geq \mu}) \geq \rk(\Qcal^{\geq \mu})$ for every $\mu \in \Q$. 
\smallskip

\item\label{integer slopes assumption for E, F and Q, slope reduction initial state} all slopes of $\Ecal, \Fcal$ and $\Qcal$ are integers. 
\smallskip

\item\label{equal rank assumption for F and Q, slope reduction initial state} $\rk(\Qcal) = \rk(\Fcal)$. 
\smallskip

\item\label{no common slopes for E and F, slope reduction initial state} $\Ecal$ and $\Fcal$ have no common slopes. 
\end{enumerate}
We wish to prove that the inequality \eqref{deg inequality for surj} holds with equality if and only if $\Qcal = \Fcal$. 

Let us define a sequence $(\Fcal_n)$ of vector bundles on $\adicff$ as follows:
\begin{enumerate}[label=(\Roman*)]
\item\label{slope reduction initial state} Set $\Fcal_0 := \Fcal$. 

\item\label{slope reduction inductive state} For each $n \geq 0$, consider the decompositions
\begin{equation}\label{max common factor decomps for Fn and Q, slope reduction step}
\Fcal_n \simeq \Ucal_n \oplus \Fcal_n' \quad\quad \text{ and } \quad\quad \Qcal \simeq \Ucal_n \oplus \Qcal_n'
\end{equation}
given by Lemma \ref{existence of maximal common factor decomp}. If $\Qcal_n' = 0$, we make $\Fcal_n$ the final term of the sequence. Otherwise, we set
\[\Fcal_{n+1} := \Ucal_n \oplus \maxslopered{\Fcal}_n'\]
where $\maxslopered{\Fcal}_n'$ denotes the maximal slope reduction of $\Fcal'_n$ to $\Qcal'_n$. 
\end{enumerate}
An induction argument using Proposition \ref{slope reduction single step} yields the following facts:
\begin{enumerate}[label=(\alph*)]
\item The sequence $(\Fcal_n)$ is well-defined with the following properties:
\begin{enumerate}[label=$(\text{\roman*}_n)$]
\item\label{slope condition on E and Fn without equality condition, slope reduction process} $\rk(\Ecal^{\leq \mu}) \geq \rk(\Fcal_n^{\leq \mu})$ for every $\mu \in \Q$ 
\smallskip

\item\label{slope condition on E and Q, slope reduction process} $\rk(\Ecal^{\leq \mu}) \geq \rk(\Qcal^{\leq \mu})$ for every $\mu \in \Q$ with equality only when $\Ecal^{\leq \mu} \simeq \Qcal^{\leq \mu}$. 
\smallskip

\item\label{slope condition on Fn and Q, slope reduction process} $\rk(\Fcal_n^{\geq \mu}) \geq \rk(\Qcal^{\geq \mu})$ for every $\mu \in \Q$. 
\smallskip

\item\label{integer slopes assumption for E, Fn and Q, slope reduction process} all slopes of $\Ecal, \Fcal_n$ and $\Qcal$ are integers. 
\smallskip

\item\label{equal rank assumption for Fn and Q, slope reduction process} $\rk(\Qcal) = \rk(\Fcal_n)$. 

\end{enumerate}
\smallskip

\item\label{statement of decreasing cEF(Q) in slope reduction process} We have an inequality
\begin{equation}\label{slope reduction step inequality for cEFn(Q)}
c_{\Ecal, \Fcal_n}(\Qcal) \geq c_{\Ecal, \Fcal_{n+1}}(\Qcal).
\end{equation}
\smallskip

\item\label{equality condition for cEF(Q) in slope reduction process} The inequality \eqref{slope reduction step inequality for cEFn(Q)} is strict if $\Ecal$ and $\Fcal_n$ satisfy the following additional properties:
\smallskip
\begin{enumerate}[label = $(\text{\roman*}_n)$, start = 6]
\item\label{no common slopes for E and Fn} $\Ecal$ and $\Fcal_n$ have no common slopes. 
\smallskip

\item\label{equality condition for rank inequality on E and Fn} For every $\mu \in \Q$, an equality $\rk(\Ecal^{\leq \mu}) = \rk(\Fcal_n^{\leq \mu})$ holds only when $\Ecal^{\leq \mu} \simeq \Fcal_n^{\leq \mu}$. 
\end{enumerate}
\end{enumerate}

\begin{figure}[H]
\begin{tikzpicture}	
		\coordinate (left) at (0, 0);
		\coordinate (q1) at (2.5, 3.5);
		\coordinate (q2) at (4, 4);
		\coordinate (q3) at (5, 3.7);

		\coordinate (p0) at (0.5, 1.5);
		\coordinate (p1) at (1.5, 2.5);
		\coordinate (p2) at (3.5, 3);
		\coordinate (p3) at (4.5, 2.5);
		\coordinate (p4) at (5, 1.5);
				
		\draw[step=1cm,thick, color=red] (left) -- (p0) --  (p1);
		\draw[step=1cm,thick, color=blue] (p1) -- (q1) -- (q2) -- (q3);
		\draw[step=1cm,thick, color=green] (p1) -- (p2) -- (p3) -- (p4);

		\draw [fill] (q1) circle [radius=0.05];		
		\draw [fill] (q2) circle [radius=0.05];		
		\draw [fill] (q3) circle [radius=0.05];		
		\draw [fill] (left) circle [radius=0.05];
		
		\draw [fill] (p0) circle [radius=0.05];		
		\draw [fill] (p1) circle [radius=0.05];		
		\draw [fill] (p2) circle [radius=0.05];		
		\draw [fill] (p3) circle [radius=0.05];		
		\draw [fill] (p4) circle [radius=0.05];		


		
		\path (q3) ++(0.2, -0.4) node {$\HN(\Fcal_n)$};
		\path (p4) ++(-0.2, -0.4) node {$\HN(\Qcal)$};
		\path (left) ++(-0.3, -0.05) node {$O$};

		\path (p0) ++(-0.2, 0.3) node {\color{red}$\Ucal_n$};
		\path (q1) ++(0, 0.3) node {\color{blue}$\Fcal_n'$};
		\path (p2) ++(0, -0.4) node {\color{green}$\Qcal_n'$};

\end{tikzpicture}
\begin{tikzpicture}[scale=0.4]
        \pgfmathsetmacro{\textycoordinate}{6}
		\draw[->, line width=0.6pt] (0, \textycoordinate) -- (1.5,\textycoordinate);
		\draw (0,0) circle [radius=0.00];	
        \hspace{0.2cm}
\end{tikzpicture}
\begin{tikzpicture}
		\pgfmathsetmacro{\reducedycoordinate}{2.5/4+2.5}
	
		\coordinate (left) at (0, 0);
		\coordinate (q1) at (2.5, 3.5);
		\coordinate (q2) at (4, 4);
		\coordinate (q3) at (5, 3.7);

		\coordinate (q2') at (4, \reducedycoordinate);
		\coordinate (q3') at (5, \reducedycoordinate-0.3);

		\coordinate (p0) at (0.5, 1.5);
		\coordinate (p1) at (1.5, 2.5);
		\coordinate (p2) at (3.5, 3);
		\coordinate (p3) at (4.5, 2.5);
		\coordinate (p4) at (5, 1.5);
				
		\draw[step=1cm,thick, color=red] (left) -- (p0) --  (p1);
		\draw[step=1cm,thick,dashed, color=blue] (p1) -- (q1) -- (q2) -- (q3);
		\draw[step=1cm,thick, color=blue] (p1) -- (q2') -- (q3');
		\draw[step=1cm,thick, color=green] (p2) -- (p3) -- (p4);

		\draw [fill] (q1) circle [radius=0.05];		
		\draw [fill] (q2) circle [radius=0.05];		
		\draw [fill] (q3) circle [radius=0.05];		
		\draw [fill] (left) circle [radius=0.05];

		\draw [fill] (q2') circle [radius=0.05];		
		\draw [fill] (q3') circle [radius=0.05];	
		
		\draw [fill] (p0) circle [radius=0.05];		
		\draw [fill] (p1) circle [radius=0.05];		
		\draw [fill] (p2) circle [radius=0.05];		
		\draw [fill] (p3) circle [radius=0.05];		
		\draw [fill] (p4) circle [radius=0.05];	

		\draw[->, line width=0.6pt, color=blue] (3, 3.5) -- (3,3);	


		
		\path (q3') ++(1.1, 0) node {$\HN(\Fcal_{n+1})$};
		\path (p4) ++(-0.2, -0.4) node {$\HN(\Qcal)$};
		\path (left) ++(-0.3, -0.05) node {$O$};

		\path (p0) ++(-0.2, 0.3) node {\color{red}$\Ucal_n$};
		\path (q2') ++(0, 0.3) node {\color{blue}$\maxslopered{\Fcal}_n'$};
		\path (p3) ++(-0.3, -0.4) node {\color{green}$\Qcal'_{n+1}$};

\end{tikzpicture}
\caption{Construction of the sequence $(\Fcal_n)$}
\end{figure}

We assert that the sequence $(\Fcal_n)$ is finite. It suffices to prove
\begin{equation}\label{increasing max common factor in slope reduction}
\rk(\Ucal_n) < \rk(\Ucal_{n+1})
\end{equation}
since we have $\rk(\Ucal_n) \leq \rk(\Qcal)$ by \eqref{max common factor decomps for Fn and Q, slope reduction step}. To this end, we align the polygons $\HN(\Fcal_n)$ and $\HN(\Qcal)$ so that their left endpoints lie at the origin. The proof of Lemma \ref{existence of maximal common factor decomp} shows that $\Ucal_n$ represents the common part of $\HN(\Fcal_n)$ and $\HN(\Qcal)$. Moreover, since $\maxslopered{\Fcal}_n'$ is the maximal slope reduction of $\Fcal'_n$ to $\Qcal'_n$, the polygons $\HN(\maxslopered{\Fcal}_n')$ and $\HN(\Qcal_n')$ with their left endpoints aligned have some nontrivial common part which we represent by a nonzero vector bundle $\Tcal_n$. Let us now consider the decompositions 
\begin{equation}\label{decomps after slope reduction}
\Fcal_{n+1} = \Ucal_n \oplus \maxslopered{\Fcal}'_n \quad\quad \text{ and } \quad\quad \Qcal = \Ucal_n \oplus \Qcal'_n
\end{equation}
given by the definition of $\Fcal_{n+1}$ and \eqref{max common factor decomps for Fn and Q, slope reduction step}. The definition of $\Fcal_{n+1}'$ assumes that $\Qcal_n' \neq 0$, so Lemma \ref{existence of maximal common factor decomp} yields
\[\mumin(\Ucal_n) > \mumax(\Qcal_n') = \mumax(\maxslopered{\Fcal}_n') \quad\quad \text{ if } \Ucal_n \neq 0.\]
Hence the decompositions \eqref{decomps after slope reduction} imply that the common part of $\HN(\Fcal_{n+1})$ and $\HN(\Qcal)$ (with their left endpoints at the origin) is given by $\Ucal_n \oplus \Tcal_n$. Now we find $\Ucal_{n+1} \simeq \Ucal_n \oplus \Tcal_n$ by the proof of Lemma \ref{existence of maximal common factor decomp}, and consequently obtain the inequality \eqref{increasing max common factor in slope reduction} as $\Tcal$ is nonzero.

Let us now denote by $r$ the index of the final term in the sequence $(\Fcal_n)$. Since $\Qcal_r' = 0$ by \ref{slope reduction inductive state}, the decompositions in \eqref{max common factor decomps for Fn and Q, slope reduction step} yield
\[ \Fcal_r \simeq \Ucal_r \oplus \Fcal_r' \quad\quad \text{ and } \quad\quad \Qcal \simeq \Ucal_r.\]
Hence the property \ref{equal rank assumption for Fn and Q, slope reduction process} for $n = r$ implies 
\begin{equation}\label{slope reduction final state}
\Fcal_r \simeq \Qcal,
\end{equation} 
which in turn yields $c_{\Ecal, \Fcal_r}(\Qcal) = 0$ by Definition \ref{definition of cEF(Q)}. Now by the fact \ref{statement of decreasing cEF(Q) in slope reduction process} we find
\begin{equation}\label{chain inequalities for slope reduction process}
c_{\Ecal, \Fcal}(\Qcal) = c_{\Ecal, \Fcal_0}(\Qcal) \geq c_{\Ecal, \Fcal_1}(\Qcal) \geq \cdots \geq c_{\Ecal, \Fcal_r}(\Qcal) = 0,
\end{equation}
thereby establishing the desired inequality \eqref{deg inequality for surj}.

Our final task is to show that equality in \eqref{deg inequality for surj} holds if and only if $\Qcal = \Fcal$. Since equality in \eqref{deg inequality for surj} evidently holds if $\Qcal = \Fcal$, we only need to show that equality in \eqref{deg inequality for surj} implies $\Qcal = \Fcal$. The properties \ref{slope condition on E and F, slope reduction initial state} and \ref{no common slopes for E and F, slope reduction initial state} together imply that whenever $r \geq 1$ we have a strict inequality
\[c_{\Ecal, \Fcal_0}(\Qcal) > c_{\Ecal, \Fcal_1}(\Qcal)\]
by the fact \ref{equality condition for cEF(Q) in slope reduction process}. Hence we deduce from \eqref{chain inequalities for slope reduction process} that equality in \eqref{deg inequality for surj} holds only if $r = 0$, which implies $\Fcal = \Fcal_0 \simeq \Qcal$ by \eqref{slope reduction final state}. 
\end{proof}

We thus complete the proof of Proposition \ref{key inequality}, and therefore the proof of Theorem \ref{classification of quotient bundles}. 


\appendix

\section{Classification of quotient bundles on $\PP^1$}

\smallskip
\begin{center}by \textsc{Serin Hong and Hannah Larson}\end{center}

\subsection{Main statement}$ $

In this appendix, we establish an analogue of Theorem \ref{classification of quotient bundles} for vector bundles on the projective line $\PP^1$ over an arbitrary field $k$. For each integer $d$, we denote by $\trivbundle(d)$ the $d$-th Serre twist of the trivial line bundle on $\PP^1$. It is a classical theorem (often attributed to Grothendieck) that every vector bundle $\Vcal$ on $\PP^1$ admits a direct sum decomposition
\begin{equation}\label{P1 HN decomp} 
\Vcal \simeq \bigoplus_{i = 1}^r \trivbundle(d_i) \quad\quad \text{ with } d_i \in \Z.
\end{equation}
Now we can state the main statement of this appendix as follows:
\begin{theorem}\label{classification of quotient bundles on P1}
Let $\Ecal$ and $\Fcal$ be vector bundles on $\PP^1$ with direct sum decompositions
\begin{equation}\label{P1 HN decomps for E, F} 
\Ecal \simeq \bigoplus_{i = 1}^r \trivbundle(a_i) \quad\quad \text{ and } \quad\quad \Fcal \simeq \bigoplus_{j=1}^s \trivbundle(b_j)
\end{equation}
for some integers $a_1 \leq \cdots \leq a_r$ and $b_1 \leq \cdots \leq b_s$. 
Then $\Fcal$ arises as a quotient of $\Ecal$ if and only if for each $j = 1, \cdots, s$, we have either $b_j \geq a_{j+1}$ or $b_i = a_i$ for all $i = 1, \cdots, j$. 
\end{theorem}

This theorem is indeed an analogue of Theorem \ref{classification of quotient bundles} for vector bundles on $\PP^1$. For every vector bundle $\Vcal$ on $\PP^1$, we can use a direct sum decomposition as in \eqref{P1 HN decomp} to define its Harder-Narasimhan polygon $\HN(\Vcal)$ and  vector bundles $\Vcal^{\leq \mu}$ for every $\mu \in \Q$. Then we can state Theorem \ref{classification of quotient bundles on P1} in exact accordance with Theorem \ref{classification of quotient bundles}.

In the subsequent sections, we present two proofs of 
Theorem \ref{classification of quotient bundles on P1}. The first proof is based on some elementary linear algebra, and is largely inspired by the argument of Eisenbud-Harris \cite[Proposition 6.30]{EH_3264andallthat}. The second proof is based on dimension analysis on moduli spaces of bundle maps, and is essentially identical to our proof of Theorem \ref{classification of quotient bundles}.

\subsection{First proof: elementary linear algebra}$ $

For the rest of this appendix, we take $\Ecal$ and $\Fcal$ to be vector bundles on $\PP^1$ with direct sum decompositions as in \eqref{P1 HN decomps for E, F}. 

\begin{lemma}\label{matrix representation of quotient bundles on P1}
The existence of a surjective bundle map $\Ecal \surj \Fcal$ amounts to the existence of an $s \times r$ matrix $M$ over the polynomial ring $k[x, y]$ with the following properties:
\begin{enumerate}[label=(\roman*)]
\item\label{degree condition on entries of matrix representing a surj bundle map} The $(p, q)$-th entry of $M$ is either zero or homogeneous of degree $b_p - a_q$. 
\smallskip

\item\label{rank condition on matrix representing a surj bundle map} The $s \times s$ minors of $M$ have no common zeros. 
\end{enumerate}
\end{lemma}

\begin{proof}
For each integer $d \geq 0$, we can canonically identify $H^0(\PP^1, \trivbundle(d))$ as the space of degree $d$ homogeneous polynomials in $k[x, y]$. In addition, we have a natural identification $\text{Hom}_{\PP^1}(\trivbundle(a_q), \trivbundle(b_p)) \cong H^0(\PP^1, \trivbundle(b_p - a_q))$ for each $p = 1, \cdots, s$ and $q = 1, \cdots, r$. Therefore every bundle map $\Ecal \to \Fcal$ can be represented by an $s \times r$ matrix $M$ over $k[x, y]$ with the property \ref{degree condition on entries of matrix representing a surj bundle map}. Moreover, the map $\Ecal \to \Fcal$ is surjective if and only if $M$ has rank $s$ at all points on $\PP^1$, which amounts to having the property \ref{rank condition on matrix representing a surj bundle map}. 
\end{proof}

\begin{remark}
In the proof, when we identify $H^0(\PP^1, \trivbundle(d))$ as the space of degree $d$ homogeneous polynomials over $k$, we take the convention that the zero polynomial is homogeneous of all nonnegative degrees. 
\end{remark}

\begin{prop}\label{classification of quotient bundles on P1, necessity part by linear algebra}
If $\Fcal$ arises as a quotient of $\Ecal$, then for each $j = 1, \cdots, s$, we have either $b_j \geq a_{j+1}$ or $b_i = a_i$ for all $i = 1, \cdots, j$. 
\end{prop}

\begin{proof}
Suppose that we have $b_j < a_{j+1}$ for some $j = 1, \cdots, s$. 
We wish to show $b_i = a_i$ for $i = 1, \cdots, j$. Since $\Fcal$ arises as a quotient of $\Ecal$, we can take an $s \times r$ matrix $M$ over $k[x, y]$ as in Lemma \ref{matrix representation of quotient bundles on P1}. Then for each $p = 1, \cdots, j$ and $q = j+1, \cdots, r$, we find $b_p \leq b_j < a_{j+1} \leq a_q$ and consequently deduce that the $(p, q)$-th entry of $M$ is zero by the property \ref{degree condition on entries of matrix representing a surj bundle map} in Lemma \ref{matrix representation of quotient bundles on P1}. In other words, $M$ has a block decomposition
\[ M = \begin{pmatrix} A & 0 \\ B & C \end{pmatrix}\]
where $A$ is a $j \times j$ matrix over $k[x, y]$. Hence any nonzero $s \times s$ minor of $M$ should be divisible by the determinant of $A$. Now the property \ref{rank condition on matrix representing a surj bundle map} in Lemma \ref{matrix representation of quotient bundles on P1} implies that the determinant of $A$ should be a constant nonzero polynomial; otherwise it has a nontrivial zero at which all $s \times s$ minors of $M$ vanish. 

We then observe that for each $i = 1, \cdots, j$ we must have $b_i \geq a_i$; otherwise, we find $b_p \leq b_i < a_i \leq a_q$ for each $p = 1, \cdots, i$ and $q = i+1, \cdots, r$, and consequently deduce by the property \ref{degree condition on entries of matrix representing a surj bundle map} in Lemma \ref{matrix representation of quotient bundles on P1} that $A$ has a block decomposition 
\[ A = \begin{pmatrix} A_{11} & 0 \\ A_{21} &  A_{22} \end{pmatrix}\]
with $A_{11}$ being an $i \times (i-1)$ matrix and thus has a zero determinant. Similarly, for each $i = 1, \cdots, j$ we must have $b_i \leq a_i$; otherwise, we find $b_p \geq b_i > a_i \geq a_q$ for each $p = i, \cdots, s$ and $q = 1, \cdots, i$, and consequently deduce by the property \ref{degree condition on entries of matrix representing a surj bundle map} in Lemma \ref{matrix representation of quotient bundles on P1} that $A$ has a block decomposition 
\[ A = \begin{pmatrix} A_{11} & A_{12} \\ A_{21} &  A_{22} \end{pmatrix}\]
where $A_{21}$ is an $(s-i+1) \times i$ matrix with entries of positive degree, and thus has a determinant which is either zero or of positive degree. Therefore we conclude that $a_i$ and $b_i$ are equal for each $i = 1, \cdots, j$ as desired. 
\end{proof}

\begin{prop}\label{classification of quotient bundles on P1, sufficienty part by linear algebra}
Assume that for each $j = 1, \cdots, s$, we have either $b_j \geq a_{j+1}$ or $b_i = a_i$ for all $i = 1, \cdots, j$. Then $\Fcal$ arises as a quotient of $\Ecal$. 
\end{prop}

\begin{proof}
Let $l$ be the largest integer with the property $a_l = b_l$. Take $M$ to be the $s \times r$ matrix whose 
nonzero entries are given as follows:
\[ M = 
\left(\begin{array}{@{}c|c@{}}
\begin{matrix} 1 & & \\ & \ddots & \\ & & 1 \end{matrix}  & \\ \hline  & \begin{matrix} x^{b_{l+1} - a_{l+1}} & y^{b_{l+1} - a_{l+2}} &  &  & \\
 & \ddots & \ddots & & \\
 & & x^{b_s - a_s} & y^{b_s - a_{s+1}} & \quad\quad  \end{matrix} 
\end{array}\right)
\]
It suffices to show that $M$ satisfies the properties \ref{degree condition on entries of matrix representing a surj bundle map} and \ref{rank condition on matrix representing a surj bundle map} in Lemma \ref{matrix representation of quotient bundles on P1}. The property \ref{degree condition on entries of matrix representing a surj bundle map} is evident by construction. The property \ref{rank condition on matrix representing a surj bundle map} follows from the fact that the $s \times s$ minor with columns $1, \cdots, s$ is a power of $x$ while the $s \times s$ minor with columns $1, \cdots, l, l+2, \cdots, s+1$ is a power of $y$. 
\end{proof}

We now deduce Theorem \ref{classification of quotient bundles on P1} from Proposition \ref{classification of quotient bundles on P1, necessity part by linear algebra} and Proposition \ref{classification of quotient bundles on P1, sufficienty part by linear algebra}.

\subsection{Second proof: dimension analysis on moduli spaces of bundle maps}$ $

Let us denote by $\Sch_{/ k}$ the category of $k$-schemes. For every vector bundle $\Vcal$ on $\PP^1$, we will write $\rkdegsum(\Vcal):= \rk(\Vcal) + \deg(\Vcal)$. In addition, for arbitrary vector bundles $\Vcal$ and $\Wcal$ on $\PP^1$, we can define the moduli functors $\Hom_{\PP^1}(\Vcal, \Wcal)$, $\Surj_{\PP^1}(\Vcal, \Wcal)$, and $\Inj_{\PP^1}(\Vcal, \Wcal)$ on $\Sch_{/ k}$ as in Definition \ref{def of moduli functors of bundle maps}. Then we have the following analogue of Proposition \ref{moduli of bundle maps dimension formulas}:

\begin{prop}\label{dimension formula for bundle map spaces P1}
Let $\Vcal$ and $\Wcal$ be vector bundles on $\PP^1$. 
\begin{enumerate}[label=(\arabic*)]
\item $\Hom_{\PP^1}(\Vcal, \Wcal)$ is represented by the affine scheme $\mathbb{A}_k^{\rkdegsum(\Vcal^\vee \otimes \Wcal)^\nonneg}$. 
\smallskip

\item $\Surj_{\PP^1}(\Vcal, \Wcal)$ and $\Inj_{\PP^1}(\Vcal, \Wcal)$ are represented by an open subscheme of $\Hom_{\PP^1}(\Vcal, \Wcal)$, and thus are either empty or of dimension $\rkdegsum(\Vcal^\vee \otimes \Wcal)^\nonneg$. 
\end{enumerate}
\end{prop}

\begin{proof}
The functor $\Hom_{\PP^1}(\Vcal, \Wcal)$ is indeed represented by $\Spec( \Sym_k H^0(\PP^1, \Vcal^\vee \otimes \Wcal)^\vee)$, whose Krull dimension is given by $\dim_k H^0(\PP^1, \Vcal^\vee \otimes \Wcal) = \rkdegsum(\Vcal^\vee \otimes \Wcal)^\nonneg$. Now we can argue exactly as in \cite[\S3.3]{Arizona_extvb} to establish the second statement. 
\end{proof}

Moreover, we have an analogue of Proposition \ref{dimension inequality for surj maps} as follows:

\begin{prop}\label{dimension inequality for surj maps over P1}
$\Fcal$ arises as a quotient of $\Ecal$ if the following conditions are satisfied:
\begin{enumerate}[label=(\roman*)]
\item\label{existence of nonzero bundle map from E to F over P1} There exists a nonzero bundle map $\Ecal \to \Fcal$. 

\item\label{positive codim for Hom minus surj for P1} For any $\Qcal \subsetneq \Fcal$ which also occurs as a quotient of $\Ecal$ we have an inequality
\[ \rkdegsum(\Ecal^\vee \otimes \Qcal)^\nonneg + \rkdegsum(\Qcal^\vee \otimes \Fcal)^\nonneg < \rkdegsum(\Ecal^\vee \otimes \Fcal)^\nonneg + \rkdegsum(\Qcal^\vee \otimes \Qcal)^\nonneg.\]
\end{enumerate}
\end{prop}

\begin{proof}
Let $S$ be the set of isomorphism classes of subsheaves $\Qcal \subseteq \Fcal$ which also occur as a quotient of $\Ecal$. 
We assert that $S$ is finite. 
The Harder-Narasimhan theory for vector bundles on $\PP^1$ is almost identical to the Harder-Narasimhan theory for vector bundles on $\adicff$, as we only need an additional requirement that all Harder-Narasimhan slopes are integers. 
In particular, for vector bundles on $\PP^1$ we can define the notion of slopewise dominance and
obtain analogues of Propositions \ref{quotient bundles necessary condition} and \ref{subbundles necessary condition}.
It follows that the slopes of every $\Qcal \in S$ are bounded by $\mumin(\Ecal)$ and $\mumax(\Fcal)$. Since vector bundles on $\PP^1$ have integer slopes in their HN polygons, we deduce that $S$ is finite. 

Let us now assume for contradiction that $\Fcal$ does not arise as a quotient of $\Ecal$. Then $S$ does not contain the isomorphism class of $\Fcal$. 
For each $\Qcal \in S$, we define $\Hom_{\PP^1}(\Ecal, \Fcal)_{\Qcal}$ to be the image of the natural map
\[\Surj_{\PP^1}(\Ecal,\Qcal) \times_{\Spec(k)} \Inj_{\PP^1}(\Qcal,\Fcal) \to \Hom_{\PP^1}(\Ecal,\Fcal)\]
induced by composition of bundle maps, and write $\overline{\Hom_{\PP^1}(\Ecal, \Fcal)_{\Qcal}}$ for its closure in $\Hom_{\PP^1}(\Ecal,\Fcal)$. Since $\Hom_{\PP^1}(\Ecal,\Fcal)$ is represented by $\mathbb{A}_k^{\rkdegsum(\Vcal^\vee \otimes \Wcal)^\nonneg}$ as noted in Proposition \ref{dimension formula for bundle map spaces P1}, for each locally closed subscheme $\dim \Hom_{\PP^1}(\Ecal, \Fcal)_{\Qcal}$ with $\Qcal \in S$ we find
\[\dim \Hom_{\PP^1}(\Ecal, \Fcal)_{\Qcal}  = \dim \overline{\Hom_{\PP^1}(\Ecal, \Fcal)_{\Qcal}}.\]
By construction, $\Hom_{\PP^1}(\Ecal,\Fcal)$ is covered by the subschemes $\overline{\Hom_{\PP^1}(\Ecal, \Fcal)_{\Qcal}}$ with $\Qcal \in S$. As the set $S$ is finite, we find
\begin{equation}\label{dimension Hom(E, F) as max dimension of strata}
\dim \Hom_{\PP^1}(\Ecal, \Fcal) = \sup_{\Qcal \in S} \dim \overline{\Hom_{\PP^1}(\Ecal, \Fcal)_\Qcal} = \sup_{\Qcal \in S} \dim \Hom_{\PP^1}(\Ecal, \Fcal)_\Qcal.
\end{equation}
In addition, we can argue exactly as in \cite[Lemma 3.3.10]{Arizona_extvb} to show that $\Hom_{\PP^1}(\Ecal, \Fcal)_{\Qcal}$ for each $\Qcal \in S$ is either empty or satisfies
\[\dim \Hom_{\PP^1}(\Ecal, \Fcal)_{\Qcal}  = \rkdegsum(\Ecal^\vee \otimes \Qcal)^\nonneg + \rkdegsum(\Qcal^\vee \otimes \Fcal)^\nonneg - \rkdegsum(\Qcal^\vee \otimes \Qcal)^\nonneg.\]
Now the assumption \ref{positive codim for Hom minus surj for P1} and Proposition \ref{dimension formula for bundle map spaces P1} together imply
\[\dim \Hom_{\PP^1}(\Ecal, \Fcal)_{\Qcal} < \dim \Hom_{\PP^1}(\Ecal,\Fcal) \quad \text{ for every } \Qcal \in S,\]
thereby yielding a contradiction by \eqref{dimension Hom(E, F) as max dimension of strata} as desired. 
\end{proof}

\begin{remark}
The proof of Proposition \ref{dimension inequality for surj maps over P1} is slightly different from the original proof of \cite[Theorem 3.3.11]{Arizona_extvb}. Here we established \eqref{dimension Hom(E, F) as max dimension of strata} by using the fact that $\Hom_{\PP^1}(\Ecal,\Fcal)$ is represented by an affine algebraic space. For \cite[Theorem 3.3.11]{Arizona_extvb}, we get an analogous identity from the fact that the topological space $|\Hom(\Ecal, \Fcal)_\Qcal|$ is stable under generalization and specialization inside $|\Hom(\Ecal, \Fcal)|$. 
\end{remark}

Let us also record some basic properties of the function $\rkdegsum$. 

\begin{lemma}\label{algebraic properties of hom dimension}
Let $\Vcal$ and $\Wcal$ be vector bundles on $\PP^1$. 
\begin{enumerate}[label=(\arabic*)]
\item\label{additivity of hom dimension} We have $\rkdegsum(\Vcal \oplus \Wcal) = \rkdegsum(\Vcal) + \rkdegsum(\Wcal)$. 
\smallskip

\item\label{hom dimension zero condition} We have $\rkdegsum(\Vcal^\vee \otimes \Wcal)^\nonneg = 0$ if and only if $\mumin(\Vcal)$ is greater than $\mumax(\Wcal)$. 
\smallskip

\item\label{twist invariance of hom dimension} For any $d \in \Z$, we have $\rkdegsum(\Vcal(d)^\vee \otimes \Wcal(d))^\nonneg = \rkdegsum(\Vcal^\vee \otimes \Wcal)^\nonneg$.
\smallskip

\item\label{slopewise dominance and hom dimension} If $\Vcal$ slopewise dominates $\Wcal$, then we have $\rkdegsum(\Vcal)^\nonneg \geq \rkdegsum(\Wcal)^\nonneg$. 
\smallskip

\item\label{hom dimension inequality for max slope reduction} For any $d \in \Z$, we have $\rkdegsum(\Vcal^\vee \otimes \Wcal^{>d})^\nonneg \geq \rkdegsum(\Vcal^\vee \otimes \trivbundle(d)^{\rk(\Wcal^{>d})})^\nonneg$. 
\end{enumerate}
\end{lemma}

\begin{proof}
The statement \ref{additivity of hom dimension} is evident by the additivity of rank and degree for vector bundles. The statements \ref{hom dimension zero condition}, \ref{twist invariance of hom dimension} and \ref{slopewise dominance and hom dimension} are straightforward to check by arguing as in their analogues, namely Corollary \ref{zero degree for completely dominating slopes}, Lemma \ref{degree after shear} and Lemma \ref{nonnegative degree for slopewise dominant pairs}. The statement \ref{hom dimension inequality for max slope reduction} is an analogue of the inequality \eqref{max slope reduction difference of e,f terms lower bound}, and can be verified by arguing as in the proof of Proposition \ref{decreasing c_EF(Q) after max reduction}; in fact, we immediately find 
\[\deg(\Vcal^\vee \otimes \Wcal^{>d})^\nonneg \geq \deg(\Vcal^\vee \otimes \trivbundle(d)^{\rk(\Wcal^{>d})})^\nonneg\] 
as the inequality \eqref{max slope reduction difference of e,f terms lower bound} holds verbatim for vector bundles on $\PP^1$, and also find 
\[\rk(\Vcal^\vee \otimes \Wcal^{>d})^\nonneg \geq \rk(\Vcal^\vee \otimes \trivbundle(d)^{\rk(\Wcal^{>d})})^\nonneg\] 
by a similar argument that considers the $x$-coordinates of HN vectors. 
\end{proof}

\begin{remark}
The function $\rkdegsum$ does not admit an analogue of Lemma \ref{degree after stretch}. However, this won't be a problem for us. Indeed, we used Lemma \ref{degree after stretch} only once in the proof of Theorem \ref{classification of quotient bundles} for reduction to the case of integer slopes. For Theorem \ref{classification of quotient bundles on P1}, we don't need such a reduction step as vector bundles on $\PP^1$ only have integer slopes. 
\end{remark}

Now we present another proof of Theorem \ref{classification of quotient bundles on P1}. 
\begin{prop}\label{classification of quotient bundles on P1 from key inequality}
Theorem \ref{classification of quotient bundles on P1} follows from Proposition \ref{dimension inequality for surj maps over P1} and Lemma \ref{algebraic properties of hom dimension}. 
\end{prop}

\begin{proof}
We can reformulate Theorem \ref{classification of quotient bundles on P1} as an analogue of Theorem \ref{classification of quotient bundles} using vector bundles $\Ecal^{\leq \mu}$ and $\Fcal^{\leq \mu}$ for $\mu \in \Q$. Then the necessity part of Theorem \ref{classification of quotient bundles on P1} becomes an analogue of Proposition \ref{quotient bundles necessary condition}, which is a formal consequence of the
Harder-Narasimhan theory for vector bundles on $\PP^1$ 
as noted in the proof of Proposition \ref{dimension inequality for surj maps over P1}. 
It remains to show that the sufficiency part of Theorem \ref{classification of quotient bundles on P1} follows from Proposition \ref{dimension inequality for surj maps over P1} and Lemma \ref{algebraic properties of hom dimension}. Assume that for each $j = 1, \cdots, s$, we have either $b_j \geq a_{j+1}$ or $b_i = a_i$ for all $i = 1, \cdots, j$. As noted already, our assumption precisely means that we have $\rk(\Ecal^{\leq \mu}) \geq \rk(\Fcal^{\leq \mu})$ for each $\mu \in \Q$ with equality if and only if $\Ecal^{\leq \mu}$ and $\Fcal^{\leq \mu}$ are isomorphic. We wish to show that $\Ecal$ and $\Fcal$ satisfy the conditions of Proposition \ref{dimension inequality for surj maps over P1}. This is essentially an analogue of Proposition \ref{key inequality}, after arguing as in Lemma \ref{reduction on common slopes} to add an assumption that $\Ecal$ and $\Fcal$ have no common slopes. Our proof of Proposition \ref{key inequality} relies only on the Harder-Narasimhan theory for vector bundles on $\adicff$ and some basic properties of the degree function such as Corollary \ref{zero degree for completely dominating slopes}, Lemma \ref{degree after shear}, Lemma \ref{nonnegative degree for slopewise dominant pairs} and some inequalities in the proof of Proposition \ref{decreasing c_EF(Q) after max reduction}. In the context of Proposition \ref{dimension inequality for surj maps over P1}, the function $\rkdegsum$ takes the role of the degree function in Proposition \ref{key inequality} and has analogous properties as summarized in Lemma \ref{algebraic properties of hom dimension}. Therefore we can verify the conditions of Proposition \ref{dimension inequality for surj maps over P1} for $\Ecal$ and $\Fcal$ by only using the Harder-Narasimhan theory for vector bundles on $\PP^1$ and Lemma \ref{algebraic properties of hom dimension}. 
\end{proof}

\begin{remark}
Our argument in this subsection suggests that Theorem \ref{classification of quotient bundles} (and Theorem \ref{classification of quotient bundles on P1}) should extend to any Harder-Narasimhan categories where certain moduli spaces of morphisms exist with locally spectral topological spaces that admit a nice dimension formula as in Proposition \ref{moduli of bundle maps dimension formulas} or Proposition \ref{dimension formula for bundle map spaces P1} with some nice algebraic properties as in Lemma \ref{algebraic properties of hom dimension}. 
\end{remark}

\bibliographystyle{amsalpha}

\bibliography{Bibliography}

\newcommand{\etalchar}[1]{$^{#1}$}
\providecommand{\bysame}{\leavevmode\hbox to3em{\hrulefill}\thinspace}
\providecommand{\MR}{\relax\ifhmode\unskip\space\fi MR }
\providecommand{\MRhref}[2]{%
  \href{http://www.ams.org/mathscinet-getitem?mr=#1}{#2}
}
\providecommand{\href}[2]{#2}
\begin{thebibliography}{BFH{\etalchar{+}}17}

\bibitem[BFH{\etalchar{+}}17]{Arizona_extvb}
Christopher Birkbeck, Tony Feng, David Hansen, Serin Hong, Qirui Li, Anthony
  Wang, and Lynelle Ye, \emph{Extensions of vector bundles on the
  {F}argues-{F}ontaine curve}, J. Inst. Math. Jussieu, to appear.

\bibitem[CFS17]{CFS_admlocus}
Miaofen Chen, Laurent Fargues, and Xu~Shen, \emph{On the structure of some
  $p$-adic period domains}, arXiv:1710.06935.

\bibitem[Che20]{Chen_FRconjnonbasic}
Miaofen Chen, \emph{Fargues-{R}apoport conjecture for $p$-adic period domains
  in the non-basic case}, arXiv:2004.01425.

\bibitem[Col02]{Colmez_BCspace}
Pierre Colmez, \emph{Espaces de banach de dimension finie}, J. Inst. Math.
  Jussieu \textbf{1} (2002), no.~3, 331--439.

\bibitem[CS17]{CS_annals}
Ana Caraiani and Peter Scholze, \emph{{O}n the generic part of the cohomology
  of compact unitary {S}himura varieties}, Annals of Math. \textbf{186} (2017),
  no.~3, 649--766.

\bibitem[EH16]{EH_3264andallthat}
David Eisenbud and Joe Harris, \emph{3264 and all that: A second course in
  algebraic geometry}, Cambridge University Press, 2016.

\bibitem[Far16]{Fargues_geomLL}
Laurent Fargues, \emph{Geometrization of the local langlands correspondence: an
  overview}, arXiv:1602.00999.

\bibitem[FF14]{FF_curvesurvey}
Laurent Fargues and Jean-Marc Fontaine, \emph{Vector bundles on curves and
  {$p$}-adic {H}odge theory}, Automorphic forms and {G}alois representations.
  {V}ol. 2, London Math. Soc. Lecture Note Ser., vol. 415, Cambridge Univ.
  Press, Cambridge, 2014, pp.~17--104.

\bibitem[FF18]{FF_curve}
\bysame, \emph{Courbes et fibr\'es vectoriels en th\'eorie de {H}odge
  p-adique}, Ast\'erisque \textbf{406} (2018).

\bibitem[FS21]{FS_geomLL}
Laurent Fargues and Peter Scholze, \emph{Geometrization of the local
  {L}anglands correspondence}, arXiv:2102.13459.

\bibitem[Han17]{Hansen_degenvb}
David Hansen, \emph{Degenerating vector bundles in $p$-adic {H}odge theory}, J.
  Inst. Math. Jussieu, to appear.

\bibitem[Hon20]{Hong_extvb}
Serin Hong, \emph{On certain extensions of vector bundles in {$p$}-adic
  geometry}, arXiv:2002.06706.

\bibitem[Hon21]{Hong_subvb}
\bysame, \emph{Classification of subbundles on the {F}argues-{F}ontaine curve},
  Algebra \& Number Theory \textbf{15} (2021), no.~5, 1127--1156.

\bibitem[HP04]{HP_equalcharFFcurve}
Urs Hartl and Richard Pink, \emph{Vector bundles with a {F}robenius structure
  on the punctured unit disc}, Comp. Math. \textbf{140} (2004), 689--716.

\bibitem[Ked05]{Kedlaya_slopefiltrations_revisited}
Kiran~S. Kedlaya, \emph{Slope filtrations revisited}, Doc. Math. \textbf{10}
  (2005), 447--525.

\bibitem[Ked16]{Kedlaya_noeth}
\bysame, \emph{Noetherian properties of {F}argues-{F}ontaine curves}, Int.
  Math. Res. Not. IMRN (2016), no.~8, 2544--2567.

\bibitem[Ked17]{Kedlaya_arizona}
\bysame, \emph{Sheaves, stacks, and shtukas}, Lecture notes for 2017 Arizona
  Winter School.

\bibitem[KL15]{KL15}
Kiran~S. Kedlaya and Ruochuan Liu, \emph{Relative {$p$}-adic {H}odge theory:
  {F}oundations}, Ast\'erisque \textbf{371} (2015), 239.

\bibitem[LB18]{LeBras_BCspaces}
Arthur-C\'esar Le~Bras, \emph{Espaces de {B}anach-{C}olmez et faisceaux
  coh\'erents sur la courbe de {F}argues-{F}ontaine}, Duke Math. J.
  \textbf{167} (2018), no.~18, 3455--3532.

\bibitem[NV21]{NV_HNstrata}
Kieu~Hieu Nguyen and Eva Viehmann, \emph{A {H}arder-{N}arasimhan stratification
  of the ${B}^+_\text{dR}$-{G}rassmannian}, arXiv:2111.01764.

\bibitem[Sch18]{Scholze_diamonds}
Peter Scholze, \emph{{\'E}tale cohomology of diamonds}, arXiv:1709.07343.

\bibitem[She19]{Shen_HNstrata}
Xu~Shen, \emph{Harder-{N}arasimhan strata and $p$-adic period domains},
  arXiv:1909.02230.

\bibitem[Vie21]{Viehmann_weakadmlocNewton}
Eva Viehmann, \emph{On {N}ewton strata in the
  ${B}^+_\text{dR}$-{G}rassmannian}, arXiv:2101.07510.

\end{thebibliography}
	
\end{document}